\documentclass[amsart,11pt,english,letter paper,two side,draft]{article}
\usepackage{
amsthm,
amsmath,
amssymb,
euscript,
mathrsfs,
MnSymbol,
verbatim,
enumerate,
multicol,
color,
babel, geometry,
 tikz,tikz-cd,
 tikz-3dplot, 
 tkz-graph,
 array,
hepth, 
enumitem
}
 \usepackage{hyperref}
 \usepackage[utf8]{inputenc}
\usepackage{thm-restate}
\usepackage{thmtools}
\usepackage[T1]{fontenc}
\usepackage{datetime}
\numberwithin{equation}{section}
\numberwithin{figure}{section}
\usetikzlibrary{arrows,positioning,decorations.pathmorphing,   decorations.markings, matrix, patterns}
\theoremstyle{plain}
\newtheorem*{thm*}{Theorem}
\theoremstyle{plain}
\newtheorem{thm}{Theorem}[section]

\newtheorem{lem}[thm]{Lemma}

\newtheorem{conj}[thm]{Conjecture}
\newtheorem*{cor}{Corollary}
\theoremstyle{definition}
\newtheorem{defn}[thm]{Definition}
\newtheorem*{defn*}{Definition}
	\newtheorem{exmp}[thm]{Example}

\newtheorem{rem}[thm]{Remark}

\usepackage[normalem]{ulem}
\tikzset{
  big arrow/.style={
    decoration={markings,mark=at position 1 with {\arrow[scale=1.5,#1]{>}}},
    postaction={decorate},
    shorten >=0.4pt},
  big arrow/.default=black}

\numberwithin{equation}{section}

\begin{document}

\title{Euler characteristics of crepant resolutions of Weierstrass  models  }
\authors{{ Mboyo Esole$^{\spadesuit}$, Patrick Jefferson$^\diamondsuit$, Monica Jinwoo Kang$^\diamondsuit$}\\
\vspace{1cm}
$^{\spadesuit}$ Department of Mathematics, Northeastern University\\
 Boston, MA 02115, USA\\
\vspace{.3cm}
$^\diamondsuit$ Department of Physics, Jefferson Physical Laboratory\\
 Harvard University, Cambridge, MA 02138, U.S.A.
}

\date{}
\abstract{

Based on an identity of Jacobi, we prove a simple formula that computes the pushforward of analytic functions of the exceptional divisor of a blowup of a projective variety along a smooth complete intersection with normal crossing. We  use  this pushforward formula to derive generating functions for Euler characteristics of crepant resolutions of singular Weierstrass models given by Tate's algorithm.  Since  the Euler characteristic depends only on the sequence of blowups and not on the Kodaira fiber itself, several distinct Tate models have the same Euler characteristic. In the case of elliptic Calabi-Yau threefolds,  using the Shioda--Tate--Wazir theorem, we also compute the Hodge numbers. For elliptically fibered Calabi-Yau fourfolds, our results also prove a conjecture of Blumenhagen, Grimm, Jurke, and Weigand based on F-theory/heterotic string duality.
}

\maketitle
\tableofcontents

\newpage
\section{Introduction}

The study of crepant resolutions of Weierstrass models, their fibral structure, and their flop transitions is an area of common interest to algebraic geometers, number theorists, and string theorists \cite{EY,ESY1,ESY2,ES, Hayashi:2014kca,FH2,Szydlo.Thesis}. The theory of elliptic surfaces has its beginnings in the 1960s, and was advanced largely by the contributions of mathematicians such as Kodaira \cite{Kodaira.IIandIII}; N\'eron \cite{Neron};  Mumford and Suominen \cite{MumfordSuominen}, Deligne \cite{Formulaire}, and Tate \cite{Tate}. Miranda studied the desingularization of elliptic threefolds and the phenomenon of collisions of singularities in \cite{Miranda.smooth}, and Szydlo subsequently generalized Miranda's work to elliptic $n$-folds \cite{Szydlo.Thesis}; the Picard number (i.e., the rank of the N\'eron-Severi group) of an elliptic fibration can be obtained using the Shioda--Tate--Wazir theorem \cite{Wazir};
the study of elliptic fibrations having the same Jacobian was developed by Dolgachev and Gross \cite{MR1242006}; and Nakayama studied local and global  properties of Weierstrass models over bases of arbitrary dimension in \cite{Nakayama.Local, Nakayama.Global}.
 Furthermore, more recent developments have been inspired by string theory (in particular, M-theory and F-theory) constructions that ascribe an interesting physical meaning to various topological and geometric properties of elliptically-fibered Calabi-Yau varieties \cite{Vafa:1996xn,Morrison:1996na,Morrison:1996pp,Bershadsky:1996nh,IMS,Denef:2008wq}. 

 A Weierstrass model provides a convenient  framework for computing the discriminant, the $j$-invariant, and  the Mordell--Weil group of an elliptic fibration. Weierstrass models are also the setting in which Tate's algorithm is defined \cite{Tate}. 
 Any elliptic fibration over a smooth base is birational to a (potentially singular) Weierstrass model \cite{Formulaire}. 
Since a Weierstrass model is a hypersurface, it is Gorenstein \cite[Corollary 21.19]{MR1322960}, and hence its canonical class is well-defined as a Cartier divisor.

In practice, it is often necessary to regularize the singularities of Weierstrass models when computing, for example, their topological invariants.
 Among the possible regularizations of a singular variety, crepant resolutions are particularly desirable as, by definition, they preserve the canonical class  and the smooth locus of the variety. In a sense, crepant resolutions  modify the variety as mildly as possible  while  regularizing its singularities. Surfaces with canonical singularities always have a  crepant resolution, which is unique up to isomorphism. However, for varieties of  dimension three or higher, crepant resolutions do not necessarily exist, and when they do, they may not be unique. Distinct crepant resolutions of the same Weierstrass model are connected by a network of flops. 
\begin{exmp}
  The  quadric cone over a conic surface $V(x_1 x_2 -x_3 x_4)\subset \mathbb{C}^4$ has two crepant resolutions related by an Atiyah flop.   By contrast, the quadric cone $V(x_1^2+x_2^2+x_3^2+x_4^2 + x_5^2) \subset \mathbb{C}^5$ does not have a crepant resolution since it has $\mathbb{Q}$-factorial terminal singularities. 
The binomial variety $V(x_1 x_2 - u_1 u_2 u_3)\subset \mathbb{C}^5$ has six crepant resolutions whose network of flops forms a hexagon \cite{EY}. For additional examples of flops involving Weierstrass models, see \cite{ESY1,ESY2,G2,MP}.
\end{exmp}
There is an important subset of singular Weierstrass models that have crepant resolutions and play a central role in string geometry, as they are instrumental in the geometric engineering of gauge theories in F-theory and M-theory. 
  We refer to them as $G$-models, they are defined in  \S \ref{Sec:G.Models} and are typically obtained  by the Weierstrass models that appear as outputs of  Tate's algorithm \cite{Tate,Bershadsky:1996nh,Katz:2011qp}. 
 The networks of crepant resolutions of these Weierstrass models are conjectured to be isomorphic to the incidence graph of the chambers of a hyperplane arrangement \cite{IMS,Hayashi:2014kca,EJJN1,EJJN2}. 

The number of distinct resolutions associated to a  $G$-model  can be  rather large \cite{EJJN1,EJJN2,Hayashi:2014kca}. It is interesting to study topological invariants that do not depend on the choice of a crepant resolution. An example of such a topological invariant is the Euler characteristic---using $p$-adic integration and  Weil conjecture, Batyrev proved that the Betti numbers of smooth varieties connected by a crepant birational map are the same \cite{Batyrev.Betti}, 
and it therefore follows that the Euler characteristics of any two crepant resolutions are the same.

The purpose of this paper is to compute the Euler  characteristics of  $G$-models obtained by crepant resolutions of Weierstrass models, where $G$ is a simple group.  Following \cite{AE1,AE2},  we allow the base to be of arbitrary dimension and we do not impose the Calabi-Yau condition.
 We work relative to a base that we leave arbitrary.   In this sense, our paper is a direct generalization of the work of Fullwood and van Hoeij  on stringy invariants of Weierstrass models \cite{FH2}.

The Euler characteristic of an elliptic fibration plays a central role in many physical problems such as the computation of gravitational anomalies of six dimensional supergravity theories \cite{GM1,Park:2011ji}  and the cancellation of tadpoles in four dimensional theories \cite{Sethi:1996es,AE1,AE2,BGJW,CDE,EFY,EKY}. 
Unfortunately, the Euler characteristics of crepant resolutions of Weierstrass models are generally not known, although they have been computed in some special cases for Calabi-Yau threefolds and fourfolds  \cite{Andreas:1999ng,Andreas:2009uf,Marsano,FH2}. 
 For instance, the Euler characteristics of  $G$-models for Calabi-Yau threefolds were studied in \cite{GM1}, and there are conjectures for the  Euler  characteristics  of $G$-models for  Calabi-Yau fourfolds based on heterotic string theory/F-theory duality \cite{BGJW}.

  As a byproduct of our results, we prove a conjecture by Blumenhagen, Grimm, Jurke, and Weigand  \cite{BGJW} on the Euler  characteristics  of Calabi-Yau fourfolds which are $G$-models for $G=$  SU$(2)$, SU($3$), SU($4$), 
  SU($5$), E$_6$, E$_7$ or E$_8$. These groups correspond to the  {\em exceptional series} E$_k$ defined on page \pageref{fn:exceptional}  with the exception of D$_5$.
 In \cite{BGJW}, the authors  conjecture the value of  the Euler characteristic using a method inspired by heterotic string theory/F-theory duality. 
 The results of our computation match their prediction precisely,  except for the  limiting case of the group $E_8$.
We also retrieve  known results for the case of $G$-models that are Calabi-Yau threefolds \cite{GM1}, while removing most of the assumptions of \cite{GM1}.

A crucial ingredient of our results is Theorem \ref{Thm:Push}, which is a  pushforward formula for any analytic function of the class of the exceptional divisor of a blowup of a nonsingular variety along a smooth complete intersection of hypersurfaces meeting transversally. 
Theorem   \ref{Thm:Push} is  a generalization to arbitrary analytic functions  of a result of Fullwood and van Hoeij \cite[Lemma 2.2 ]{FH2}, which relies on a theorem of Aluffi \cite{Aluffi_CBU} simplifying the classic formula of Porteous on Chern classes of the tangent bundle of a  blowup \cite{Porteous}. 
Theorem \ref{Thm:Push} profoundly simplifies the algebraic manipulations necessary to compute pushforwards, and  therefore has a large range of  applications  independently of the specific applications discussed in this paper.

For the reader's convenience, we provide tables specializing our results to the cases of 
elliptic threefolds and fourfolds, and further to the cases of Calabi-Yau threefolds and fourfolds, including an explicit computation of the Hodge numbers in the Calabi-Yau threefold case. 
We emphasize that our results are insensitive to the particular choice of  a crepant resolution due to Batyrev's theorem on the Betti numbers of crepant birational equivalent varieties \cite{Batyrev.Betti} and Kontsevich's theorem on the Hodge numbers of birational equivalent Calabi-Yau varieties \cite{Kontsevich.Orsay}.       

\subsection{Conventions} \label{Sec:Conv}
Throughout this paper, we work over the field of complex numbers. 
A variety is a reduced and irreducible algebraic scheme. 
 We denote the vanishing locus of the sections $f_1, \ldots, f_n$ by $V(f_1, \ldots, f_n)$. 
  The tangent bundle of a variety $X$ is denoted by $TX$ and the  normal bundle of a subvariety $Z$ of a  variety $X$ is denoted by  $N_Z X$. 
 Let $\mathscr{V}\rightarrow B$ be a vector bundle over a variety $B$. We denote the by $\mathbb{P}(\mathscr{V})$ the projective bundle of lines in  $\mathscr{V}$.
We use Weierstrass models defined with respect to the projective bundle $\pi : X_0 = \mathbb P[\mathscr O_B \oplus \mathscr L^{\otimes 2} \oplus \mathscr L^{\otimes 3} ] \rightarrow B$  where $\mathscr{L}$ is a line bundle of $B$. 
We denote the pullback of $\mathscr{L}$ with respect to $\pi$ by $\pi^* \mathscr{L}$. 
We denote by $\mathscr{O}_{X_0} (1)$  the  canonical line bundle on $X_0$, i.e., the  dual of the tautological line bundle of $X_0$  (see \cite[Appendix B.5]{Fulton.Intersection}).
The first Chern class of  $\mathscr{O}_{X_0} (1)$ is denoted $H$ and the first Chern class of $\mathscr{L}$ is denoted $L$. 
 The Weierstrass model $\varphi : Y_0 \rightarrow B$ is defined as  the zero-scheme of a section of $\mathscr{O}_{X_0} (3) \otimes \pi^* \mathscr L^{\otimes 6}$---Weierstrass models are studied in more detail in \S   \ref{sec:Wmodel}.
The Chow group $A_*(X)$ of a nonsingular variety $X$ is the group of divisors modulo rational equivalence \cite[Chap. 1,\S 1.3]{Fulton.Intersection}.
We use $[V]$ to refer to the class of a subvariety $V$ in $A_*(X)$.
 Given a class $\alpha \in A_*(X)$, the degree of $\alpha$ is denoted $\int_X \alpha$ (or simply $\int \alpha$ if $X$ is clear from the context.) Only the zero component of $\alpha$ is relevant in computing $\int_X \alpha$---see \cite[Definition 1.4, p. 13]{Fulton.Intersection}. We use $c(X)=c(TX)\cap [X]$ to refer to the total homological Chern class of a nonsingular variety $X$, and likewise we use $c_i(TX)$ to denote the $i$th Chern class of the tangent bundle $TX$.  
  Given   two varieties $X, Y$ and a proper morphism $f: X \rightarrow Y$, the  proper pushforward associated to 
$f$ is denoted $f_*$.  If $g: X\rightarrow Y$ is a flat morphism, the pullback of $g$ is denoted $g^*$ and by  definition $g^*[V]=[g^{-1} (V)]$, see  \cite[Chap 1, \S 1.7]{Fulton.Intersection}.
Given a formal series $Q(t)= \sum_{i=0}^\infty Q_i t^i$, we define $[t^n] Q(t)=Q_n$.

Our conventions for affine Dynkin diagrams are as follows. A projective Dynkin diagram is denoted ${M}_k$ where $M$ is $A$, $B$, $C$, $D$, $E$, $F$, or $G$, and $k$ is the number of nodes.
  An affine Dynkin diagram that becomes a projective Dynkin diagram $\mathfrak{g}$ after removing a node of multiplicity one is denoted $\widetilde{\mathfrak{g}}$.  We denote by $\widetilde{\mathfrak{g}}^t$ the 
    (the possibly twisted) affine Dynkin diagram whose  Cartan matrix  is the transpose of the Cartan matrix of $\widetilde{\mathfrak{g}}$.
The graph of   $\widetilde{\mathfrak{g}}^t$ is   obtained by exchanging the directions of all the arrows of $\widetilde{\mathfrak{g}}$.  
When the extra node is removed, the dual graph of  $\widetilde{\mathfrak{g}}^t$ reduces to the dual graph of the Langlands dual of $\mathfrak{g}$. 
      The affine Dynkin diagrams $\widetilde{\mathfrak{g}}^t$ and $\widetilde{\mathfrak{g}}$ are distinct only when $\mathfrak{g}$ is not simply laced (i.e., when $\mathfrak{g}$ is G$_2$, F$_4$, B$_k$,  or  C$_k$).
The notation $\widetilde{\mathfrak{g}}^t$ follows Carter\footnote{There is a  typo on page 570 of \cite{Carter} in the first Dynkin diagram of  $ \widetilde{\text{B}}_{\ell}$ on the top of the page, where the arrow is in the wrong direction but correctly oriented in the rest of the page. }  \cite[Appendix, p. 540-609]{Carter} and is equivalent to the notation $\widetilde{\mathfrak{g}}^\vee$ used by MacDonald in \S 5  of \cite{MR0357528}. 
 The multiplicities define a  zero vector  of the extended Cartan matrix.  In the notation of Kac \cite{Kac},  
$ \widetilde{\text{B}}_{\ell}^t$ ($\ell\geq 3$), $ \widetilde{\text{C}}_{\ell}^t $ ($\ell\geq 2$), $ \widetilde{\text{G}}_{2}^t$, and  $\widetilde{\text{F}}_{4}^t$ are respectively denoted 
  ${\text{A}}_{2\ell-1}^{(2)}$, ${\text{D}}_{\ell+1}^{(2)}$, ${\text{D}}_{4}^{(3)}$, and  ${\text{E}}_{6}^{(2)}$; while 
$ \widetilde{\text{B}}_{\ell}$ ($\ell\geq 3$), $ \widetilde{\text{C}}_{\ell} $ ($\ell\geq 2$), $ \widetilde{\text{G}}_{2}$, and  $\widetilde{\text{F}}_{4}$ are respectively denoted 
  ${\text{B}}_{\ell}^{(1)}$, ${\text{C}}_{\ell}^{(1)}$, $ {\text{G}}_{2}^{(1)}$, and  ${\text{F}}_{4}^{(1)}$. When $\mathfrak{g}$ is non-simply laced, the affine Dynkin diagrams $\widetilde{\mathfrak{g}}^t$ and $\widetilde{\mathfrak{g}}$ differ from each  by the directions of their  arrows and also by 
  the multiplicities of their nodes (see Figure \ref{Fig:AffineLieAlgebras}).

Given a complete intersection $Z$ of hypersurfaces  $Z_i=V(z_i)$ in a variety $X$, we denote the blowup $\widetilde{X}=\text{Bl}_Z X$ of $X$ along $Z$  with exceptional divisor $E=V(e)$ as 
$$
\begin{tikzpicture}
	\node(X0) at (0,0){${X}$};
	\node(X1) at (2.5,0){$\widetilde{X}.$};
	\draw[big arrow] (X1) -- node[above,midway]{$(z_1,\ldots,z_n|e)$} (X0);	
 \end{tikzpicture}
$$
\begin{figure}[htb]
\begin{center}
 \scalebox{1}{$			
\begin{array}{r l  l l }

												   \begin{array}{c}
												  \\
						 \widetilde{\text{B}}_{3+\ell}^t  \\
						 \\
						 \end{array}
						 &\scalebox{.9}{$\begin{array}{c} \begin{tikzpicture}
				\node[draw,circle,thick,scale=1.25,fill=black,label=above:{1}] (1) at (-.1,.7){};
				\node[draw,circle,thick,scale=1.25,label=above:{1}] (2) at (-.1,-.7){};	
				\node[draw,circle,thick,scale=1.25,label=above:{2}] (3) at (1,0){};
				\node[draw,circle,thick,scale=1.25,label=above:{2}] (4) at (2.1,0){};
				\node[draw,circle,thick,scale=1.25,label=above:{2}] (5) at (3.3,0){};	
				\node[draw,circle,thick,scale=1.25,label=above:{1}] (6) at (4.6,0){};	
				\draw[thick] (1) to (3) to (4);
				\draw[thick] (2) to (3);
				\draw[ultra thick, loosely dotted] (4) to (5) {};
				\draw[thick] (3.5,-0.05) --++ (.9,0){};
				\draw[thick] (3.5,+0.05) --++ (.9,0){};
				\draw[thick]
					(3.9,0) --++ (60:.25)
					(3.9,0) --++ (-60:.25);
			\end{tikzpicture}\end{array}$}		&

\scalebox{.9}{$\begin{array}{c} \begin{tikzpicture}
				\node[draw,circle,thick,scale=1.25,fill=black,label=above:{1}] (1) at (-.1,.7){};
				\node[draw,circle,thick,scale=1.25,label=above:{1}] (2) at (-.1,-.7){};	
				\node[draw,circle,thick,scale=1.25,label=above:{2}] (3) at (1,0){};
				\node[draw,circle,thick,scale=1.25,label=above:{2}] (4) at (2.1,0){};
				\node[draw,circle,thick,scale=1.25,label=above:{2}] (5) at (3.3,0){};	
				\node[draw,circle,thick,scale=1.25,label=above:{2}] (6) at (4.6,0){};	
				\draw[thick] (1) to (3) to (4);
				\draw[thick] (2) to (3);
				\draw[ultra thick, loosely dotted] (4) to (5) {};
				\draw[thick] (3.5,-0.05) --++ (.9,0){};
				\draw[thick] (3.5,+0.05) --++ (.9,0){};
				\draw[thick]
					(4.1,0) --++ (120:.25)
					(4.1,0) --++ (-120:.25);
			\end{tikzpicture}\end{array}$}				
			&
			  \begin{array}{c}
												  \\
						 \widetilde{\text{B}}_{3+\ell}  \\
						 \\
						 \end{array}
			\\
			
 \widetilde{\text{C}}_{2+\ell}^t  
						   &
			\scalebox{.9}{$\begin{array}{c} \begin{tikzpicture}
				\node[draw,circle,thick,scale=1.25,fill=black,label=above:{1}] (1) at (-.2,0){};
				\node[draw,circle,thick,scale=1.25,label=above:{1}] (3) at (.8,0){};
				\node[draw,circle,thick,scale=1.25,label=above:{1}] (4) at (1.8,0){};
				\node[draw,circle,thick,scale=1.25,label=above:{1}] (5) at (2.8,0){};	
				\node[draw,circle,thick,scale=1.25,label=above:{1}] (6) at (3.8,0){};	
				\node[draw,circle,thick,scale=1.25,label=above:{1}] (7) at (4.8,0){};	
				\draw[thick]   (3) to (4);
				\draw[thick] (5) to (6);
				\draw[ultra thick, loosely dotted] (4) to (5) {};
				\draw[thick] (4.,-0.05) --++ (.6,0){};
				\draw[thick] (4.,+0.05) --++ (.6,0){};
								\draw[thick] (-.2,-0.05) --++ (.8,0){};
				\draw[thick] (-.2,+0.05) --++ (.8,0){};

				\draw[thick]
					(4.4,0) --++ (-120:.25)
					(4.4,0) --++ (120:.25);
					\draw[thick]
					(0.2,0) --++ (-60:.25)
					(0.2,0) --++ (60:.25);
			\end{tikzpicture}\end{array}$}
& 
 \scalebox{.9}{$\begin{array}{c} \begin{tikzpicture}
				\node[draw,circle,thick,scale=1.25,fill=black,label=above:{1}] (1) at (-.2,0){};
				\node[draw,circle,thick,scale=1.25,label=above:{2}] (3) at (.8,0){};
				\node[draw,circle,thick,scale=1.25,label=above:{2}] (4) at (1.8,0){};
				\node[draw,circle,thick,scale=1.25,label=above:{2}] (5) at (2.8,0){};	
				\node[draw,circle,thick,scale=1.25,label=above:{2}] (6) at (3.8,0){};	
				\node[draw,circle,thick,scale=1.25,label=above:{1}] (7) at (4.8,0){};	
				\draw[thick]   (3) to (4);
				\draw[thick] (5) to (6);
				\draw[ultra thick, loosely dotted] (4) to (5) {};
				\draw[thick] (4.,-0.05) --++ (.6,0){};
				\draw[thick] (4.,+0.05) --++ (.6,0){};
				\draw[thick] (-.2,-0.05) --++ (.8,0){};
				\draw[thick] (-.2,+0.05) --++ (.8,0){};

				\draw[thick]
					(4.2,0) --++ (-60:.25)
					(4.2,0) --++ (60:.25);
					\draw[thick]
					(0.4,0) --++ (-120:.25)
					(0.4,0) --++ (120:.25);
			\end{tikzpicture}\end{array}$}

			&
			  \begin{array}{c}
												  \\
						 \widetilde{\text{C}}_{2+\ell}  \\
						 \\
						 \end{array}

			\\
\widetilde{\text{F}}_4^t 
&			

\scalebox{.93}{$\begin{array}{c}\begin{tikzpicture}
				\node[draw,circle,thick,scale=1.25,fill=black,label=above:{1}] (1) at (0,0){};
				\node[draw,circle,thick,scale=1.25,label=above:{2}] (2) at (1,0){};
				\node[draw,circle,thick,scale=1.25,label=above:{3}] (3) at (2,0){};
				\node[draw,circle,thick,scale=1.25,label=above:{2}] (4) at (3,0){};
				\node[draw,circle,thick,scale=1.25,label=above:{1}] (5) at (4,0){};
				\draw[thick] (1) to (2) to (3);
				\draw[thick]  (4) to (5);
				\draw[thick] (2.2,0.05) --++ (.6,0);
				\draw[thick] (2.2,-0.05) --++ (.6,0);
				\draw[thick]
					(2.4,0) --++ (60:.25)
					(2.4,0) --++ (-60:.25);
			\end{tikzpicture}\end{array}$} & 
			\scalebox{.93}{$\begin{array}{c}\begin{tikzpicture}
				\node[draw,circle,thick,scale=1.25,fill=black,label=above:{1}] (1) at (0,0){};
				\node[draw,circle,thick,scale=1.25,label=above:{2}] (2) at (1,0){};
				\node[draw,circle,thick,scale=1.25,label=above:{3}] (3) at (2,0){};
				\node[draw,circle,thick,scale=1.25,label=above:{4}] (4) at (3,0){};
				\node[draw,circle,thick,scale=1.25,label=above:{2}] (5) at (4,0){};
				\draw[thick] (1) to (2) to (3);
				\draw[thick]  (4) to (5);
				\draw[thick] (2.2,0.05) --++ (.6,0);
				\draw[thick] (2.2,-0.05) --++ (.6,0);
				\draw[thick]
					(2.6,0) --++ (120:.25)
					(2.6,0) --++ (-120:.25);
			\end{tikzpicture}\end{array}$}
				&
				  \begin{array}{c}
												  \\
						 \widetilde{\text{F}}_{4}  \\
						 \\
						 \end{array}
				\\

			\widetilde{\text{G}}_2^t

 &			\scalebox{.93}{$\begin{array}{c}\begin{tikzpicture}
				\node[draw,circle,thick,scale=1.25,label=above:{1}, fill=black] (1) at (0,0){};
				\node[draw,circle,thick,scale=1.25,label=above:{2}] (2) at (1.3,0){};
				\node[draw,circle,thick,scale=1.25,label=above:{1}] (3) at (2.6,0){};
				\draw[thick] (1) to (2);
				\draw[thick] (1.5,0.09) --++ (.9,0);
				\draw[thick] (1.5,-0.09) --++ (.9,0);
				\draw[thick] (1.5,0) --++ (.9,0);
				\draw[thick]
					(1.9,0) --++ (60:.25)
					(1.9,0) --++ (-60:.25);
			\end{tikzpicture}\end{array}
		$} 		& 
		\scalebox{.93}{$\begin{array}{c}\begin{tikzpicture}
				\node[draw,circle,thick,scale=1.25,label=above:{1}, fill=black] (1) at (0,0){};
				\node[draw,circle,thick,scale=1.25,label=above:{2}] (2) at (1.3,0){};
				\node[draw,circle,thick,scale=1.25,label=above:{3}] (3) at (2.6,0){};
				\draw[thick] (1) to (2);
				\draw[thick] (1.5,0.09) --++ (.9,0);
				\draw[thick] (1.5,-0.09) --++ (.9,0);
				\draw[thick] (1.5,0) --++ (.9,0);
				\draw[thick]
					(2,0) --++ (120:.25)
					(2,0) --++ (-120:.25);
			\end{tikzpicture}\end{array}
		$} 
&

						 \widetilde{\text{G}}_{2}  				\\\end{array}$}
										\end{center}	
												 \caption{ Twisted affine Lie algebras vs affine Lie algebras for non-simply laced cases.
												 Only those on the left appears in the theory of elliptic fibrations as dual graphs of  the fiber over the generic point of an irreducible component of the discriminant locus. 
												  \label{Fig:AffineLieAlgebras}}
						  \end{figure}
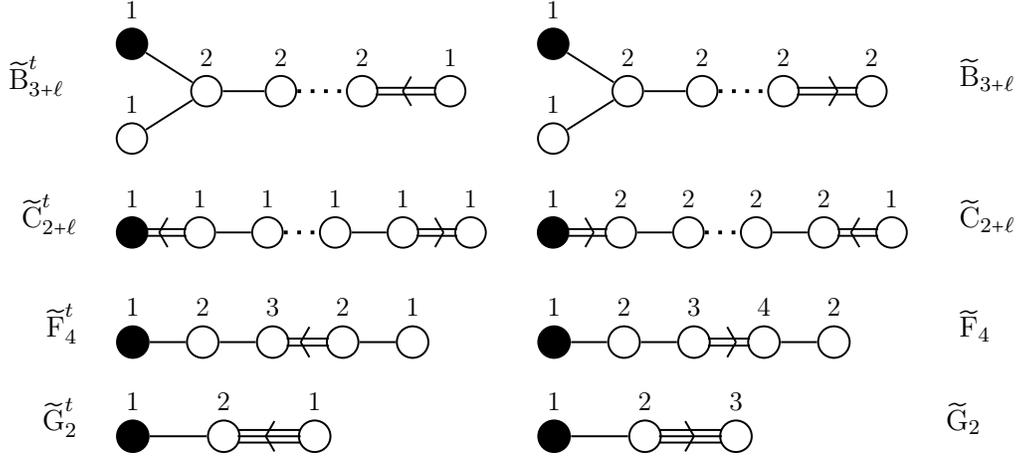
\subsection{$G$-models} \label{Sec:G.Models}
In this section, we recall how a Lie group is naturally associated with an elliptic fibration and introduce the notion of a $G$-model. 
 Our notation for dual graphs and Kodaira fibers is spelled out in \S \ref{Sec:Conv}, and Tables \ref{Table:Affine} and \ref{Table:DualGraph}. 
 See also  Appendix \ref{Sec:BasicNotions} for the definitions of a \emph{fiber type}, a \emph{generic fiber}, and a \emph{geometric generic fiber}. 
 \begin{defn}[$\mathcal{K}$-model]
 Let $\mathcal{ K}$ be the type of a generic fiber. Let $S\subset B$ be a smooth divisor of a projective variety $B$. 
 An elliptic fibration  $\varphi:Y\longrightarrow B$ over $B$  is said to be  a  \emph{$\mathcal{K}$-model} if
 \begin{enumerate}
\item The discriminant locus $\Delta(\varphi)$ contains as an irreducible component the divisor  $S\subset B$. 
\item  The generic fiber over $S$ is of type $\mathcal{ K}$.
\item Any other  fiber away from $S$ is irreducible. 
\end{enumerate}
 If the dual graph of $\mathcal{ K}$ corresponds to an affine Dynkin diagram of type $\widetilde{\mathfrak{g}}^t$, where $\mathfrak{g}$ is a Lie algebra, then the $\mathcal{K}$-model   is also called a \emph{$\mathfrak{g}$-model}. 
 \end{defn}

 In F-theory, a Lie group $G(\varphi)$ attached to a given elliptic fibration $\varphi: Y\longrightarrow B$ depends on the type of generic singular fibers and the Mordell-Weil group MW$(\varphi)$ of the elliptic fibration \cite{deBoer:2001wca}.  
  The Lie algebra $\mathfrak{g}$ associated to the elliptic fibration is then the Langlands dual $\mathfrak{g}^\vee= \bigoplus_i \mathfrak{g}_i^\vee$ of $\mathfrak{g}= \bigoplus_i \mathfrak{g}_i$. 
  If we denote by $\exp (\mathfrak{g}^\vee)$ the  unique (up to isomorphism) simply connected compact simple group  whose Lie algebra is $\mathfrak{g}^\vee$, then the group associated to the elliptic fibration $\varphi:Y\longrightarrow B$ is: 
\begin{equation}\nonumber
G (\varphi):= \frac{\exp (\mathfrak{g}^\vee)}{\text{MW}_{\text{tor}}(\varphi)} \times U(1)^{\mathrm{rk}\; \text{MW}(\varphi)},
\end{equation}
where $\mathrm{rk}\,\text{MW}(\varphi)$ is the rank of the Mordell-Weil group of $\varphi$ and MW$_{\text{tor}}(\varphi)$ is the torsion subgroup of the Mordell-Weil group of $\varphi$. 
Defining  properly the quotient of $\exp (\mathfrak{g}^\vee)$ by the Mordell--Weil group requires a choice of  embedding of the Mordell--Weil group in the center  of $\exp (\mathfrak{g}^\vee)$ \cite{Mayrhofer:2014opa}.

\begin{defn}[$G$-model]
 An elliptic fibration  $\varphi:Y\longrightarrow B$ with an associated Lie group $G=G(\varphi)$  is called a \emph{$G$-model}. 
  \end{defn}

If the reduced discriminant locus has a unique irreducible component  $S$ over which the generic fiber is not irreducible, the group $G(\varphi)$ is simple. 
The relevant fiber $\widetilde{\mathfrak{g}}^t$ can be realized by resolving the singularities of a Weierstrass model  derived from Tate's algorithm. 
The relation between the fiber type and the group $G(\varphi)$ is not one-to-one. 
For example, an SU(2)-model can be given by a divisor $S$ with a fiber of type  I$_2^\text{s}$, I$_2^{\text{ns}}$, III, IV$^{\text{ns}}$, or I$_3^{\text{ns}}$. 
For that reason, a given decorated Kodaira fiber provides a more refined characterization of a $G$-model.

\begin{exmp}
For $n\geq 4$, an SU($n$)-model is a I$_n^\text{s}$-model with a trivial Mordell-Weil group. 
For $n\geq 0$, a Spin(8+2$n$)-model is  an  I$_{n}^{*\text{s}}$-model with trivial Mordell-Weil group. 
For $n\geq 1$, a Spin(7+2$n$)-model is an I$_{n}^{*\text{ns}}$-model with  trivial Mordell-Weil group. 
A G$_2$-model is an I$_0^{*\text{ns}}$-model with a trivial Mordell-Weil group. 
A Spin(7)-model is an I$_{0}^{*\text{ss}}$-model with a trivial Mordell-Weil group. 
\end{exmp}

\begin{exmp}[See \cite{MP}]
The SO($3$), SO($5$), SO($6$), and SO($7$)-models are respectively I$_2^{\text{ns}}$, I$_4^\text{ns}$, I$_4^\text{s}$, and I$_0^\text{*\text{ss}}$-models with MW=$\mathbb{Z}/2\mathbb{Z}$.
For $n\geq 0$, an SO($8+2n$)-model is an I$_n^\text{*\text{s}}$-model with a Mordell-Weil group MW=$\mathbb{Z}/2\mathbb{Z}$. 
For $n\geq 1$, an SO($7+2n$)-model is an I$_n^\text{*\text{ns}}$-model with Mordell-Weil group MW=$\mathbb{Z}/2\mathbb{Z}$. 
\end{exmp}

 \begin{exmp}
If the Mordell-Weil group is trivial, $\mathcal{ K}$-models with $\mathcal{ K}=$  I$_2^\text{s}$, I$_2^{\text{ns}}$, III,  IV$^{\text{ns}}$, or I$_3^{\text{ns}}$, are all SU(2)-models. 
An A$_2$-model can be given by a IV$^\text{s}$-model or a I$_3$-model.  
If the Mordell-Weil group is trivial, both a IV$^\text{s}$-model or a I$_3^\text{s}$-model give a SU(3)-model. 
A  C$_{\ell}$-model can be given by an I$_{2\ell+2}^{\text{ns}}$-model or an I$_{2\ell+3}^{\text{ns}}$-model,  and if  the Mordell-Weil group is trivial, these both give a USp($2\ell$)-model.
\end{exmp}

\begin{rem}
Not all singular Weierstrass models are $G$-models as the reducible singular fibers might not appear in codimension one. See, for example, the Jacobians of the elliptic fibrations discussed in  \cite{AE1,AE2,EFY,EKY}.
\end{rem}

\subsection{The pushforward theorem and Jacobi's identity}
\label{subsec:pushthm}

As explained earlier, one of our key results is a pushforward theorem that streamlines all the computations of this paper. We present the pushforward theorem in this subsection.

\begin{restatable}{thm}{Push} \label{Thm:Push}
    Let the nonsingular variety $Z\subset X$ be a complete intersection of $d$ nonsingular hypersurfaces $Z_1$, \ldots, $Z_d$ meeting transversally in $X$. Let $E$ be the class of the exceptional divisor of the blowup $f:\widetilde{X}\longrightarrow X$ centered 
at $Z$.
 Let $\widetilde{Q}(t)=\sum_a f^* Q_a t^a$ be a formal power series with $Q_a\in A_*(X)$.
 We define the associated formal power series  ${Q}(t)=\sum_a Q_a t^a$ whose coefficients pullback to the coefficients of $\widetilde{Q}(t)$. 
 Then the pushforward $f_*\widetilde{Q}(E)$ is:
 $$
  f_*  \widetilde{Q}(E) =  \sum_{\ell=1}^d {Q}(Z_\ell) M_\ell, \quad \text{where} \quad  M_\ell=\prod_{\substack{m=1\\
 m\neq \ell}}^d  \frac{Z_m}{ Z_m-Z_\ell }.
 $$ 
\end{restatable}
We call the coefficient $M_\ell$ the {\em $\ell$-moment} of the blowup $f$. 
\begin{rem} 
 Given a blowup $f:\widetilde{X}\longrightarrow X$, any element $\alpha$ of the Chow ring $A_*(\widetilde{X})$  can be expressed as $\alpha= \sum_{n=0}^\infty f^* \alpha_i  E^i$ where $\alpha_i$ are elements of the Chow ring $A_*(X)$. So Theorem \ref{Thm:Push} can be used to pushforward any 
element of $A_*(\widetilde{X})$. 
\end{rem}

Theorem \ref{Thm:Push} is proven in \S \ref{Sec:Pushforward}. 
By the projection formula and the linearity of the pushforward, the proof of Theorem \ref{Thm:Push} is  almost trivial once it is established in the special case of a monic monomial $Q(t)=t^k$.  This special case is Lemma \ref{lem:PushE} on page \pageref{lem:PushE}. The proof of the Lemma  \ref{lem:PushE} relies on an identity due to Carl Gustave Jacobi that gives a partial fraction formula for  homogeneous complete symmetric polynomials:

\begin{restatable}[Jacobi]{lem}{Jacobi}
\label{Jacobi} Let $h_r(x_1, \ldots, x_d)$ be the  homogeneous complete symmetric polynomial of degree $r$ in $d$ variables of an integral domain. Then: 
$$
h_{r}(x_1, \cdots, x_d)=
\sum_{\ell=1}^d x_\ell^{r+d-1} \prod^d_{\substack{m=1\\
 m\neq \ell}} \frac{1}{ x_\ell -x_m }.
$$
\end{restatable}
Jacobi first proved this identity in 1825 in a slightly different form in his doctoral thesis\footnote{ \cite[Section III.17, p.  29-30]{Jacobi.Thesis}, Jacobi asserts:
$$
\prod_i \frac{1}{x-a_i}=\sum_i \frac{1}{x-a_i} \prod_{\ell\neq i}\frac{1}{a_\ell-a_i}
$$
} as a partial fraction  reformulation of the generating function of complete homogeneous  polynomials. 
Lemma  \ref{Jacobi} was rediscovered  in many different mathematical and physical problems, as discussed elegantly in \cite{Gustafson1983177}. For example, a proof using Schur polynomials was proposed as the solution to Exercise 7.4 of \cite{Stanley.II}.
 For a proof using integrals and residues see Appendix A of  \cite{Louck:1970jd}; for a proof using matrices, see  \cite{MR2840951}.
We give a short and simple proof of this identity in {Appendix} \ref{Sec:Jacobi}.

We also make use of a second pushforward theorem that concerns the projection from the ambient projective bundle to the base $B$ over which the Weierstrass model is defined.  
Let $\mathscr{V}$ be a vector bundle of rank $r$ over a nonsingular variety $B$. The Chow ring of a projective bundle $\pi: \mathbb{P}(\mathscr{V})\longrightarrow B$ is  isomorphic to the module $ A_*(B)[\zeta]$ modded out by the relation \cite[Remark 3.2.4, p. 55]{Fulton.Intersection}
 $$\zeta^r+ c_1(\pi^* \mathscr{V}) \zeta^{r-1}+\cdots+ c_i(\pi^* \mathscr{V}) \zeta^{r-i}+\cdots + c_r (\pi^* \mathscr{V})=0, \quad \zeta=c_1 \Big( \mathscr{O}_{\mathbb{P}(\mathscr{V})}(1)\Big).$$

\begin{thm}[{See \cite{AE1,AE2,Fullwood:SVW}}]\label{Thm:PushH}
Let $\mathscr{L}$ be a line bundle over a variety $B$ and $\pi: X_0=\mathbb{P}[\mathscr{O}_B\oplus\mathscr{L}^{\otimes 2} \oplus \mathscr{L}^{\otimes 3}]\longrightarrow B$ a projective bundle over $B$. 
 Let $\widetilde{Q}(t)=\sum_a \pi^* Q_a t^a$ be a formal power series in  $t$ such that $Q_a\in A_*(B)$. Define the auxiliary power series $Q(t)=\sum_a Q_a t^a$. 
Then 
$$
\pi_* \widetilde{Q}(H)=-2\left. \frac{{Q}(H)}{H^2}\right|_{H=-2L}+3\left. \frac{{Q}(H)}{H^2}\right|_{H=-3L}  +\frac{Q(0)}{6 L^2},
$$
 where  $L=c_1(\mathscr{L})$ and $H=c_1(\mathscr{O}_{X_0}(1))$ is the first Chern class of the dual of the tautological line bundle of $X_0$.
\end{thm}
\begin{proof}
Using the functoriality of Segre classses, we can write
\begin{equation}
\pi_* \left( \frac{1}{1-H} \right)=\frac{1}{(1+2L)(1+3L)}=\frac{-2}{1+2L}+\frac{3}{1+3L},\nonumber
\end{equation}
which can be expanded on the both sides. This gives the following expressions for the pushforward of each power of $H$:
\begin{equation}
\pi_* 1 = 0 ~~, ~~ \pi_* H = 0~~,~~\pi_* H^{i+2} =\left[ -2(-2)^i +3(-3)^i \right] L^i \nonumber
\end{equation}
where $i$ is nonnegative. Then, expanding $Q(H)$ as a power series with coefficients in $A_*(B)$,
\begin{equation}
\widetilde{Q}(H)=\sum_{i=0}^{\infty} \pi^* \alpha_i H^i =\pi^*\alpha_0 +(\pi^*\alpha_1) H +H^2 \sum_{k=0}^{\infty} (\pi^* \alpha_k) H^k,\nonumber
\end{equation}
the pushforward of $Q(H)$ can hence be computed as
\begin{align}
 \pi_* \widetilde{Q}(H) &=-2\sum_{k=0}^{\infty} \alpha_k (-2L)^k +3\sum_{k=0}^{\infty} \alpha_k (-3L)^k\nonumber \\
&= -2\frac{Q(H)-\alpha_1 H-\alpha_0}{H^2} \Big|_{H=-2L}+3\frac{Q(H)-\alpha_1 H-\alpha_0}{H^2}\Big|_{H=-3L} \nonumber\\
&= -2\frac{Q(H)}{H^2}\Big|_{H=-2L}+3\frac{Q(H)}{H^2}\Big|_{H=-3L}+\frac{Q(0)}{6L^2}.\nonumber
\end{align}
\end{proof}

\subsection{Strategy}
{We take an intersection theory point of view inspired by  Fulton \cite{Fulton.Intersection} and Aluffi \cite{Aluffi_CBU}, and use explicit crepant resolutions of Tate models to compute their Euler characteristics. }
Using Chern classes, we evaluate the Euler characteristic without dealing with the combinatorics of the fiber structure. 
Instead, we compute the pushforward of the homological Chern class of the variety to the base of the fibration. 
Since the Euler characteristics of two crepant resolutions of the same Weierstrass model are the same \cite{Batyrev.Betti}, we do not need to explore the network of all flops to arrive at our conclusions.

Our method for computing the Euler characteristics of $G$-models is as follows. 
Given a choice of Lie group $G$, we first use Tate's algorithm to determine a singular Weierstrass model $Y_0\longrightarrow B$  such that $G$ is the Lie group attached to the elliptic fibration following the F-theory algorithm discussed in \S \ref{Sec:G.Models}.
We then determine a crepant resolution $f:Y\longrightarrow Y_0$ of the singular Weierstrass model to obtain an explicit realization of the $G$-model as a smooth projective variety. 
By doing so, we retrieve the data necessary to compute the total homological Chern class of the crepant resolution $f:Y\longrightarrow Y_0$. 
We apply Theorem \ref{Thm:Push} repeatedly to push this class forward to the projective bundle $X_0$ in which the Weierstrass model is defined. 
Finally, we use Theorem \ref{Thm:PushH}  to push the total Chern class forward to $B$. In doing so, we obtain a generating function of the form 
$$
\chi(Y)=\int_B Q(L,S) c(B), \quad c(B):=c(TB)\cap [B],
$$
 where$\int_B$ indicates the degree,  $Q(L,S)$ is a rational function in $L$ and $S$ such that 
$$Q(L,0)=\frac{12L}{1+6L} c(B).$$
 $Q(L,0)$ is the generating function for the Euler characteristic of a smooth Weierstrass model \cite{AE1}. 
 The rational expression $
Q(L,S) c(B)$ is defined in the Chow ring $A_*(B)$ of the base. The expression $\chi(Y)$ is a generating function in the following sense. 
 If the base has dimension $d$, the Euler characteristic is then given by the coefficient of $t^d$ in a power series expansion in the parameter $t$: 
 $$
 \chi(Y)= [t^d] \   \Big(Q(tL,tS) c_t(TB)\Big), \quad \text{where} \quad d:=\dim\  B, 
 $$ 
 where $[t^n] g(t)=g_n$ for a formal series $g(t)=\sum_{i=0}^\infty g_i t^i$, and 
 $$c_t(TB)= 1+c_1( TB) t + \cdots+c_d(TB) t^d,$$
  is the Chern polynomial of the tangent bundle of $B$.

It follows from the adjunction formula that  one can further impose the Calabi-Yau condition by setting $L= c_1(TB)$; see Tables \ref{Table:Euler3} and \ref{Table:Euler4} for the Euler characteristics of elliptic threefold and fourfold $G$-models.

In Table \ref{tab:incl}, we organize the Lie algebras associated to our choices of Tate models into a network, where an arrow indicates inclusion as a subalgebra.
\begin{table}[t]
$$
\begin{tikzcd}[column sep=.6cm, row sep=.2cm]
  &  & &  \text{C}_2  \arrow[rightarrow]{dd}   \arrow[rightarrow]{r}  & {\text{A}}_3  \arrow[rightarrow]{dd}     \arrow[rightarrow]{r}  &  {\text{A}}_4   \arrow[rightarrow]{r}  & {\text{D}}_5   \arrow[rightarrow]{dd} & & \\
 &&&&&&&&&\\
 {\text{A}}_1  \arrow[rightarrow]{r}  & {\text{A}}_2   \arrow[rightarrow]{r} & {\text{G}}_2 \arrow[rightarrow]{r} & {\text{B}}_3 \arrow[rightarrow]{r} & {\text{D}} _4 \arrow[rightarrow]{r} &{\text{F}}_4 \arrow[rightarrow]{r} &{\text{E}}_6  \arrow[rightarrow]{r} & {\text{E}}_7  \arrow[rightarrow]{r} & {\text{E}}_8 \end{tikzcd}
$$
$$
\begin{tabular}{|c|c|c|c|c|c|c|c|c|c|c|c|c|}
\hline
I$_2^\text{s}$, \   I$_2^{\text{ns}}$,\    III, \   IV$^{\text{ns}}$,\   I$_3^{\text{ns}}$  &IV$^\text{s}$,  I$_3^\text{s}$ &  I$_4^\text{s}$ & I$_4^{\text{ns}}$ &  I$_5^\text{s}$ &  I$_0^{*\text{ns}}$ &  I$_0^{*\text{ss}}$ &  I$_0^{*\text{s}}$ &  IV$^{*\text{ns}}$ &   I$_1^{*\text{s}}$ & IV$^{*\text{s}}$ & III$^*$ & II$^*$\\
\hline
A$_1$ & A$_2$ & A$_3$ & C$_2$ & A$_4$ & G$_2$ & B$_3$ & D$_4$ & F$_4$ & D$_5$& E$_6$ & E$_7$ & E$_8$\\
\hline
\end{tabular}
$$

\caption{Models studied in this paper. }
\label{tab:incl}
\end{table} As is evident from Table \ref{tab:incl}, the results of this paper cover all instances of Kodaira fibers  with the exception of the general cases of  I$_k$ and I$_k^*$ that will be discussed in a follow up paper.
In particular, our list contains:
\begin{itemize}
\item $G$-models corresponding to Deligne exceptional series: 
$$
\{e\}\subset \text{A}_1 \subset \text{A}_2 \subset \text{G}_2 \subset \text{D}_4\subset \text{F}_4\subset \text{E}_6 \subset \text{E}_7\subset \text{E}_8.
$$
\item 
$G$-models for the extended exceptional series\footnote{\label{fn:exceptional} We recall that  the Dynkin diagram of E$_n$ is the same as  A$_n$ but with the $n$th node connected with the third node. In particular, $\text{E}_4\cong \text{A}_4$, $\text{E}_5\cong \text{D}_5$, 
E$_3=$ A$_2\times$ A$_1$, E$_2=$A$_2$, and E$_1=$A$_1$. 
}: 
$$
\{e\}\subset \text{A}_1 \subset \text{A}_2 \subset  \text{A}_3 \subset  \text{E}_4 \subset \text{E}_5 \subset \text{E}_6 \subset \text{E}_7\subset \text{E}_8.
$$
\item $G$-models for simple orthogonal groups of small rank\footnote{These models require a Mordell-Weil group $\mathbb{Z}/2\mathbb{Z}$; see \cite{MP}.}: 
$$
\{e\}\subset \text{SO}(3) \subset \text{SO}(5) \subset \text{SO}(6).
$$
\item $G$-models of the  I$_0^*$ series \cite{G2}:
$$
\{e\}\subset \text{G}_2 \subset \text{Spin}(7) \subset \text{Spin}(8).
$$

\end{itemize}

\subsection{Organization of the paper}

 The remainder of the paper is organized as follows.
  In Section \ref{sec:Euler} we discuss some general properties of the Euler characteristic of  an elliptic fibration. 
 In Section \ref{Sec:Pushforward} we discuss the pushforward theorem and explain the details of our computation of the Euler characteristic. Section \ref{sec:Hodge} then describes how these results can be used to calculate the Hodge numbers of Calabi-Yau threefold $G$-models. In Section \ref{sec:example}, we describe the simplest model, the SU(2)-model, as an example of our computation. We present the results of our computation in a series of tables in Section \ref{sec:results}. Finally, in Section \ref{sec:discuss} we conclude with a discussion of the computation and comment on possible future research directions. A proof of Jacobi's partial fraction identity is given in Appendix \ref{Sec:Jacobi}, an explanation of the Euler  characteristic as the degree of the top Chern class is given in Appendix \ref{appB.Euler}, and some basic facts about Kodaira fibers, elliptic fibrations, Weierstrass models and Tate's algorithm are collected in Appendix \ref{Sec:BasicNotions}.

\begin{table}[hbt]
\begin{center}
\scalebox{1}{$			
\begin{array}{|c | c | c |} 
\hline
\text{Fiber type} & \text{Dynkin diagram} & \text{Kodaira type}
\\\hline 
			\widetilde{\text{A}}_0
 &			\scalebox{.97}{$\begin{array}{c}\begin{tikzpicture}
				\node[draw,circle,thick,scale=1.25,label=above:{1}, fill=black] (1) at (0,0){};
			\end{tikzpicture}\end{array}
		$} 		
				& \text{I}_1, \quad \text{II}\\\hline

			\widetilde{\text{A}}_1
 &			\scalebox{.97}{$\begin{array}{c}\begin{tikzpicture}
				\node[draw,circle,thick,scale=1.25,label=above:{1}, fill=black] (1) at (0,0){};
				\node[draw,circle,thick,scale=1.25,label=above:{1}] (2) at (1.3,0){};
				\draw[thick] (0.15,0.1) --++ (.95,0);
				\draw[thick] (0.15,-0.09) --++ (.95,0);
				
			\end{tikzpicture}\end{array}
		$} 		& \text{I}_2^\text{s}, \quad \text{I}_2^{\text{ns}}, \quad\text{I}_3^{\text{ns}},\quad \text{III},\quad \text{IV}^{\text{ns}}
				\\\hline
{ \widetilde{\text{A}}_{\ell-1}\quad {(\ell\geq 3)}}
& \scalebox{.97}{$\begin{array}{c} \begin{tikzpicture}
				\node[draw,circle,thick,fill=black,scale=1.25,label=above:{1}] (0) at (90:1.1){};	
				\node[draw,circle,thick,scale=1.25,label=above:{1}] (1) at (-2,0){};
				\node[draw,circle,thick,scale=1.25,label=below:{1}] (2) at (-1,0){};
				\node[draw,circle,thick,scale=1.25,label=below:{1}] (3) at (1,0){};	
				\node[draw,circle,thick,scale=1.25,label=above:{1}] (4) at (2,0){};	
				\draw[thick]   (2)--(1)--(0)--(4)--(3);
\draw[ultra thick, loosely dotted] (2) to (3) {};

			\end{tikzpicture}\end{array}$}            &{ \text{I}_{\ell}^\text{s}   }    \\\hline

\widetilde{\text{D}}_{4+\ell}{\quad  (\ell\geq 0)}
						& 
						\scalebox{.97}{$\begin{array}{c} \begin{tikzpicture}
				\node[draw,circle,thick,scale=1.25,fill=black,label=above:{1}] (1) at (-.1,.7){};
				\node[draw,circle,thick,scale=1.25,label=above:{1}] (2) at (-.1,-.7){};	
				\node[draw,circle,thick,scale=1.25,label=above:{2}] (3) at (1,0){};
				\node[draw,circle,thick,scale=1.25,label=above:{2}] (4) at (2.1,0){};
				\node[draw,circle,thick,scale=1.25,label=above:{2}] (5) at (3.3,0){};	
				\node[draw,circle,thick,scale=1.25,label=above:{1}] (6) at (4.6,.7){};	
								\node[draw,circle,thick,scale=1.25,label=above:{1}] (7) at (4.6,-.7){};	
				\draw[thick] (1) to (3) to (4);
				\draw[thick] (2) to (3);
				\draw[ultra thick, loosely dotted] (4) to (5) {};
				\draw[thick] (5) to (6); 
				\draw[thick] (5) to (7);
			\end{tikzpicture}\end{array}$}&  \text{I}_{\ell}^{*\text{s}}    
												  \\\hline
												  
				\widetilde{\text{E}}_{6}								  &
			\scalebox{.92}{$\begin{array}{c}\begin{tikzpicture}
				\node[draw,circle,thick,scale=1.25,fill=black,label=below:{1}] (0) at (0,0){};
				\node[draw,circle,thick,scale=1.25,label=below:{2}] (1) at (1,0){};
				\node[draw,circle,thick,scale=1.25,label=below:{3}] (2) at (1*2,0){};
				\node[draw,circle,thick,scale=1.25,label=below:{2}] (3) at (1*3,0){};
				\node[draw,circle,thick,scale=1.25,label=below:{1}] (4) at (1*4,0){};
								\node[draw,circle,thick,scale=1.25,label=left:{2}] (5) at (1*2,1*1){};
																\node[draw,circle,thick,scale=1.25,label=above:{1}] (6) at (1*2,1*2){};
				\draw[thick] (0)--(1)--(2)--(3)--(4);
				\draw[thick] (2)--(5)--(6);

			\end{tikzpicture}\end{array}$}  &  \text{IV}^{*\text{s}}   
		
			\\\hline
\widetilde{\text{E}}_{7}	
								& \scalebox{.94}{$\begin{array}{c}
 \begin{tikzpicture}
				\node[draw,circle,thick,scale=1.25,fill=black, label=below:{1}] (0) at (0,0){ };
				\node[draw,circle,thick,scale=1.25,label=below:{2}] (1) at (1,0){ };
				\node[draw,circle,thick,scale=1.25,label=below:{3}] (2) at (1*2,0){ };
				\node[draw,circle,thick,scale=1.25,label=below:{4}] (3) at (1*3,0){ };
				\node[draw,circle,thick,scale=1.25,label=below:{3}] (4) at (1*4,0){ };
								\node[draw,circle,thick,scale=1.25,label=below:{2}] (5) at (5,0){ };
												\node[draw,circle,thick,scale=1.25,label=below:{1}] (6) at (6,0){ };
																\node[draw,circle,thick,scale=1.25,label=above:{2}] (7) at (3,1){ };
																								
				\draw[thick] (0)--(1)--(2)--(3)--(4)--(5)--(6);
				\draw[thick] (3)--(7);

			\end{tikzpicture}
			\end{array}$}		&  \text{III}^{*}    \\\hline	  
					\widetilde{\text{E}}_{8}							  
						&
						 \scalebox{.94}{
 $\begin{array}{c}
 \begin{tikzpicture}
				\node[draw,circle,thick,scale=1.25,fill=black,label=below:{1}] (0) at (0,0){};
				\node[draw,circle,thick,scale=1.25,label=below:{2}] (1) at (1*1,0){ };
				\node[draw,circle,thick,scale=1.25,label=below:{3}] (2) at (1*2,0){ };
				\node[draw,circle,thick,scale=1.25,label=below:{4}] (3) at (1*3,0){ };
				\node[draw,circle,thick,scale=1.25,label=below:{5}] (4) at (1*4,0){ };
								\node[draw,circle,thick,scale=1.25,label=below:{6}] (5) at (1*5,0){ };
												\node[draw,circle,thick,scale=1.25,label=below:{4}] (6) at (1*6,0){ };
																\node[draw,circle,thick,scale=1.25,label=above:{3}] (8) at (1*5,1*1){ };
																\node[draw,circle,thick,scale=1.25,label=below:{2}] (7) at (1*7,0){ };
								
				\draw[thick] (0)--(1)--(2)--(3)--(4)--(5)--(6)--(7);
				\draw[thick] (5)--(8);

			\end{tikzpicture}
			\end{array}$}  &  {\text{II}^{*} }  \\\hline
												  
												   \begin{array}{c}
												  \\
 \widetilde{\text{B}}_{3+\ell}^t {\quad (\ell\geq 0)}   \\
						 \\
						 \end{array}
						 &\scalebox{.97}{$\begin{array}{c} \begin{tikzpicture}
				\node[draw,circle,thick,scale=1.25,fill=black,label=above:{1}] (1) at (-.1,.7){};
				\node[draw,circle,thick,scale=1.25,label=above:{1}] (2) at (-.1,-.7){};	
				\node[draw,circle,thick,scale=1.25,label=above:{2}] (3) at (1,0){};
				\node[draw,circle,thick,scale=1.25,label=above:{2}] (4) at (2.1,0){};
				\node[draw,circle,thick,scale=1.25,label=above:{2}] (5) at (3.3,0){};	
				\node[draw,circle,thick,scale=1.25,label=above:{1}] (6) at (4.6,0){};	
				\draw[thick] (1) to (3) to (4);
				\draw[thick] (2) to (3);
				\draw[ultra thick, loosely dotted] (4) to (5) {};
				\draw[thick] (3.5,-0.05) --++ (.9,0){};
				\draw[thick] (3.5,+0.05) --++ (.9,0){};
				\draw[thick]
					(3.9,0) --++ (60:.25)
					(3.9,0) --++ (-60:.25);
			\end{tikzpicture}\end{array}$}		& {\begin{cases} \text{I}_{0}^{*\text{ss}}  \quad \text{for} \quad \ell= 0 \\  	 \text{I}_{\ell}^{*\text{ns}}\quad  \text{for}\quad \ell\geq 1\end{cases} }
			\\\hline
			
{\widetilde{\text{C}}_{2+\ell}^t}  {\quad  (\ell\geq 0)}
						   &
			\scalebox{.97}{$\begin{array}{c} \begin{tikzpicture}
				\node[draw,circle,thick,scale=1.25,fill=black,label=above:{1}] (1) at (-.2,0){};
				\node[draw,circle,thick,scale=1.25,label=above:{1}] (3) at (.8,0){};
				\node[draw,circle,thick,scale=1.25,label=above:{1}] (4) at (1.8,0){};
				\node[draw,circle,thick,scale=1.25,label=above:{1}] (5) at (2.8,0){};	
				\node[draw,circle,thick,scale=1.25,label=above:{1}] (6) at (3.8,0){};	
				\node[draw,circle,thick,scale=1.25,label=above:{1}] (7) at (4.8,0){};	
				\draw[thick]   (3) to (4);
				\draw[thick] (5) to (6);
				\draw[ultra thick, loosely dotted] (4) to (5) {};
				\draw[thick] (4.,-0.05) --++ (.6,0){};
				\draw[thick] (4.,+0.05) --++ (.6,0){};
								\draw[thick] (-.2,-0.05) --++ (.8,0){};
				\draw[thick] (-.2,+0.05) --++ (.8,0){};

				\draw[thick]
					(4.4,0) --++ (-120:.25)
					(4.4,0) --++ (120:.25);
					\draw[thick]
					(0.2,0) --++ (-60:.25)
					(0.2,0) --++ (60:.25);
			\end{tikzpicture}\end{array}$}
& {\text{I}_{4+2\ell}^{\text{ns}},\quad    \text{I}_{5+2\ell}^{\text{ns}}}

			\\\hline
\widetilde{\text{F}}_4^t 
&			

\scalebox{.97}{$\begin{array}{c}\begin{tikzpicture}
				\node[draw,circle,thick,scale=1.25,fill=black,label=above:{1}] (1) at (0,0){};
				\node[draw,circle,thick,scale=1.25,label=above:{2}] (2) at (1,0){};
				\node[draw,circle,thick,scale=1.25,label=above:{3}] (3) at (2,0){};
				\node[draw,circle,thick,scale=1.25,label=above:{2}] (4) at (3,0){};
				\node[draw,circle,thick,scale=1.25,label=above:{1}] (5) at (4,0){};
				\draw[thick] (1) to (2) to (3);
				\draw[thick]  (4) to (5);
				\draw[thick] (2.2,0.05) --++ (.6,0);
				\draw[thick] (2.2,-0.05) --++ (.6,0);
				\draw[thick]
					(2.4,0) --++ (60:.25)
					(2.4,0) --++ (-60:.25);
			\end{tikzpicture}\end{array}$} &  \text{IV}^{*\text{ns}}
				
				\\\hline

			\widetilde{\text{G}}_2^t

 &			\scalebox{.95}{$\begin{array}{c}\begin{tikzpicture}
				\node[draw,circle,thick,scale=1.25,label=above:{1}, fill=black] (1) at (0,0){};
				\node[draw,circle,thick,scale=1.25,label=above:{2}] (2) at (1.3,0){};
				\node[draw,circle,thick,scale=1.25,label=above:{1}] (3) at (2.6,0){};
				\draw[thick] (1) to (2);
				\draw[thick] (1.5,0.09) --++ (.9,0);
				\draw[thick] (1.5,-0.09) --++ (.9,0);
				\draw[thick] (1.5,0) --++ (.9,0);
				\draw[thick]
					(1.9,0) --++ (60:.25)
					(1.9,0) --++ (-60:.25);
			\end{tikzpicture}\end{array}
		$} 		&  \text{I}_0^{*\text{ns}}

				\\\hline\end{array}$}
	\end{center}
	\caption{
{ Affine Dynkin diagrams appearing as dual graphs of decorated Kodaira fibers. 
 } \label{Table:Affine}
}
\end{table}
\clearpage

\newpage
\clearpage

\begin{table}[htb]
\begin{center}
\scalebox{.95}{
$			
\begin{array}{|c|c| c  |} \hline
\vrule width 0pt height 3ex 
 \text{Fiber Type} & \text{ Dual graph  } & \text{Dual graph of Geometric fiber } \\\hline
 \			
			\widetilde{\text{A}}_1
 &			\scalebox{1}{$\begin{array}{c}\begin{tikzpicture}
				\node[draw,circle,thick,scale=1.25,label=above:{1}, fill=black] (1) at (0,0){};
				\node[draw,circle,thick,scale=1.25,label=above:{1}] (2) at (1.3,0){};
				\draw[thick] (0.15,0.1) --++ (.95,0);
				\draw[thick] (0.15,-0.09) --++ (.95,0);
				
			\end{tikzpicture}\end{array}
		$} &
		\scalebox{1}{$\begin{array}{c}\begin{tikzpicture}
				\node[draw,circle,thick,scale=1.25,label=above:{1}, fill=black] (1) at (-.6,0){};
				\node[draw,circle,thick,scale=1.25,label=above:{1}] (2) at (45:.6){};
				\node[draw,circle,thick,scale=1.25,label=below:{1}] (3) at (-45:.6){};
				\draw[thick] (0,0) --(1);
				\draw[thick] (0,0) --(2);
				\draw[thick] (0,0) --(3);
								\draw[<->,>=stealth',semithick,dashed]  (.8,-.5) arc (-30:30:1cm);
			\end{tikzpicture}\end{array}$}

				\\\hline

						 \begin{array}{c}

						  \text{I}^{*\text{ns}}_{\ell-3}\\
						  \\
						 \widetilde{\text{B}}_{\ell}^t  \\
						 \\
						 (\ell\geq 3)
						 \end{array}
						 &\scalebox{1}{$\begin{array}{c} \begin{tikzpicture}
				\node[draw,circle,thick,scale=1.25,fill=black,label=above:{1}] (1) at (-.1,.7){};
				\node[draw,circle,thick,scale=1.25,label=above:{1}] (2) at (-.1,-.7){};	
				\node[draw,circle,thick,scale=1.25,label=above:{2}] (3) at (1,0){};
				\node[draw,circle,thick,scale=1.25,label=above:{2}] (4) at (2.1,0){};
				\node[draw,circle,thick,scale=1.25,label=above:{2}] (5) at (3.3,0){};	
				\node[draw,circle,thick,scale=1.25,label=above:{1}] (6) at (4.6,0){};	
				\draw[thick] (1) to (3) to (4);
				\draw[thick] (2) to (3);
				\draw[ultra thick, loosely dotted] (4) to (5) {};
				\draw[thick] (3.5,-0.05) --++ (.9,0){};
				\draw[thick] (3.5,+0.05) --++ (.9,0){};
				\draw[thick]
					(3.9,0) --++ (60:.25)
					(3.9,0) --++ (-60:.25);
			\end{tikzpicture}\end{array}$}&
			
\scalebox{1}{$\begin{array}{c} \begin{tikzpicture}
				\node[draw,circle,thick,scale=1.25,fill=black,label=above:{1}] (1) at (-.1,.7){};
				\node[draw,circle,thick,scale=1.25,label=above:{1}] (2) at (-.1,-.7){};	
				\node[draw,circle,thick,scale=1.25,label=above:{2}] (3) at (1,0){};
				\node[draw,circle,thick,scale=1.25,label=above:{2}] (4) at (2.1,0){};
				\node[draw,circle,thick,scale=1.25,label=above:{2}] (5) at (3.3,0){};	
				\node[draw,circle,thick,scale=1.25,label=above:{1}] (6) at (4.6,.7){};	
								\node[draw,circle,thick,scale=1.25,label=above:{1}] (7) at (4.6,-.7){};	
				\draw[thick] (1) to (3) to (4);
				\draw[thick] (2) to (3);
				\draw[ultra thick, loosely dotted] (4) to (5) {};
				\draw[thick] (5) to (6); 
				\draw[thick] (5) to (7);
				\draw[<->,>=stealth',semithick,dashed]  (5,-.5) arc (-30:30:1.2cm);
			\end{tikzpicture}\end{array}$}
			
			\\\hline
			\begin{array}{c}
			     \text{I}_{2\ell+2}^{\text{ns}}
			\\
			\\
			\vrule width 0pt height 3ex 
 \widetilde{\text{C}}_{\ell+1}^t  \\
						 \\
						 (\ell\geq 1)
 \end{array}
   &
			\scalebox{.95}{$\begin{array}{c} \begin{tikzpicture}
				\node[draw,circle,thick,scale=1.25,fill=black,label=above:{1}] (1) at (-.2,0){};
				\node[draw,circle,thick,scale=1.25,label=above:{1}] (3) at (.8,0){};
				\node[draw,circle,thick,scale=1.25,label=above:{1}] (4) at (1.8,0){};
				\node[draw,circle,thick,scale=1.25,label=above:{1}] (5) at (2.8,0){};	
				\node[draw,circle,thick,scale=1.25,label=above:{1}] (6) at (3.8,0){};	
				\node[draw,circle,thick,scale=1.25,label=above:{1}] (7) at (4.8,0){};	
				\draw[thick]   (3) to (4);
				\draw[thick] (5) to (6);
				\draw[ultra thick, loosely dotted] (4) to (5) {};
				\draw[thick] (4.,-0.05) --++ (.6,0){};
				\draw[thick] (4.,+0.05) --++ (.6,0){};
								\draw[thick] (-.2,-0.05) --++ (.8,0){};
				\draw[thick] (-.2,+0.05) --++ (.8,0){};

				\draw[thick]
					(4.4,0) --++ (-120:.25)
					(4.4,0) --++ (120:.25);
					\draw[thick]
					(0.2,0) --++ (-60:.25)
					(0.2,0) --++ (60:.25);
			\end{tikzpicture}\end{array}$}

			&

			\scalebox{.95}{$\begin{array}{c} \begin{tikzpicture}
				\node[draw,circle,thick,scale=1.25,fill=black,label=above:{1}] (1) at (-.2,0){};	
				\node[draw,circle,thick,scale=1.25,label=below:{1}] (3a) at (.8,-.8){};
				\node[draw,circle,thick,scale=1.25,label=below:{1}] (4a) at (1.8,-.8){};
				\node[draw,circle,thick,scale=1.25,label=below:{1}] (5a) at (2.8,-.8){};	
				\node[draw,circle,thick,scale=1.25,label=below:{1}] (6a) at (3.8,-.8){};	
								\node[draw,circle,thick,scale=1.25,label=above:{1}] (3b) at (.8,.8){};
				\node[draw,circle,thick,scale=1.25,label=above:{1}] (4b) at (1.8,.8){};
				\node[draw,circle,thick,scale=1.25,label=above:{1}] (5b) at (2.8,.8){};	
				\node[draw,circle,thick,scale=1.25,label=above:{1}] (6b) at (3.8,.8){};	

				\node[draw,circle,thick,scale=1.25,label=above:{1}] (7) at (4.8,0){};	
				\draw[thick]   (4b)--(3b)--(1)--(3a)--(4a);
				\draw[thick]   (5b)--(6b)--(7)--(6a)--(5a);
				\draw[ultra thick, loosely dotted] (4a) to (5a) {};
								\draw[ultra thick, loosely dotted] (4b) to (5b) {};
\draw[<->,>=stealth',semithick,dashed] ($(4a)+(0,0.3)$) --($(4b)-(0,0.3)$) {};
\draw[<->,>=stealth',semithick,dashed] ($(5a)+(0,0.3)$) --($(5b)-(0,0.3)$) {};
\draw[<->,>=stealth',semithick,dashed] ($(6a)+(0,0.3)$) --($(6b)-(0,0.3)$) {};
\draw[<->,>=stealth',semithick,dashed] ($(3a)+(0,0.3)$) --($(3b)-(0,0.3)$) {};
			\end{tikzpicture}\end{array}$}

			\\\hline

			\begin{array}{c}
			     \text{I}_{2\ell+3}^{\text{ns}}
			\\
			\\
			\vrule width 0pt height 3ex 
 \widetilde{\text{C}}_{\ell+1}^t   \\
						 \\
						 (\ell\geq 1)
 \end{array}
   &
			\scalebox{.95}{$\begin{array}{c} \begin{tikzpicture}
				\node[draw,circle,thick,scale=1.25,fill=black,label=above:{1}] (1) at (-.2,0){};
				\node[draw,circle,thick,scale=1.25,label=above:{1}] (3) at (.8,0){};
				\node[draw,circle,thick,scale=1.25,label=above:{1}] (4) at (1.8,0){};
				\node[draw,circle,thick,scale=1.25,label=above:{1}] (5) at (2.8,0){};	
				\node[draw,circle,thick,scale=1.25,label=above:{1}] (6) at (3.8,0){};	
				\node[draw,circle,thick,scale=1.25,label=above:{1}] (7) at (4.8,0){};	
				\draw[thick]   (3) to (4);
				\draw[thick] (5) to (6);
				\draw[ultra thick, loosely dotted] (4) to (5) {};
				\draw[thick] (4.,-0.05) --++ (.6,0){};
				\draw[thick] (4.,+0.05) --++ (.6,0){};
								\draw[thick] (-.2,-0.05) --++ (.8,0){};
				\draw[thick] (-.2,+0.05) --++ (.8,0){};

				\draw[thick]
					(4.4,0) --++ (-120:.25)
					(4.4,0) --++ (120:.25);
					\draw[thick]
					(0.2,0) --++ (-60:.25)
					(0.2,0) --++ (60:.25);
			\end{tikzpicture}\end{array}$}

			&
			
				\scalebox{.95}{$\begin{array}{c} \begin{tikzpicture}
				\node[draw,circle,thick,scale=1.25,fill=black,label=above:{1}] (1) at (-.2,0){};	
				\node[draw,circle,thick,scale=1.25,label=below:{1}] (3a) at (.8,-.8){};
				\node[draw,circle,thick,scale=1.25,label=below:{1}] (4a) at (1.8,-.8){};
				\node[draw,circle,thick,scale=1.25,label=below:{1}] (5a) at (2.8,-.8){};	
				\node[draw,circle,thick,scale=1.25,label=below:{1}] (6a) at (3.8,-.8){};	
								\node[draw,circle,thick,scale=1.25,label=above:{1}] (3b) at (.8,.8){};
				\node[draw,circle,thick,scale=1.25,label=above:{1}] (4b) at (1.8,.8){};
				\node[draw,circle,thick,scale=1.25,label=above:{1}] (5b) at (2.8,.8){};	
				\node[draw,circle,thick,scale=1.25,label=above:{1}] (6b) at (3.8,.8){};	

				\node[draw,circle,thick,scale=1.25,label=below:{1}] (7a) at (4.8,-.8){};
				\node[draw,circle,thick,scale=1.25,label=above:{1}] (7b) at (4.8,.8){};	
				\draw[thick]   (4b)--(3b)--(1)--(3a)--(4a);
				\draw[thick]   (5b)--(6b)--(7b)--(7a)--(6a)--(5a);
				\draw[ultra thick, loosely dotted] (4a) to (5a) {};
								\draw[ultra thick, loosely dotted] (4b) to (5b) {};

\draw[<->,>=stealth',semithick,dashed] ($(4a)+(0,0.3)$) --($(4b)-(0,0.3)$) {};
\draw[<->,>=stealth',semithick,dashed] ($(5a)+(0,0.3)$) --($(5b)-(0,0.3)$) {};
\draw[<->,>=stealth',semithick,dashed] ($(6a)+(0,0.3)$) --($(6b)-(0,0.3)$) {};
\draw[<->,>=stealth',semithick,dashed] ($(3a)+(0,0.3)$) --($(3b)-(0,0.3)$) {};
\draw[<->,>=stealth',semithick,dashed] (5.2,-1) arc (-70:70:1) {};
			\end{tikzpicture}\end{array}$}

			\\\hline
			\begin{array}{c}
			\text{IV}^{* \text{ns}}
			\\ 
			\vrule width 0pt height 3ex 
\widetilde{\text{F}}_4^t      
\end{array}
&			

\scalebox{1}{$\begin{array}{c}\begin{tikzpicture}
				\node[draw,circle,thick,scale=1.25,fill=black,label=above:{1}] (1) at (0,0){};
				\node[draw,circle,thick,scale=1.25,label=above:{2}] (2) at (1,0){};
				\node[draw,circle,thick,scale=1.25,label=above:{3}] (3) at (2,0){};
				\node[draw,circle,thick,scale=1.25,label=above:{2}] (4) at (3,0){};
				\node[draw,circle,thick,scale=1.25,label=above:{1}] (5) at (4,0){};
				\draw[thick] (1) to (2) to (3);
				\draw[thick]  (4) to (5);
				\draw[thick] (2.2,0.05) --++ (.6,0);
				\draw[thick] (2.2,-0.05) --++ (.6,0);
				\draw[thick]
					(2.4,0) --++ (60:.25)
					(2.4,0) --++ (-60:.25);
			\end{tikzpicture}\end{array}$} &
			\scalebox{1}{$\begin{array}{c}\begin{tikzpicture}
				\node[draw,circle,thick,scale=1.25,fill=black,label=below:{1}] (0) at (0,0){};
				\node[draw,circle,thick,scale=1.25,label=below:{2}] (1) at (.8,0){};
				\node[draw,circle,thick,scale=1.25,label=below:{3}] (2) at (.8*2,0){};
				\node[draw,circle,thick,scale=1.25,label=below:{2}] (3) at (.8*3,0){};
				\node[draw,circle,thick,scale=1.25,label=below:{1}] (4) at (.8*4,0){};
								\node[draw,circle,thick,scale=1.25,label=left:{2}] (5) at (.8*2,.8*1){};
																\node[draw,circle,thick,scale=1.25,label=above:{1}] (6) at (.8*2,.8*2){};
				\draw[thick] (0)--(1)--(2)--(3)--(4);
				\draw[thick] (2)--(5)--(6);
				\draw[<->,>=stealth',semithick,dashed]  (2.5,0.3) arc (25:65:1.3cm);
				\draw[<->,>=stealth',semithick,dashed]  (3.3,0.3) arc (25:65:3cm);

			\end{tikzpicture}\end{array}$}
		
			\\
			\hline
			\begin{array}{c}
			\text{I}^{*\text{ss}}_{0}
			\\
			\widetilde{\text{B}}_3^t
\end{array}
 &			\scalebox{1}{$\begin{array}{c}\begin{tikzpicture}
				\node[draw,circle,thick,scale=1.25,label=below:{1}] (1) at (0,-.5){};
				\node[draw,circle,thick,scale=1.25,label=above:{2}] (2) at (1.3,0){};
				\node[draw,circle,thick,scale=1.25,label=above:{1}] (3) at (2.6,0){};
								\node[draw,circle,thick,scale=1.25,label=above:{1}, fill=black] (4) at (0,.5){};
				\draw[thick] (1) to (2);\draw[thick] (2) to (4);
				\draw[thick] (1.5,0.09) --++ (.9,0);
				\draw[thick] (1.5,-0.09) --++ (.9,0);
				\draw[thick]
					(1.9,0) --++ (60:.25)
					(1.9,0) --++ (-60:.25);
			\end{tikzpicture}\end{array}
		$} & 
		\scalebox{.8}{$\begin{array}{c}\begin{tikzpicture}
				\node[draw,circle,thick,scale=1.25,label=30:{2}] (0) at (0,0){};
				\node[draw,circle,thick,scale=1.25,label=below:{1}, fill=black] (1) at (-1,0){};
				\node[draw,circle,thick,scale=1.25,label=below:{1}] (2) at (1,0){};
				\node[draw,circle,thick,scale=1.25,label=above:{1}] (3) at (90:1){};      								\node[draw,circle,thick,scale=1.25,label=below:{1}] (4) at (90:-1){};
				\draw[thick] (0) to (1);
				\draw[thick] (0) to (2);
				\draw[thick] (0) to (3);
				\draw[thick] (0) to (4);
								\draw[<->,>=stealth',semithick,dashed]  (1.1,0.3) arc (25:65:1.7cm);
			\end{tikzpicture}\end{array}
		$}

				\\\hline

			\begin{array}{c}
			\text{I}^{*\text{ns}}_{0}
			\\
			\widetilde{\text{G}}_2^t
\end{array}
 &			\scalebox{1}{$\begin{array}{c}\begin{tikzpicture}
				\node[draw,circle,thick,scale=1.25,label=above:{1}, fill=black] (1) at (0,0){};
				\node[draw,circle,thick,scale=1.25,label=above:{2}] (2) at (1.3,0){};
				\node[draw,circle,thick,scale=1.25,label=above:{1}] (3) at (2.6,0){};
				\draw[thick] (1) to (2);
				\draw[thick] (1.5,0.09) --++ (.9,0);
				\draw[thick] (1.5,-0.09) --++ (.9,0);
				\draw[thick] (1.5,0) --++ (.9,0);
				\draw[thick]
					(1.9,0) --++ (60:.25)
					(1.9,0) --++ (-60:.25);
			\end{tikzpicture}\end{array}
		$} & 
		\scalebox{.8}{$\begin{array}{c}\begin{tikzpicture}
				\node[draw,circle,thick,scale=1.25,label=30:{2}] (0) at (0,0){};
				\node[draw,circle,thick,scale=1.25,label=above:{1}, fill=black] (1) at (-1,0){};
				\node[draw,circle,thick,scale=1.25,label=right:{1}] (2) at (1,0){};
				\node[draw,circle,thick,scale=1.25,label=above:{1}] (3) at (90:1){};
								\node[draw,circle,thick,scale=1.25,label=below:{1}] (4) at (90:-1){};
				\draw[thick] (0) to (1);
				\draw[thick] (0) to (2);
				\draw[thick] (0) to (3);
				\draw[thick] (0) to (4);
				\draw[<->,>=stealth',semithick,dashed]  (1.1,0.3) arc (25:65:1.7cm);
				\draw[<->,>=stealth',semithick,dashed]  (1.1,-0.3) arc (-25:-65:1.7cm);
			\end{tikzpicture}\end{array}
		$}

				\\\hline

				\end{array}$}
	\end{center}
	\caption{
{ Dual graphs for elliptic fibrations}  \label{Table:DualGraph}.
}
\end{table}

\clearpage

\begin{table}[hbt]
\begin{center}
\scalebox{.98}{
\begin{tabular}{|c|c|c| c | c  |c|c|c|}
\hline
& & & & & & & \\
{\small Type }&{\small  $v(c_4)$} &{\small $v(c_6)$ }& {\small $v(\Delta)$}

                                                                                          & $j$  &

                                                                                          {\small  Monodromy}                & Fiber  &
                                                                                          \begin{tabular}{c}
                                                                                          Dual \\
                                                                                          Graph
                                                                                          \end{tabular}  \\
& & & & &   & & \\

                               \hline
I$_0$ & $\geq 0$  & $\geq 0$ &  {\small $0$} &
{
$\mathbb{C}$
}
 & I$_2$ & Smooth & -    \\
\hline
& & & & & & &\\
I$_1$ & $0$ & $0$ &  {\small $1$}  & $\infty$ &
$
\begin{pmatrix}
1& 1\\
0 & 1
\end{pmatrix}
$

 &
 \begin{tikzpicture}[scale=1.5]
								\draw[scale=.5,domain=-1.2:1.2,variable=\x,  thick] plot({\x*\x-1,\x*\x*\x-\x-5});
							\end{tikzpicture}

& $\widetilde{\text{A}}_0$ 
\\
\hline
&&&&&&&\\
II & $\geq 1$ &  $1$ & $2$   &$0$&
{\small $
\begin{pmatrix}
1& 1\\
-1 & 0
\end{pmatrix}
$

}
&
 \begin{tikzpicture}[scale=1.8]
								\draw[scale=1,domain=-.7:.7,variable=\x,  thick] plot({\x*\x,1.5*\x*\x*\x*\x*\x});
							\end{tikzpicture}

& $\widetilde{\text{A}}_0$ 
 \\
\hline
&&&&&&&\\
III& {\small$1$ } & {\small $\geq 2$ }  & {\small  $3$} &  {\small $1728$} &
{\small $
\begin{pmatrix}
0 & 1\\
-1 & 0
\end{pmatrix}
$
}
  &
 \begin{tikzpicture}[scale=1.5]
								\draw[scale=.8,domain=-.7:.7,variable=\x,  thick] plot({\x*\x,.7*\x});
								\draw[scale=.8,domain=-.7:.7,variable=\x,  thick] plot({-\x*\x,.7*\x});
							\end{tikzpicture}

 & $\widetilde{\text{A}}_1$ 
 \\
\hline
IV &  {\small $\geq 2$} & {\small $2$} & {\small $4$}  &    {\small $0$} &
{\small $
\begin{pmatrix}
0 & 1\\
-1 & -1
\end{pmatrix}
$
}
 &
{  \setlength{\unitlength}{1 mm}
\begin{picture}(10,10)(-5,-2)
\put(0,0){\qbezier(-2.5,-4.33)(0,0)(2.5,4.33)
\qbezier(-2.5,4.33)(0,0)(2.5,-4.33)
\qbezier(-5,0)(0,0)(5,0)
}
\end{picture}} & $\widetilde{\text{A}}_2$
\\
\hline
I$_n$ &{\small $0$} &  {\small $0$} & {\small $n>1$ }  & $\infty$ &
{\small $
\begin{pmatrix}
1& n\\
0 & 1
\end{pmatrix}
$
}
 &
 \scalebox{.95}{$\begin{array}{c} \begin{tikzpicture}
				\node[draw,circle,thick,scale=.9,label=above:{}] (0) at (90:1.1){1};	
				\node[draw,circle,thick,scale=.9,label=above:{}] (1) at (-2,0){1};
				\node[draw,circle,thick,scale=.9,label=above:{}] (2) at (-1,0){1};
				\node[draw,circle,thick,scale=.9,label=above:{}] (3) at (1,0){1};	
				\node[draw,circle,thick,scale=.9,label=above:{}] (4) at (2,0){1};	
												\draw[thick]   (2)--(1)--(0)--(4)--(3);
\draw[ultra thick, loosely dotted] (2) to (3) {};

			\end{tikzpicture}\end{array}$}

& $\widetilde{\text{A}}_{n-1} $ 
 \\
\hline
I$^*_n$ & {\small $2$} & {\small $\geq 3$ } & {\footnotesize $n+6$} & {\small  $\infty$} &
{\small $
\begin{pmatrix}
-1& -n\\
0 &- 1
\end{pmatrix}
$
}
  &
\scalebox{.8}{   
\setlength{\unitlength}{.9 mm}
\begin{tikzpicture}[scale=1.1]
				\node[draw,circle,thick,scale=1.25,label=above:{}] (1) at (-.1,.7){1};
				\node[draw,circle,thick,scale=1.25,label=above:{}] (2) at (-.1,-.7){1};	
				\node[draw,circle,thick,scale=1.25,label=above:{}] (3) at (1,0){2};
				\node[draw,circle,thick,scale=1.25,label=above:{}] (4) at (2.1,0){ 2};
				\node[draw,circle,thick,scale=1.25,label=above:{}] (5) at (3.3,0){2};	
				\node[draw,circle,thick,scale=1.25,label=above:{}] (6) at (4.6,.7){ 1};	
								\node[draw,circle,thick,scale=1.25,label=above:{}] (7) at (4.6,-.7){ 1};	
				\draw[thick] (1) to (3) to (4);
				\draw[thick] (2) to (3);
				\draw[ultra thick, loosely dotted] (4) to (5) {};
				\draw[thick] (5) to (6); 
				\draw[thick] (5) to (7);
			\end{tikzpicture}
  }

&  $\widetilde{\text{D}}_{n+4}$ 
\\ 
 \cline{2-4}  &  {\small $\geq 2$} &{\small  $3$} &{\footnotesize  $n+6$} &  
& & &   \\
\hline
IV$^*$ & {\small  $\begin{matrix}\\   \geq 3\\  \\  \end{matrix}$} &{\small  $4$ }& {\small $8$ }&    {\small $0$
}&
{\small $
\begin{pmatrix}
-1& -1\\
1 & 0
\end{pmatrix}
$
}

&
\setlength{\unitlength}{.9 mm}
\begin{tikzpicture}
				\node[draw,circle,thick,scale=1,label=below:{}] (0) at (0,0){1};
				\node[draw,circle,thick,scale=1,label=below:{}] (1) at (1,0){ 2};
				\node[draw,circle,thick,scale=1,label=below:{}] (2) at (1*2,0){ 3};
				\node[draw,circle,thick,scale=1,label=below:{}] (3) at (1*3,0){ 2};
				\node[draw,circle,thick,scale=1,label=below:{}] (4) at (1*4,0){ 1};
								\node[draw,circle,thick,scale=1,label=left:{}] (5) at (1*2,1*1){ 2};
																\node[draw,circle,thick,scale=1,label=above:{}] (6) at (2,2){ 1};
				\draw[thick] (0)--(1)--(2)--(3)--(4);
				\draw[thick] (2)--(5)--(6);

			\end{tikzpicture}

& $\widetilde{\text{E}}_6$ 

 \\
\hline
III$^*$ &
{\small  $\begin{matrix}\\   3\\   \\  \end{matrix}$ }& {\small  $\geq 5$} & {\small $9$}&  {\small $1728$}
& {\small  $
\begin{pmatrix}
0& -1\\
1 & 0
\end{pmatrix}
$}
 &
 \scalebox{.8}{
 \setlength{\unitlength}{1.2 mm}
 $\begin{array}{c}
 \begin{tikzpicture}
				\node[draw,circle,thick,scale=1,label=below:{}] (0) at (0,0){ 1};
				\node[draw,circle,thick,scale=1,label=below:{}] (1) at (1,0){ 2};
				\node[draw,circle,thick,scale=1,label=below:{}] (2) at (1*2,0){ 3};
				\node[draw,circle,thick,scale=1,label=below:{}] (3) at (1*3,0){ 4};
				\node[draw,circle,thick,scale=1,label=below:{}] (4) at (1*4,0){ 3};
								\node[draw,circle,thick,scale=1,label=below:{}] (5) at (5,0){ 2};
												\node[draw,circle,thick,scale=1,label=below:{}] (6) at (6,0){ 1};
																\node[draw,circle,thick,scale=1,label=below:{}] (7) at (3,1){2};
																								
				\draw[thick] (0)--(1)--(2)--(3)--(4)--(5)--(6);
				\draw[thick] (3)--(7);

			\end{tikzpicture}
			\end{array}$}

& $\widetilde{\text{E}}_7$ 
\\
\hline
II$^*$ & {\small  $\begin{matrix}\\  \geq 4\\  \\    \end{matrix}$} & {\small  $5$} & {\small  $10$} & {\small  $0$ }&
{\small
$
\begin{pmatrix}
0& -1\\
1 & 1
\end{pmatrix}
$
}
&
 \scalebox{.7}{
 $\begin{array}{c}
 \begin{tikzpicture}
				\node[draw,circle,thick,scale=1,label=below:{}] (0) at (0,0){ 1};
				\node[draw,circle,thick,scale=1,label=below:{}] (1) at (1*1,0){ 2};
				\node[draw,circle,thick,scale=1,label=below:{}] (2) at (1*2,0){ 3};
				\node[draw,circle,thick,scale=1,label=below:{}] (3) at (1*3,0){ 4};
				\node[draw,circle,thick,scale=1,label=below:{}] (4) at (1*4,0){ 5};
								\node[draw,circle,thick,scale=1,label=below:{}] (5) at (1*5,0){ 6};
												\node[draw,circle,thick,scale=1,label=below:{}] (6) at (1*6,0){ 4};
																\node[draw,circle,thick,scale=1,label=below:{}] (8) at (1*5,1*1){ 3};
																\node[draw,circle,thick,scale=1,label=below:{}] (7) at (1*7,0){ 2};
								
				\draw[thick] (0)--(1)--(2)--(3)--(4)--(5)--(6)--(7);
				\draw[thick] (5)--(8);

			\end{tikzpicture}
			\end{array}$}

& $\widetilde{\text{E}}_8$ 
\\
\hline
\end{tabular}}
\end{center}
\caption{ {  { Kodaira-N\' eron  classification of  geometric fibers over codimension one points of the base of an elliptic fibration \cite{Kodaira.IIandIII, Neron}.  The $j$-invariant of the  I$_0^*$ is never $\infty$ and can take any finite value.
} 
 }
}\label{Table.KodairaTate}
\end{table}
\clearpage

\section{Euler Characteristic of Elliptic Fibrations}
\label{sec:Euler}

 The Euler characteristic of a smooth 
Weierstrass model $\varphi: Y\longrightarrow B$ over a base $B$ 
 is given by the following formula\cite{AE1,AE2} 
 $$
 \chi(Y)=\int\frac{12L}{1+ 6L} c(B),
 $$
 where  $c(B)=c(TB)\cap [B]$  is the total  homological Chern class and $L=c_1(\mathscr{L})$ is the first Chern class of the fundamental line bundle  $\mathscr{L}=(R^1\varphi_*\mathscr{O}_Y)^{-1}$ of the elliptic fibration. This expression is the generating function for the Euler characteristic.  
  Assigning weight $n$ to the $n$th Chern class,  the Euler characteristic of $Y$ is  the component of weight $ d = \text{dim} B$.
  A direct expansion gives 
 $$
 \chi(Y) =-2\sum_{i=1}^d (-6 L)^{i} c_{d-i}(TB)\cap[B].
 $$
The Euler characteristic of an elliptic surface is given by  Kodaira's formula \cite[III, Theorem 12.2, p. 14]{Kodaira.IIandIII}: 
 $$
 \chi(Y)= \sum_i v(\Delta_i),
 $$
 where the discriminant   $\Delta=\sum_i \Delta_i $ is a sum of points $\Delta_i$ and $v(\Delta_i)$ denotes the valuation of $\Delta_i$. 
 In particular, the Euler characteristic of the resolution of a Weierstrass model over a curve is always $12 \int L$:
 $$
 \chi(Y)=\int 12 L.
 $$
 
 There are several different ways to compute the Euler characteristic of an elliptic fibration.  The Euler characteristic (with compact support) is multiplicative on local trivial fibrations  and satisfies the excision property ($\chi(X/Z)=\chi(X)-\chi(Z)$ for any closed $Z\subset X$); moreover,  if $\phi:M\to N$ is a smooth proper morphism, then $\chi(M)=\chi(N)\chi(N_\eta)$ where $\chi(N_\eta)$ is the Euler characteristic of the generic fiber.   
It follows from these properties that the Euler characteristic of an elliptic fibrations gets all its contribution from its discriminant locus since the Euler characteristic of a smooth elliptic curve is zero. 
 One can identify a partition of the discriminant locus by subvarieties $V_i$ over which the generic fiber is constant. The Euler characteristic is then 
 $$
 \chi(Y)=\sum_i \chi(V_i) \chi(Y_{\eta_i}),
 $$
where $Y_{\eta_i}$ is the fiber over the generic point  $\eta_i$ of  $V_i$. This method increases quickly in complexity when the fiber structure becomes more involved \cite{GM1}.

A more effective way to compute the Euler characteristic is to use the   Poincar\'e--Hopf theorem, which asserts that the Euler characteristic  of $X$ equals  the  degree of the top Chern class of the tangent bundle $TX$ evaluated on the homological class of the 
variety. 
In other words,  the Euler characteristic is the degree of the total homological Chern class: 
$$
\chi(X)=\int c(X), \quad c(X):= c(TX)\cap (X).
$$
This method is explained in Section \ref{Sec:Batyrev} and can also be thought of  as an algebraic version of the Chern--Gauss--Bonnet theorem. 
 We give three different proofs  in Appendix \ref{appB.Euler}. 

\subsection{Crepant resolutions and flops} 

Let $X$ be a projective variety with at worst canonical Gorenstein singularities. 
We denote the canonical class by $K_X$. 
\begin{defn}
A birational projective  morphism $\rho:Y\longrightarrow X$ is called a \emph{crepant desingularization} of $X$ if $Y$ is smooth and 
$K_Y=\rho^* K_X$. 
\end{defn}

\begin{defn}
A resolution of singularities of a variety $Y$ is a proper surjective birational morphism $\varphi:\widetilde{Y}\longrightarrow Y$  such that  
$\widetilde{Y}$ is nonsingular
and  $\varphi$ is an isomorphism away  from the singular  locus of $Y$. In other words, $\widetilde{Y}$ is nonsingular and  if $U$ is the singular locus of $Y$, $\varphi$ maps $\varphi^{-1}(Y\setminus U)$  isomorphically  onto $Y\setminus U$.  
 A \emph{crepant resolution of singularities}  is a resolution of singularities such that  $K_Y=f^* K_X$. 
\end{defn}

\begin{rem}
In dimension two, there is one and only one crepant resolution of a variety with canonical singularities.
In dimension three, crepant resolutions of Gorenstein singularities always exist but are usually not unique. 
In dimension four or greater, crepant resolutions are not always possible. However, one can always find a crepant birational morphism from a $\mathbb{Q}$-factorial variety with terminal singularities. 
\end{rem}

\begin{defn}[$D$-flop {(See \cite[p. 156-157]{Matsuki})}]
Let $f_1: X_1\longrightarrow X$ a small contraction. 
Let $D$ be a $\mathbb{Q}$-Cartier divisor in $X_1$. 
A \emph{$D$-flop} is a birational morphism $f:X_1--\rightarrow X_2$ fitting into a triangular diagram where $f_1$ and $f_2$ are birational morphisms 
$$
\begin{tikzcd}[column sep=huge] 
  X_1  \arrow[rightarrow,dashed]{rr} {f}  \arrow[rightarrow]{rd} [below=.1cm]{\displaystyle f_1}  & & X_2  \arrow[rightarrow]{ld} [below=.1cm]{\displaystyle f_2}  \\
  &  X & 
  \end{tikzcd}
  $$
  such that  
\begin{enumerate}
\item $X_i$ are normal varieties with at worst terminal singularities. 
\item $f_i$ are small contractions (i.e. their exceptional loci are in codimension two or higher). 
\item $K_{X_i}$ is numerically trivial along the fibers of $f_i$ (i.e. $K_{X_i} \cdot \ell=0$ for any curve $\ell$ contracted by $f_i$).
\item The $\mathbb{Q}$-divisor $-D$ is $f_1$-ample.
\item  The strict $f$-transform $D^+$ of $D$ is $f_2$-ample.
\end{enumerate}
\end{defn}

\begin{defn}[flop]
The morphism $f_2:X_2\longrightarrow X$ is said to be a {\em flop} of $f_1:X_1\longrightarrow X$ if there exists a divisor $D\subset X_1$ such that  $f_2$ is  a $D$-flop of $f_1$. 
\end{defn}
\subsection{Batyrev's theorem and the Chern class of a crepant resolution} \label{Sec:Batyrev}

We denote the Chow ring of a nonsingular variety $X$ by $A_*(X)$. The free group of generated by subvarieties of dimension $r$ modulo rational equivalence is denoted by  $A_r(X)$. 
The degree of a class $\alpha$ of $A_*(X)$ is denoted by  $\int_X \alpha$ (or simply $\int \alpha$ if there is no ambiguity in the choice of $X$), and is defined to be the degree of its component in $A_0(X)$.
The total homological Chern class $c(X)$ of any nonsingular variety $X$ of dimension $d$ is defined by:
\begin{equation*}
c(X)=c(TX)\cap [X],
\end{equation*}
where $TX$ is the tangent bundle of $X$ and $[X]$ is the class of $X$ in the Chow ring. The degree of $c(X)$ is the topological Euler characteristic of $X$: 
\begin{equation*}
\chi(X)=\int_X c(X).
\end{equation*}
Motivated by string geometry, 
Batyrev and Dais proposed in  \cite[Conjecture 1.3]{Batyrev.Dais} the following conjecture.
\begin{conj}[Batyrev and Dais{, see \cite{Batyrev.Dais}}]
Hodge numbers of smooth crepant resolutions of an algebraic variety defined over the complex numbers with at worse Gorenstein canonical singularities  do not depend on the choice of such a resolution.
\end{conj}

Using $p$-adic integration and the Weil conjecture, Batyrev proved the following slightly weaker proposition: 
 
\begin{thm}[Batyrev, \cite{Batyrev.Betti}]
\label{thm:Batyrev}
Let $X$ and $Y$ be irreducible birational smooth $n$-dimensional projective algebraic varieties 
over $\mathbb{C}$. Assume that there exists a birational rational map $\varphi: X -\,  - \rightarrow Y$ which does not 
change the canonical class. Then $X$ and $Y$ have the same Betti numbers. 
\end{thm}
 Batyrev's result was strongly inspired by string dualities, in particular by the work of Dixon, Harvey, Vafa, and Witten \cite{Dixon:1986jc}. 
 Kontsevitch proved the Batyrev--Dais conjecture for the special case of Calabi-Yau varieties as a corollary of his newly invented theory of motivic integration; the proof relies on Hodge theory and geometrizes Batyrev's use of $p$-adic integration. 
\begin{thm}[Kontsevitch, \cite{Kontsevich.Orsay}]
\label{thm:Kontsevitch}
Let $X$ and $Y$ be birationally-equivalent smooth Calabi-Yau varieties. Then $X$ and $Y$ have the same Hodge numbers.
\end{thm}
As a direct consequence of Batyrev's theorem, the Euler characteristic of a crepant resolution of a variety with Gorenstein canonical singularities is independent on the choice of resolution. 
We identify the Euler characteristic as the degree (see Definition \ref{def:degree}) of the total  (homological) Chern class of a crepant resolution $f: \widetilde{Y }\longrightarrow Y$ of a Weierstrass model $Y\longrightarrow B$:
$$
\chi(\widetilde{Y})=\int c(\widetilde{Y}).
$$
We then use the birational invariance of the degree under the pushfoward to express the Euler characteristic as a class in the Chow ring of the projective bundle $X_0$. We subsequently push this class forward to the base to obtain a rational function depending upon only the total Chern class of the base $c(B)$, the first Chern class $c_1(\mathscr L)$, and the class $S$ of the divisor in $B$:
$$
\chi(\widetilde{Y})=\int_B \pi_* f_* c(\widetilde{Y}).
$$
In view of Theorem \ref{thm:Batyrev}, this Euler characteristic is independent of the choice of a crepant resolution. We discuss pushforwards and their role in the computation of the Euler characteristic in more detail in Section \ref{sec:push}.

\section{Pushforwards and Computing the Euler Characteristic}
\label{sec:push}
\begin{defn}[Pushforward,  {\cite[Chap. 1, p. 11]{Fulton.Intersection}}] 
 Let $f: X\longrightarrow Y$ be a proper morphism. 
Let $V$ be a subvariety of $X$, the image $W=f(V)$ a subvariety of $Y$, and the function field $R(V)$ an extension of the function field $R(W)$. 
The pushforward $f_* : A_*(X)\to A_*(Y)$ is defined as follows
$$
f_* [V]= \begin{cases}
0  &  \text{if} \quad \dim V\neq \dim W,\\
[R(V):R(W)] \   [W]  & \text{if} \quad \dim V= \dim W,
\end{cases}
$$
where $[R(V):R(W)]$ is the degree of the field extension $R(V)/R(W)$. 
\end{defn}
\begin{lem}[{\cite[Chap. 1, p. 13]{Fulton.Intersection}}]  \label{lem:Push} Let $f:X\longrightarrow Y$  be a proper map between varieties. 
 For any class $\alpha$ in the Chow ring $A_*(X)$ of $X$:
$$\int_X \alpha=\int_Y f_* \alpha.$$
\end{lem}
 Lemma \ref{lem:Push} means that an intersection number in $X$ can be computed in $Y$ through a pushforward. 
 This simple fact  has far-reaching consequences and  characterizes the point of view taken in this paper, as it allows us to express the topological invariants of an elliptic fibration in terms of those of the base. 
\label{Sec:Pushforward}
\subsection{The pushforward theorem}
 A formula for the Chern classes of blowups of a smooth variety along a smooth center was conjectured by Todd and Segre and proven in the general case by 
Porteous \cite{Porteous} using the Riemann-Roch theorem. A proof using  Riemann-Roch ``without denominators'' is presented in   \S 15.4 of \cite{Fulton.Intersection}. 
A proof without Riemann-Roch was derived by Lascu and Scott \cite{Lascu.Scott,Lascu1978}.  A generalization of the formula to potentially singular varieties was obtained by Aluffi \cite{Aluffi_CBU}.

The blowup formula simplifies dramatically when the center of the blowup is a nonsingular complete intersection of nonsingular hypersurfaces meeting transversally. 
Aluffi gives an elegant short proof using functorial properties of Chern classes and Chern classes of bundles of tangent fields with logarithmic zeros:
\begin{thm}[Aluffi, {
\cite[Lemma 1.3]{Aluffi_CBU}}]
\label{Thm:AluffiCBU}
Let $Z\subset X$ be the  complete intersection  of $d$ nonsingular hypersurfaces $Z_1$, \ldots, $Z_d$ meeting transversally in $X$.  Let  $f: \widetilde{X}\longrightarrow X$ be the blowup of $X$ centered at $Z$. We denote the exceptional divisor of $f$  by $E$. The total Chern class of $\widetilde{X}$ is then:
\begin{equation}
c( T{\widetilde{X}})=(1+E) \left(\prod_{i=1}^d  \frac{1+f^* Z_i-E}{1+ f^* Z_i}\right)  f^* c(TX).
\end{equation}
\end{thm}

\begin{lem}
\label{lem:symhom}
Let  $f: \widetilde{X}\longrightarrow X$ be the blowup of $X$ centered at $Z$. We denote the exceptional divisor of $f$  by $E$. Then 
\begin{equation}
 f_* E^n=
 (-1)^{d+1} h_{n-d} (Z_1, \cdots, Z_d) Z_1\cdots Z_d,\nonumber
\end{equation}
where $h_i(x_1, \cdots, x_k)$ is the complete homogeneous symmetric polynomial of degree $i$ in $(x_1, \cdots, x_k)$ with the convention that $h_i$ is identically zero for $i<0$ and $h_0=1$.  
\end{lem}
\begin{proof}
The exceptional locus of the blowup  of $X$ centered at $Z$ is the projective bundle $\mathbb{P}(N_Z X)$. Let $E=c_1(\mathscr{O}_{\mathbb{P}(N_Z X)}(1))$.  By the  functoriality  properties of Segre classes, we have: 
\begin{equation}\label{Segre}
f_* \frac{1}{1+E}\cap [E] =\frac{1}{c(N_Z X)}\cap[Z]= \prod_{i=1}^d\frac{Z_i}{ 1+ Z_i} ,
\end{equation}
where $N_Z X$ is the normal bundle of $Z$ in $X$. The generating function of complete homogeneous symmetric polynomials in $(x_1, \dots, x_d)$ is $\prod_{\ell=1}^{d}(1-x_\ell t)^{-1}$:
\begin{equation}
\sum_{n=1}^\infty h_n(x_1, \cdots, x_d) t^n = \prod_{\ell=1}^{d} \frac{1}{1-x_\ell t}.\nonumber
\end{equation}
By matching terms of the same dimensions in equation \eqref{Segre}, we can compute $f_* E^n$ in terms of  complete homogeneous symmetric polynomials $h_i(Z_1, \ldots, Z_d)$ in the classes $Z_i$: 
\begin{equation}
 f_* E^n=(-1)^{n-1}[t^n] \left( \prod_{i=1}^d\frac{ tZ_i}{ 1+ t Z_i}\right)=
 (-1)^{d+1} h_{n-d} (Z_1 , \cdots, Z_d) Z_1\cdots Z_d,\nonumber
\end{equation}
where $[t^n] g(t)=g_n$ for a formal series $g(t)=\sum_{i=0}^\infty g_i t^i$  and $h_i$ is identically zero for $i<0$ and $h_0=1$.
\end{proof}
\begin{exmp}
If  $d=2$, we have 
\begin{align}\nonumber
f_* E=0, \quad f_* E^2=-Z_1 Z_2, \quad f_*E^3= -(Z_1 +Z_2)Z_1 Z_2, \quad f_* E^4=-(Z_1^2+Z_2^2+ Z_1 Z_2) Z_1 Z_2.
\end{align}
\end{exmp}

\begin{exmp}
If  $d=3$, we have 
\begin{align}\nonumber
f_* E=0, \quad f_* E^2=0, \quad f_*E^3= Z_1 Z_2 Z_3, \quad f_* E^4=(Z_1+Z_2) Z_1 Z_2 Z_3.
\end{align}
\end{exmp}
 A direct consequence of Theorem \ref{Thm.Jacobi} (Jacobi's identity) and Lemma \ref{lem:symhom} is the following pushforward formula (see \cite{FH2}): 
\begin{lem}\label{lem:PushE}
 Let $Z\subset X$ be the  complete intersection  of $d$ nonsingular hypersurfaces $Z_1$, \ldots, $Z_d$ meeting transversally in $X$.  Let  $f: \widetilde{X}\longrightarrow X$ be the blowup of $X$ centered at $Z$ with exceptional divisor $E$. Then  for any integer $n\geq 0$: 
 \begin{align}\nonumber
 f_* E^n = \sum_{\ell=1}^d Z_\ell^{n} M_\ell, \quad M_\ell=\prod_{\substack{m=1\\
 m\neq \ell}} \frac{Z_m }{ Z_m -Z_\ell}.
\end{align}
 The coefficient $M_\ell$ is the $\ell$-moment of the blowup $f$ defined after Theorem \ref{Thm:Push}.
\end{lem}
\begin{proof}
\begin{align}
 f_* E^n  & = 
 (-1)^{d+1} h_{n-d} (Z_1 , \cdots, Z_d) Z_1\cdots Z_d  \nonumber  &&\quad \text{(by Lemma \ref{lem:symhom})} \\
&=  (-1)^{d+1}\sum_{\ell=1}^d Z_\ell^{n-1}\Big( \prod^d_{\substack{m=1\\
 m\neq \ell}} \frac{1}{ Z_\ell -Z_m } \Big) Z_1\cdots Z_d  &&\quad\text{ (by Lemma \ref{Jacobi})} \nonumber\\
 &=  (-1)^{d+1}\sum_{\ell=1}^d Z_\ell^{n} \Big( \prod^d_{\substack{m=1\\
 m\neq \ell}} \frac{ Z_m}{ Z_\ell -Z_m } \Big) \nonumber    &&\quad\text{ (by the identity  $Z_1\cdots Z_d=Z_\ell \prod^d_{\substack{m=1\\
 m\neq \ell}}  Z_m$)}  \\
 &=  \sum_{\ell=1}^d Z_\ell^{n} \Big( \prod^d_{\substack{m=1\\
 m\neq \ell}} \frac{ Z_m}{ Z_m -Z_\ell } \Big) \nonumber &&\quad\text{ (since $\prod^d_{\substack{m=1\\
 m\neq \ell}}  \frac{ Z_m}{ Z_\ell -Z_m }=(-1)^{d-1}\prod^d_{\substack{m=1\\
 m\neq \ell}}  \frac{ Z_m}{ Z_m -Z_\ell }$)}
\end{align}
\end{proof}
To compute topological invariants of a blowup, we often have to pushforward analytic expressions of $E$. Let $\widetilde{Q}(t)=\sum_a f^* Q_a t^a$ be a formal power series with $Q_a\in A_*(X)$. 
 The formal series $Q(E)$ is a well-defined element of $A_*(\widetilde{X})$.  
We recall Theorem \ref{Thm:Push}:

\Push*

\begin{proof}

 \begin{align}
  f_*\widetilde{Q}(E) & =f_* \sum_a (f^* Q_a) E^a = \sum_a Q_a f_* E^a = \sum_a Q_a  \sum_{\ell=1}^d Z_\ell^{a} M_\ell = \sum_{a}  \sum_{\ell=1}^d Q_a  Z_\ell^{a} M_\ell  =  \sum_{\ell=1}^d  {Q}(Z_\ell) M_\ell.
 \end{align}
  \end{proof}
 
\subsection{Classes of the blowup centers of crepant resolutions}
\label{sec:centers}

We denote the projective bundle of the Weierstrass model to be $X_0 = \mathbb{P}[\mathscr{O}_B\oplus \mathscr{L}^{\otimes 2} \oplus \mathscr{L}^{\otimes 3}]$ and
the elliptic fibration $\varphi: Y_0 \rightarrow B$ to be the zero-scheme of a section of $\mathscr{O}(3) \otimes \pi^* \mathscr{L}^{\otimes 6}$. We denote by $\mathscr{O}(1)$ the dual of the tautological line bundle of $X_0$. We denote by $H$ the first Chern class of 
$\mathscr{O}(1)$, and by $L$ the first Chern class of $\mathscr{L}$. The elliptic fibration $\varphi: Y_0\longrightarrow B$  is of class $[Y_0]=3H+6\pi^* L$. The classes of the generators of the blowup centers are $Z_i^{(n)}$, where $n$ is the number of the blowup map and $i$ is the number of the center. For example, consider the following blowup: 
\begin{align}
\label{eqn:blowupexample}
\begin{array}{c}
\begin{tikzpicture}
	\node(X0) at (0,0){$X_0$};
	\node(X1) at (2.5,0){$X_1$};
	\node(X2) at (5,0){$X_2$};
	\draw[big arrow] (X1) -- node[above,midway]{$(x,y,s|e_1)$} (X0);	
	\draw[big arrow] (X2) -- node[above,midway]{$(y,e_1|e_2)$} (X1);		
 \end{tikzpicture}
 \end{array}
\end{align}
where  each arrow above denotes a blowup, 
 $V(s)$ is a smooth divisor in $X$,
and where $E_n =  V(e_n)$ is the exceptional divisor of the $n$th blowup. The first exceptional divisor is a projective bundle whose fibers have projective coordinates $[x':y':s']$, where
\begin{equation}\nonumber
x=x' e_1~,~ y=y' e_1~,~ s=s' e_1.
\end{equation}
For notational convenience, we drop the prime superscripts (${}^\prime$) appearing after each blowup.

The classes associated to the center of the first blowup in \eqref{eqn:blowupexample} are:
\begin{equation}\nonumber
Z_1^{(1)}=[x]=H+2\pi^*L, \quad Z_2^{(1)}=[y]=H+3\pi^* L, \quad Z_3^{(1)}=[s]=\pi^*S.
\end{equation}
Likewise, the classes associated to the center of the second blowup are
\begin{equation}\nonumber
Z_1^{(2)}=[y]=f_1^*(H+3\pi^* L)-E_1, \quad Z_2^{(2)}=[e_1]=E_1.
\end{equation}
Let us adapt the above data into a matrix-inspired notation, such that $i$ denote columns and $n$ denotes rows. This notation allows us to read the classes of the blowup center by each row. In this notation, the above results can be expressed as follows:
\begin{equation}\nonumber
Z=
\begin{pmatrix}
Z_1^{(1)} & Z_2^{(1)} & Z_3^{(1)} \\
Z_1^{(2)} & Z_2^{(2)} & 
\end{pmatrix}
=
\begin{pmatrix}
H+2\pi^* L & H+3\pi^* L & \pi^*S \\
f_1^*(H+3\pi^* L)-E_1 & E_1 & 
\end{pmatrix}.
\end{equation}
See Table \ref{Table:BlowCenterZ} for an exhaustive list of the generator classes associated to the blowup centers of the crepant resolutions in Table \ref{tab:blowupcenters}. Note that we streamline our notation by omitting the explicit pullback maps from the expressions for the classes appearing in these tables.

\section{Hodge Numbers of Elliptically Fibered Calabi-Yau Threefolds}\label{sec:Hodge}

 Using motivitic integration, Kontsevich shows in his famous ``String Cohomology'' Lecture at Orsay that birational equivalent Calabi-Yau varieties have the same class in the completed Grothendieck ring \cite{Kontsevich.Orsay}. 
Hence, birational equivalent Calabi-Yau varieties have the same  Hodge-Deligne polynomial, Hodge numbers, and Euler characteristic. 
 In this section, we compute the Hodge numbers of crepant resolutions of Weierstrass models in the case of  Calabi-Yau threefolds.

\begin{thm}[Kontsevich, (see \cite{Kontsevich.Orsay})]
Let $X$ and $Y$ be birational equivalent Calabi-Yau varieties over the complex numbers. Then $X$ and $Y$ have the same Hodge numbers. 
\end{thm}
\begin{rem}
In Kontsevich's theorem, a Calabi-Yau variety is a nonsingular complete projective variety of dimension $d$ with a trivial canonical divisor. 
To compute Hodge numbers in this section, we use the following stronger definition of a Calabi-Yau variety.
\end{rem}

\begin{defn}\label{defn:CY}
A \emph{Calabi-Yau variety} is a smooth compact projective variety $Y$ of dimension $n$ with a trivial canonical class and such that  $H^i(Y,\mathscr{O}_X)=0$ for $1\leq i\leq n-1$.
\end{defn}

 We first recall some basic definitions and relevant classical theorems.  

\begin{defn}
The \emph{N\'eron-Severi group $\mathrm{NS}(X)$} of a variety $X$ is the group of divisors of $X$ modulo algebraic equivalence. 
The rank of the N\'eron-Severi group of $X$ is called the \emph{Picard number} and is denoted $\rho(X)$. 
\end{defn}
\begin{thm}[Lefschetz $(1,1)$-{theorem}, {see \cite[Theorem 7.2, p. 157]{Voisin.Hodge} }]
If $X$ is compact K\"ahler manifold, then the map  $c_1: Pic(X)\to H^{1,1}(X,\mathbb{Z})=H^{1,1}(X,\mathbb{C})\cap H^2(X,\mathbb{Z})$ is   well-defined and surjective. 
In addition,  the Picard number $\rho(X)$  is equal to the Hodge number $h^{1,1}(X):=\mathrm{dim}\  H^{1,1}(X,\mathbb{Z})$. 
\end{thm} 
\begin{thm}[Noether's formula]
 If $B$ is a smooth compact, connected, complex surface with canonical class $K_B$ and Euler number  $c_2$:
$$
\chi(\mathscr{O}_B)=1-h^{0,1}(B)+h^{0,2}(B),\quad    \chi(\mathscr{O}_B)=\frac{1}{12}(K^2+c_2).
$$
\end{thm}
When $B$ is a smooth compact rational surface, we have a simple  expression of   $h^{1,1}(B)$ as a function of $K^2$ using the following lemma. 
\begin{lem}\label{lem:NoetherRational}
Let $B$ be a smooth compact rational surface with  canonical class $K$. Then 
\begin{equation}
h^{1,1}(B)=10-K^2.
\end{equation}
\end{lem}
\begin{proof}
Since $B$ is a rational surface, $h^{0,1}(B)=h^{0,2}(B)=0$. Hence $c_2=2+h^{1,1}(B)$ and the  lemma follows from  Noether's formula. 
\end{proof}

We now compute  $h^{1,1}(Y)$ using the Shioda-Tate-Wazir theorem. 
\begin{thm}[{Shioda--Tate--Wazir; see  \cite[Corollary 4.1]{Wazir}}]\label{Thm:STW}
Let $\varphi:Y\rightarrow B$ be a smooth elliptic fibration, then 
$$
\rho(Y)=\rho(B)+f+\mathrm{rank}(\text{MW} (\varphi))+1
$$
where $f$ is the number of geometrically irreducible fibral divisors not touching the zero section.
\end{thm}
\begin{thm}\label{Thm:STW2} Let $Y$ be a smooth Calabi-Yau threefold  elliptically fibered over a smooth variety $B$ with Mordell-Weil group of rank zero. Then,
\begin{equation}\nonumber
h^{1,1}(Y)=h^{1,1}(B)+f+1, \quad h^{2,1}(Y)=h^{1,1}(Y)-\frac{1}{2}\chi(Y),
\end{equation}
where $f$ is the number of geometrically irreducible fibral divisors not touching the zero section. In particular, if $Y$ is a $G$-model with $G$ a simple group, $f$ is the rank of $G$. 
\end{thm}
\begin{proof}
In the statement of the Shioda--Tate--Wazir theorem, we can replace  the Picard numbers $\rho(Y)$ and $\rho(B)$ by the Hodge numbers $h^{1,1}(Y)$ and $h^{1,1}(B)$ using  Lefschetz's (1,1)-theorem. That gives 
$h^{1,1}(Y)=h^{1,1}(B)+f+1$.
Since the Euler characteristic of a Calabi-Yau threefold is 
$
\chi(Y)=2(h^{1,1}-h^{2,1}),
$ 
and assuming that both $\chi(Y)$ and $h^{1,1}(Y)$ are known, it follows that 
$
h^{2,1}(Y)=h^{1,1}(Y)-\frac{1}{2}\chi(Y).
$
 \end{proof}
\begin{rem}
For $G$-models with $G$ a simple group, $f$ will be the rank of $G$ \cite[\S 4]{Morrison:1996pp}. 
\end{rem}

 \section{An Illustrative Example: SU(2)-Models}
 \label{sec:example}

In this section, we discuss in detail the computation of the Euler characteristic of SU(2)-models. Note that the results presented in this section are equivalent for each of the four possible Kodaira fibers (namely, types I$_2^\text{s}$, I$_2^\text{s}$, I$_3^\text{s}$,  
 III, IV$^\text{ns}$) realizing an SU(2)-model; see Section \ref{sec:results} for a list of the Weierstrass equations defining the various SU(2)-models. We find 
\begin{align}
c(X_0)&=(1+H)(1+H+3\pi^* L)(1+H+2\pi^* L) c(B)\nonumber\\
c(Y_0)&=(3H+6\pi^* L)\frac{c(X_0)}{1+3H+6\pi^* L}.\nonumber
\end{align}
The singular elliptic fibration is resolved by a unique blowup with center $(x,y,s)$ \cite{ESY1}. We denote the blowup by  $f: X_1 \longrightarrow X_0$ and the exceptional divisor by  $E_1$.
The center is a complete intersection of hypersurfaces $V(x)$, $V(y)$, and $V(s)$, whose classes are respectively 
\begin{equation}\nonumber
 Z_1= 2\pi^*L+H, \quad Z_2= 3\pi^* L+H, \quad Z_3=\pi^* S.
\end{equation}
The proper transform  of the elliptic fibration $Y_0$ is denoted $Y$, and is obtained from the total transform of $Y$ by removing $2E_1$. It follows that the class of $Y$ in $X_1$ is 
$$
[Y]=[f^*(3H+6 \pi^* L)-2 E_1]\cap[X_1]
$$
Moreover, we have the following Chern classes:
\begin{align} \nonumber
c(TX_1)&=(1+E_1)\frac{(1+f^* Z_1-E_1)(1+f^* Z_2 -E_1)(1+f^* Z_3-E_1)}{(1+f^* Z_1)(1+f^* Z_2)(1+f^* Z_3)} f^* c(TX_0)\\
c(TY)&=\frac{(1+E_1)(1+f^* Z_1-E_1)(1+f^* Z_2 -E_1)(1+f^* Z_3-E_1)}{(1+3H+6L-2E_1)(1+f^* Z_1)(1+f^* Z_2)(1+f^* Z_3)} f^* c(TX_0)\nonumber
\end{align}
By an expansion of $c(TY)$ in first order, we can easily check that the resolution is crepant: 
$$
c(TY)= f^* c(TY_0).
$$
After the blowup, the homological total Chern class is $c(Y)=c(TY)\cap [Y]$:
\begin{equation}
c(Y)=
(3f^* H+6f^* \pi^* L-2E_1)
(1+E_1)\frac{(1+f^* Z_1-E_1)(1+f^* Z_2 -E_1)(1+f^* Z_3-E_1)}{(1+f^* Z_1)(1+f^* Z_2)(1+f^* Z_3)} f^*c(X_0). \nonumber
\end{equation}
To compute the Euler characteristic, we have to evaluate 
$$
\chi(Y)= \int_{Y} c(Y).
$$
The first pushforward requires the following data: 
\begin{equation}
M_1= \frac{Z_2 Z_3}{(Z_2-Z_1)(Z_3-Z_1)}, \quad 
M_2= \frac{Z_1 Z_3}{(Z_1-Z_2)(Z_3-Z_2)}, \quad 
M_3= \frac{Z_1 Z_2}{(Z_1-Z_3)(Z_2-Z_3)}.\nonumber
\end{equation}
Applying the pushforward theorem is now a purely algebraic routine that can be  easily implemented in one's favorite algebraic software. 
 Using Theorem \ref{Thm:Push}, we pushforward $c(Y)$ from the Chow ring $A_*(X_1)$ to the Chow ring $A_*(X_0)$. 
Using Theorem \ref{Thm:PushH}, we then pushforward $f_* c(Y)$ to the Chow ring of the base. 
 When the dust settles, we find an expression of $\chi(Y)$ in the Chow ring of the base: 
\begin{equation}
\chi(Y)=\int_{Y} c(TY)=\int_{X_0} f_* c(TY)=  \int_B  \pi_* f_* c(TY)=\int_B   6\frac{ 2L+3 L S  - S^2}{(1+S) (1+6 L - 2 S )} c(TB).
\label{eqn:genfcn}\nonumber
\end{equation}
Concretely, we replace $c(TB)$ by the Chern polynomial $c_t(TB)=1+c_1t + c_2t^2 + c_3 t^3+\cdots$, $L$ by $ Lt$, and $S$ by $ St$; if $d$ is the dimension of $B$, the Euler characteristic of $Y$ is given by the coefficient of $t^d$ in the Taylor expansion centered at $t=0$ of the generating function: 
\begin{align}
\begin{aligned}
\chi(Y) & =   6\frac{ 2Lt+3 L St^2  - S^2t^2}{(1+S t) (1+6 L t- 2 St )} c_t(TB)\\
&=
12 L t + 6 t^2 (2 c_1 L - 12 L^2 + 5 L S - S^2) + \\
& \  \    +  6 t^3   (-12 c_1 L^2+5 c_1 L S-c_1 S^2+2 c_2 L+72 L^3-54 L^2 S+15 L S^2-S^3)+\cdots
    \end{aligned}\nonumber
\end{align}
\begin{thm}
If $B$ is a curve, the Euler characteristic of an SU(2) model is $12 L$. If $B$ is a surface, the Euler characteristic is $6 (2 c_1 L - 12 L^2 + 5 L S - S^2)$. 
If $B$ is a threefold, the Euler characteristic is   $6 (-12 c_1 L^2+5 c_1 L S-c_1 S^2+2 c_2 L+72 L^3-54 L^2 S+15 L S^2-S^3)$.
\end{thm}
In order to consider the Calabi-Yau case, we set $L=c_1(TB)$ in the above expression, which gives 
$$
\chi(Y)=
12 c_1 t - 6 t^2 (10 c_1^2 - 5 c_1 S + S^2) + 
 6 t^3 (60 c_1^3 - 49 c_1^2 S + 2 c_1 c_2 + 14 c_1 S^2 - S^3)+\cdots
$$
Note that we retrieve the result for a smooth Weierstrass model if we further impose $S=0$. 

 \begin{rem}
As a byproduct of the computation of the Euler characteristic of the resolution,  we can also easily evaluate the contribution from the singularities to be 
$$
 6\frac{ 2L+3 L S  - S^2}{(1+S) (1+6 L - 2 S )} c(TB)-\frac{12 L}{1+ 6 L} c(TB)=6 \frac{(5L+6 L^2  -2 L S  -S ) S}{(1+6 L ) (1+6 L-2 S ) (1+S )}  c(TB),
$$
which can be rewritten as 
$$
\chi(Y)-\chi(Y_0)=6 \frac{6 L^2  -2 L S+5 L  -S }{(1+6 L ) (1+6 L-2 S )}  c(S), \quad c(S)= \frac{S}{1+S} c(TB)\cap [B]. 
$$
In the Calabi-Yau case $L=c_1(TB)$, the above quantity usually has a physical meaning. 
For example, if $Y$ is a Calabi-Yau fourfold, this expression reduces to 
$-6S(7 c_1-S)^2\cap [B]$, 
which is  the contribution of branes to the Euler characteristic.  
In another limit, the above expression can be understood as the contribution of the $G_4$-flux in M-theory to the M2-brane flux or brane flux in type IIB string theory: 
$$
\frac{1}{2}\int_{Y_0} G_4\wedge G_4=\frac{1}{2} \int_S F\wedge F=-6 \int_S (7 c_1-S)^2.
$$
\end{rem}
\section{Tables of Results}
\label{sec:results}
The $G$-models studied in this paper are all realized as crepant resolutions of the singular Weierstrass model 
	\begin{align}	
 y^2z+ a_1 xy z + a_3  yz^2 -(x^3+ a_2 x^2 z + a_4 x z^2 + a_6 z^3)=0,\nonumber
	\end{align}
where the desired singularity structures corresponding to the decorated Kodaira fibers can be specified by the valuation of the coefficients of the Weierstrass equation with respect to  the divisor $S=V(s)$. 
Following Tate's algorithm, we use the notation $a_{i,p}= 	a_i / s^p $, where the valuations $p$ are the minimal values dictated by Tate's algorithm and we assume that the coefficients $a_{i,p}$ are generic. 
 
 We present the results of our computation of the Euler characteristic generating functions for various $G$-models. The generating functions are the pushforwards of the homological total Chern class of the resolved Weierstrass model to the base $B$, and are expressed as rational functions of the classes $S$ and $L$ (where $L = c_1(\mathscr L)$ is the class of the fundamental line bundle and $S$ is the class of the divisor in the base $B$), multiplied by the total Chern class of the base, $c(B)$---see Table \ref{Table:Push}. Tables \ref{Table:Euler3}-\ref{table:CY4} specialize the results to (respectively) elliptic threefolds, fourfolds, and elliptic Calabi-Yau fourfolds, while Table \ref{Table.HodgeCY3} summarizes the Hodge numbers for Calabi-Yau threefold $G$-models.  

When  computing  Hodge numbers of a $G$ model which is a Calabi-Yau threefold, we recall that we  assume that the base is a rational surface. This is a direct consequence of  Definition \ref{defn:CY}.
 Moreover, for a $G$-model with $G$ a simple group, the integer $f$ that  enters in Theorem \ref{Thm:STW2} is the rank of $G$.

For the SO(3), SO(5), and SO(6)-models, the class $S$ is given by \cite{MP}:
\begin{align}
\begin{cases}
	S =  4 L ~~\text{for SO(3)} , \\
	S=2 L ~~ \text{for SO(5)},\\
	S=2L ~~ \text{for SO(6)}.
	\end{cases}
\nonumber
\end{align}	
 
  Below we list the various Weierstrass equations we use to compute  the $G$-models, labeled by their Kodaira fiber type and associated Lie group $G$. It is necessary to specify a crepant resolution in order to actually compute the total Chern class and Euler characteristic of a $G$-model. 
 There could be several distinct crepant resolutions for a $G$-model. However, Theorem \ref{thm:Kontsevitch} assures that the Euler characteristic is insensitive to the choice of crepant resolution and therefore we only need one crepant resolution to compute the Euler characteristic of a $G$-model defined by the crepant resolution of a Weierstrass model. The models associated to the groups SU($n$) and USp($2n$) are  \cite{Katz:2011qp}:
\begin{align}
	\label{eqn:SO}
&\text{I}_2^{\text{s}} &&\text{SU(2)}&:&\quad 	&&y^2z+a_1 x y z + a_{3,1} s y z =x^3 + a_{2,1} s  x^2z + a_{4,1}s x z^2+a_{6,2} s^2 z^3,\\
&\text{I}_{2n}^{\text{ns}} &&\text{USp($2n$)}&:&\quad 	&&y^2z =x^3 + a_2 x^2z + a_{4,n} s^n x z^2+a_{6,2n} s^{2n} z^3,\\
&\text{I}_{2n+1}^{\text{ns}} &&\text{USp($2n$)}&:&\quad 	&&y^2z =x^3 + a_2 x^2z + a_{4,n+1} s^{n+1} x z^2+ a_{6,2n+1} s^{2n+1} z^3,\\
&\text{I}_{2n}^{\text{s}} &&\text{SU($2n$)}&:&\quad 	&&y^2z+ a_1 x y z =x^3 + a_{2,1} s x^2z + a_{4,n} s^n x z^2+a_{6,2n} s^{2n} z^3,\\
&\text{I}_{2n+1}^{\text{s}} &&\text{SU($2n+1$)}&:&\quad 	&&y^2z+a_1 x y z+a_{3, n} s^n y z^2   =x^3 + a_{2,1}s x^2z + a_{4,n+1} s^{n+1} x z^2+a_{6,2n+1} s^{2n+1} z^3.
\end{align}
The Weierstrass models for SO(3), SO(5), and SO(6)  are discussed in \cite{MP}; these models require a Mordell-Weil group $\mathbb{Z}/2 \mathbb{Z}$. 
 The crepant resolutions  of the Weierstrass models for G$_2$, Spin(7), and Spin(8) models are studied in \cite{G2} and require a careful analysis of the Galois group of an associated polynomial. 
The Weierstrass equations defining these models  along with the remaining $G$-models, with $G$ one of the exceptional groups are given below \cite{G2,MP,Esole:2018vnm,Katz:2011qp}:
\begin{align}
&\text{I}_2^{\text{ns}} &&\text{SO(3)}&:&\quad 	&&y^2z =x( x^2 + a_2 xz + a_4  z^2),\\
&\text{I}_4^{\text{ns}}&&\text{SO(5)}&:&\quad &&y^2z  =( x^3 + a_2 x^2z + s^2 xz^2),\\
&\text{I}_4^\text{s}\quad &&\text{SO(6)}&:&\quad  &&y^2z + a_1  x yz = x^3 + ms x^2z + s^2 xz^2 ,\quad m\in\mathbb{C}, \quad m\neq -2, 0, 2,\\
&\text{I}_0^{*\text{ss}}\quad &&\text{Spin(7)}&:& \quad && y^2z= x^3 + a_{2,1} s x^2 z+ a_{4,2} s^2 x z^2+ a_{6,4} s^{4} z^3,\\
&\text{I}_0^{*\text{s}}\quad &&\text{Spin(8)}&:& \quad && y^2z= (x-x_1sz)(x-x_2 sz)(x-x_3 sz)+  s^2 r x^2z + s^3 q x z^2+ s^4 t z^3, \\
&\text{III}&&\text{SU(2)}&:&\quad &&y^2z  =x^3  + s a_{4,1} xz^2 + s^2 a_{6,2} z^3,\\
&\text{IV}^{\text{ns}} &&\text{SU(2)}&:&\quad  &&y^2z  =x^3 +s^2 a_{4,2} xz^2 + s^2 a_{6,2} z^3,\\
&\text{IV}^{\text{s}}\quad  &&\text{SU(3)}&:& \quad &&  y^2z + a_{3,1} s y z^2 =x^3  + s^2 a_{4,2} xz^2 + s^3a_{6,3} z^3,\\
	&\text{I}_0^{*\text{ns}}\quad  &&\text{G}_2&:& \quad &&y^2z=x^3+s^{2} a_{4,2}xz^2 +s^3 a_{6,3}z^3,\\
&\text{IV}^{*\text{ns}} \quad &&\text{F}_4&:&\quad 	&&y^2z =x^3 + s^{3} a_{4,3}xz^2 +s^4 a_{6,4}  z^3,\\
&\text{IV}^{*\text{s}} \quad\ &&\text{E}_6&:&\quad &&y^2z+a_{3,2} s^2 y z^2   = x^3 +s^{3}  a_{4,3} x z^2 +  s^5 a_{6,5}z^3,\\
&\text{III}^{*} \quad\   &&\text{E}_7&:&\quad  &&y^2z =x^3 +s^3 a_{4,3}  xz+s^{5} a_{6,5} z^3, \\
&\text{II}^{*} \quad\  &&\text{E}_8&:& \quad &&y^2z=x^3+s^4 a_{4,4} xz^2 +s^5 a_{6,5} z^3.
\end{align}

\begin{thm}
Let $Y_0\to B$ be a singular Weierstrass model of a $G$-model. If $f:Y\to Y_0$ is a crepant resolution of $Y_0$ given by one of the sequence of blowups given in Table \ref{tab:blowupcenters}, the generating function of the Euler characteristic of any crepant resolution of $Y_0$ is given by the corresponding entry in Table \ref{Table:Push}.  
\end{thm}

\begin{rem}
The theorem does not address if the sequence of blowups define a crepant resolution. 
One usually has to assume some conditions on the coefficients of the Weierstrass equations.  See for example \cite{G2}.
In some cases, the dimension of the base plays a role too \cite{G2}. 
\end{rem}

\begin{table}[h!]
\centering
$
\scalebox{.91}{$
	\begin{array}{|c|c|c|}
	\hline
	\text{Group} & \text{Fiber Type} & \text{Crepant Resolution} \\\hline
		\text{SU}(2) & 
		\begin{matrix}
		\text{I}_2^\text{s},     \text{I}^{\text{ns}}_2\\
		 {\text{I}^{\text{ns}}_3},  \text{III} \\   \text{IV}^{\text{ns}} 
		 \end{matrix} & \begin{array}{c} \begin{tikzpicture}
	\node(X0) at (0,0){$X_0$};
	\node(X1) at (2.5,0){$X_1$};
	\draw[big arrow] (X1) -- node[above,midway]{$(x,y,s|e_1)$} (X0);		
 \end{tikzpicture}\end{array} \\\hline
 
		\begin{matrix} \text{SU}(3) \\ \text{G}_2 \end{matrix} & \begin{matrix} \text{I}_3^\text{s} , { \text{IV}^\text{s}}  \\ \text{I}_0^{* \text{ns}} \end{matrix} & \begin{array}{c} \begin{tikzpicture}
	\node(X0) at (0,0){$X_0$};
	\node(X1) at (2.5,0){$X_1$};
	\node(X2) at (5,0){$X_2$};
	\draw[big arrow] (X1) -- node[above,midway]{$(x,y,s|e_1)$} (X0);	
	\draw[big arrow] (X2) -- node[above,midway]{$(y,e_1|e_2)$} (X1);		
 \end{tikzpicture}\end{array}  \\\hline
	\begin{matrix}	\text{SU}(4) \\ \text{Spin}(7) \end{matrix} &\begin{matrix}  \text{I}_4^\text{s} \\ \text{I}_0^{* \text{ss}} \end{matrix} & \begin{array}{c} \begin{tikzpicture}
	\node(X0) at (0,0){$X_0$};
	\node(X1) at (2.5,0){$X_1$};
	\node(X2) at (5,0){$X_2$};
	\node(X3) at (7.5,0){$X_3$};
	\draw[big arrow] (X1) -- node[above,midway]{$(x,y,s|e_1)$} (X0);	
	\draw[big arrow] (X2) -- node[above,midway]{$(y,e_1|e_2)$} (X1);
	\draw[big arrow] (X3) -- node[above,midway]{$(x,e_2|e_3)$} (X2);		
 \end{tikzpicture}\end{array}  \\\hline
		\text{Spin}(8) &    \text{I}^{*\text{s}}_0 &  \begin{array}{c} \begin{tikzpicture}
	\node(X0) at (0,0){$X_0$};
	\node(X1) at (2.5,0){$X_1$};
	\node(X2) at (5,0){$X_2$};
	\node(X3) at (8,0){$X_3$};
	\node(X4) at (11,0){$X_4$};
	\draw[big arrow] (X1) -- node[above,midway]{$(x,y,s|e_1)$} (X0);	
	\draw[big arrow] (X2) -- node[above,midway]{$(y,e_1|e_2)$} (X1);
	\draw[big arrow] (X3) -- node[above,midway]{$(x-x_i s z,e_2|e_3)$} (X2);		
	\draw[big arrow] (X4) -- node[above,midway]{$(x-x_j s z,e_2|e_4)$} (X3);
 \end{tikzpicture}\end{array} \\\hline
		\text{F}_4 &  \text{IV}^{*\text{ns}} &  \begin{array}{c} \begin{tikzpicture}
	\node(X0) at (0,0){$X_0$};
	\node(X1) at (2.5,0){$X_1$};
	\node(X2) at (5,0){$X_2$};
	\node(X3) at (8,0){$X_3$};
	\node(X4) at (11,0){$X_4$};
	\draw[big arrow] (X1) -- node[above,midway]{$(x,y,s|e_1)$} (X0);	
	\draw[big arrow] (X2) -- node[above,midway]{$(y,e_1|e_2)$} (X1);
	\draw[big arrow] (X3) -- node[above,midway]{$(x,e_2|e_3)$} (X2);		
	\draw[big arrow] (X4) -- node[above,midway]{$(e_3,e_2|e_4)$} (X3);
 \end{tikzpicture}\end{array}\\\hline
 		\text{SU}(5) &  \text{I}_4^{\text{s}} &  \begin{array}{c} \begin{tikzpicture}
	\node(X0) at (0,0){$X_0$};
	\node(X1) at (2.5,0){$X_1$};
	\node(X2) at (5,0){$X_2$};
	\node(X3) at (8,0){$X_3$};
	\node(X4) at (11,0){$X_4$};
	\draw[big arrow] (X1) -- node[above,midway]{$(x,y,s|e_1)$} (X0);	
	\draw[big arrow] (X2) -- node[above,midway]{$(x,y,e_1|e_2)$} (X1);
	\draw[big arrow] (X3) -- node[above,midway]{$(y,e_1|e_3)$} (X2);		
	\draw[big arrow] (X4) -- node[above,midway]{$(y,e_2|e_4)$} (X3);
 \end{tikzpicture}\end{array}\\\hline
 		\text{Spin}(10) &  \text{I}_1^{*\text{s}} &  \begin{array}{c} \begin{tikzpicture}
	\node(X0) at (0,0){$X_0$};
	\node(X1) at (2.5,0){$X_1$};
	\node(X2) at (5,0){$X_2$};
	\node(X3) at (7.5,0){$X_3$};
	\node(X4) at (10,0){$X_4$};
	\node(X5) at (12.5,0){$X_5$};
	\draw[big arrow] (X1) -- node[above,midway]{$(x,y,s|e_1)$} (X0);	
	\draw[big arrow] (X2) -- node[above,midway]{$(y,e_1|e_2)$} (X1);
	\draw[big arrow] (X3) -- node[above,midway]{$(x,e_2|e_3)$} (X2);		
	\draw[big arrow] (X4) -- node[above,midway]{$(y,e_3|e_4)$} (X3);
	\draw[big arrow] (X5) -- node[above,midway]{$(e_2,e_3|e_5)$} (X4);
 \end{tikzpicture}\end{array}\\\hline
 \text{E}_6  & \text{IV}^{*\text{s}} &   \begin{array}{c} \begin{tikzpicture}
	\node(X0) at (0,0){$X_0$};
	\node(X1) at (2.5,0){$X_1$};
	\node(X2) at (5,0){$X_2$};
	\node(X3) at (7.5,0){$X_3$};
	\node(X4) at (10,0){$X_4$};
	\node(X5) at (10,-1){$X_5$};
	\node(X6) at (7.5,-1){$X_6$};
	\draw[big arrow] (X1) -- node[above,midway]{$(x,y,s|e_1)$} (X0);	
	\draw[big arrow] (X2) -- node[above,midway]{$(y,e_1|e_2)$} (X1);
	\draw[big arrow] (X3) -- node[above,midway]{$(x,e_2|e_3)$} (X2);		
	\draw[big arrow] (X4) -- node[above,midway]{$(e_2,e_3|e_4)$} (X3);
	\draw[big arrow] (X5) -- node[right,midway]{$(y,e_3|e_5)$} (X4);
	\draw[big arrow] (X6) -- node[above,midway]{$(y,e_4|e_6)$} (X5);
 \end{tikzpicture}\end{array} \\\hline
 \text{E}_7 &\text{III}^* &  \begin{array}{c}\begin{tikzpicture}
 	\node(X0) at (0,0){$X_0$};
	\node(X1) at (2.5,0){$X_1$};
	\node(X2) at (5,0){$X_2$};
	\node(X3) at (7.5,0){$X_3$};
	\node(X4) at (10,0){$X_4$};
	\node(X5) at (10,-1){$X_5$};
	\node(X6) at (7.5,-1){$X_6$};
	\node(X7) at (5,-1){$X_7$};
	\draw[big arrow] (X1) -- node[above,midway]{$(x,y,s|e_1)$} (X0);	
	\draw[big arrow] (X2) -- node[above,midway]{$(y,e_1|e_2)$} (X1);
	\draw[big arrow] (X3) -- node[above,midway]{$(x,e_2|e_3)$} (X2);		
	\draw[big arrow] (X4) -- node[above,midway]{$(y,e_3|e_4)$} (X3);
	\draw[big arrow] (X5) -- node[right,midway]{$(e_2,e_3|e_5)$} (X4);
	\draw[big arrow] (X6) -- node[above,sloped,midway]{$(e_2,e_4|e_6)$} (X5);
	\draw[big arrow] (X7) -- node[above,sloped,midway]{$(e_4,e_5|e_7)$} (X6);
 \end{tikzpicture}	\end{array}\\\hline
 \text{E}_8 & \text{II}^*  & \begin{array}{c}\begin{tikzpicture}
 	\node(X0) at (0,0){$X_0$};
	\node(X1) at (2.5,0){$X_1$};
	\node(X2) at (5,0){$X_2$};
	\node(X3) at (7.5,0){$X_3$};
	\node(X4) at (10,0){$X_4$};
	\node(X5) at (10,-1){$X_5$};
	\node(X6) at (7.5,-1){$X_6$};
	\node(X7) at (5,-1){$X_7$};
	\node(X8) at (2.5,-1){$X_8$};
	\draw[big arrow] (X1) -- node[above,midway]{$(x,y,s|e_1)$} (X0);	
	\draw[big arrow] (X2) -- node[above,midway]{$(y,e_1|e_2)$} (X1);
	\draw[big arrow] (X3) -- node[above,midway]{$(x,e_2|e_3)$} (X2);		
	\draw[big arrow] (X4) -- node[above,midway]{$(y,e_3|e_4)$} (X3);
	\draw[big arrow] (X5) -- node[right,midway]{$(e_2,e_3|e_5)$} (X4);
		\draw[big arrow] (X6) -- node[above,midway]{$(e_4,e_5|e_6)$} (X5);
			\draw[big arrow] (X7) -- node[above,midway]{$(e_2,e_4,e_6|e_7)$} (X6);
				\draw[big arrow] (X8) -- node[above,midway]{$(e_4,e_7|e_8)$} (X7);
 \end{tikzpicture}	\end{array}\\\hline\hline
 \text{SO}(3) & \text{I}_2^\text{ns} & \begin{array}{c} \begin{tikzpicture}
	\node(X0) at (0,0){$X_0$};
	\node(X1) at (2.5,0){$X_1$};
	\draw[big arrow] (X1) -- node[above,midway]{$(x,y|e_1)$} (X0);		
 \end{tikzpicture}\end{array} \\\hline
 \begin{matrix}
 \text{SO}(5)
 \end{matrix}
  & \text{I}_4^{\text{ns}} & \begin{array}{c} \begin{tikzpicture}
	\node(X0) at (0,0){$X_0$};
	\node(X1) at (2.5,0){$X_1$};
	\node(X2) at (5,0){$X_2$};
	\draw[big arrow] (X1) -- node[above,midway]{$(x,y,s|e_1)$} (X0);	
	\draw[big arrow] (X2) -- node[above,midway]{$(x,y,e_1|e_2)$} (X1);		
 \end{tikzpicture}\end{array}  \\\hline
 \text{SO}(6) & \text{I}_4^\text{s} & \begin{array}{c} \begin{tikzpicture}
	\node(X0) at (0,0){$X_0$};
	\node(X1) at (2.5,0){$X_1$};
	\node(X2) at (5,0){$X_2$};
	\node(X3) at (7.5,0){$X_3$};
	\draw[big arrow] (X1) -- node[above,midway]{$(x,y,s|e_1)$} (X0);	
	\draw[big arrow] (X2) -- node[above,midway]{$(y,e_1|e_2)$} (X1);
	\draw[big arrow] (X3) -- node[above,midway]{$(x,e_2|e_3)$} (X2);		
 \end{tikzpicture}\end{array} \\\hline
	\end{array}
	$}
	$
	\caption{The blowup centers of the crepant resolutions. See the beginning of Section \ref{sec:centers} for an explanation of our notation. }
	\label{tab:blowupcenters}
\end{table}

\begin{table}
\begin{center}
\scalebox{.91}{$
\begin{array}{|c|c|   c   |}
\hline 
\text{Algebra} & \text{Group} & \text{Generator classes of the blowup centers }(Z_i^{(n)}) \\
\hline
\text{A}_1 &\begin{array}{c} \text{SU}(2) \end{array}& \begin{array}{c}\begin{pmatrix}
H+2L & H+3L & S
\end{pmatrix}\end{array}
\\
\hline
\begin{array}{c}
\text{A}_2 \\
\text{G}_2 
\end{array} & 
\begin{array}{c} 
\begin{array}{c} \begin{matrix} \text{SU}(3) \\ \text{G}_2 \end{matrix}  \end{array} \\
\end{array} & \begin{array}{c} \begin{pmatrix}
H+2L & H+3L & S \\
H+3L-E_1& E_1 & 
\end{pmatrix}  \end{array} \\
\hline 
\text{A}_3 &\begin{array}{c} \text{SU}(4) \\ \text{Spin}(7) \end{array}&\begin{array}{c} \begin{pmatrix}
H+2L & H+3L & S \\
H+3L-E_1& E_1 &  \\
H+2L-E_1& E_2 & 
\end{pmatrix} \end{array} \\
\hline
\text{D}_4 & \text{Spin}(8) & \begin{pmatrix}
H+2L & H+3L & S \\
H+3L-E_1& E_1 &  \\
H+2L-E_1& E_2 & \\
H+2L-E_1& E_2-E_3& 
\end{pmatrix} \\
\hline 
\text{F}_4 & \text{F}_4 & \begin{pmatrix}
H+2L & H+3L & S \\
H+3L-E_1& E_1 &  \\
H+2L-E_1& E_2 & \\
E_2-E_3& E_3 &
\end{pmatrix} \\
\hline 
\text{A}_4 &\text{SU}(5) & \begin{pmatrix}
H+2L & H+3L & S \\
H + 2 L - E_1 & H + 3 L - E_1 & E_1 \\
H + 3 L - E_1 - E_2 & E_1 - E_2 \\
H + 3 L - E_1 - E_2 - E_3 & E_2 &
\end{pmatrix} \\
\hline 
\text{D}_5 &\text{Spin}(10) &\begin{pmatrix}
H+2L & H+3L & S \\
H + 3L - E_1& E_1&  \\
H + 2L - E_1& E_2&  \\
H + 3 L - E_1 - E_2& E_3&  \\
E_2-E_3& E_3-E_4& 
\end{pmatrix}  \\
\hline 
\text{E}_6 & \text{E}_6 & \begin{pmatrix}
H+2L & H+3L & S \\
H+3L-E_1& E_1 &  \\
H+2L-E_1& E_2 & \\
E_2-E_3& E_3 & \\
H+3L-E_1-E_2& E_3-E_4&  \\
H+3L-E_1-E_2-E_5& E_4&\\
\end{pmatrix}  \\
\hline 
\text{E}_7 & \text{E}_7 & \begin{pmatrix}
H+2L & H+3L & S \\
H+3L-E_1& E_1 &  \\
H+2L-E_1& E_2 & \\
H + 3 L - E_1 - E_2& E_3&  \\
E_2-E_3& E_3-E_4& \\
E_2-E_3-E_5&E_4 & \\
E_4-E_6 & E_5
\end{pmatrix} \\
\hline 
\text{E}_8 & \text{E}_8 & \begin{pmatrix}
H+2L & H+3L & S \\
H+3L-E_1& E_1 &  \\
H + 2 L - E_1& E_2&\\
H + 3 L - E_1 - E_2& E_3&\\
E_2 - E_3& E_3 - E_4&\\
E_4& E_5&\\
E_2 - E_3 - E_5& E_4 - E_6& E_6\\
E_4 - E_6 - E_7 & E_7
\end{pmatrix} \\
\hline  \hline
\text{A}_1 &\text{SO}(3) & \begin{pmatrix} H + 2 L & H + 3 L \end{pmatrix} \\\hline
\text{B}_2& \text{SO}(5) & \begin{pmatrix} H + 2 L & H + 3 L  & 2L \\ H + 2L - E_1 & H + 3 L - E_1 & E_1 \end{pmatrix} \\\hline
\hline
\text{A}_3 & \text{SO}(6) & \text{same as\ \   SU(4), but with $S = 2L$} \\
\hline
\end{array}
$}
\end{center}
\caption{The classes of the centers of the blowups for all $G$-models \label{Table:BlowCenterZ}}
\end{table}


 \begin{table}[hb]
 \begin{center}
\renewcommand{\arraystretch}{2}
\scalebox{1}{$
\begin{array}{|c |c |   c |    c|}
\hline 
\text{Algebra} & \text{Group} &\text{ Kodaira Fiber} &  \chi(Y)=\pi_*\Big( f_* c(TY) \cap [Y]\Big) \\
\hline
-& \{\mathrm{e}\}&\text{I}_1 &\displaystyle \frac{12L}{1+6L} c(B)\\
\hline
\text{A}_1 & \text{SU}(2) &
\displaystyle 
\begin{matrix}
		\text{I}_2^\text{s},     \text{I}^{\text{ns}}_2\\
		 {\text{I}^{\text{ns}}_3},  \text{III} \\   \text{IV}^{\text{ns}} 
		 \end{matrix}
&\displaystyle 6\frac{ 2L+3 L S  - S^2}{(1+S) (1+6 L - 2 S )} c(B)\\
\hline
\begin{array}{c}
\text{A}_2 \\
\text{G}_2
\end{array}
& 
\begin{array}{c}
\text{SU}(2) \\
\text{G}_2
\end{array}
& 
\begin{array}{c}
  \text{I}^{\text{s}}_3, \    { \text{IV}^{\text{s}}} \\
 \text{I}^{* \text{ns}}_0 
\end{array}
 & \displaystyle 12\ \frac{L+2S L- S^2 }{(1 + S) (1+6L-3 S )}c(B)\\
  \hline
  \text{C}_2 & \text{USp}(4) &  \text{I}_4^{\text{ns}} &{\displaystyle \ 4  \frac{3 L (1 + 2 L) + 6 L (1 + 2 L) S - (5 + 8 L) S^2 }{(1 + 2 L) (1 + 6 L - 4 S) (1 + S)}c(B)}\\
  \hline
\begin{array}{c}{ 
\text{A}_3} \\
\text{B}_3 \\
\end{array}& \begin{array} {c}
\text{SU}(4) \\
\text{Spin}(7) \\
\end{array} 
&
 \begin{array} {c}
\text{I}^{\text{s}}_4 \\
\text{I}^{* \text{ss}}_0
\end{array} 
&  \displaystyle 4\  \frac{3 L+12 L^2+L S-5 S^2+30 L^2 S-35 L S^2+10 S^3}{(1+S) (1+6L-4 S) (1+4L-2 S)} c(B)\\
\hline
\begin{array}{c}
\text{D}_4 \\
\text{F}_4 
\end{array}& \begin{array} {c}
\text{Spin}(8) \\
\text{F}_4
\end{array} & \begin{array} {c}
\text{I}^{*\text{s}}_0 \\
 \text{IV}^{*\text{ns}}
\end{array} & \displaystyle 12\ \frac{L+3S L-2 S^2 }{(1 + S) (1+6L-4 S )}c(B) \\
\hline 
\text{A}_4 &\text{SU}(5) &\text{I}_5^\text{s} & \displaystyle 	{ \frac{12 L+42 L^2 S+12 L^2-35 L S^2+32 L S-30 S^2}{(1+L) (1+S) (1+6 L-5 S)} c(B)}\\
\hline 
\text{D}_5 &\text{Spin}(10) &\text{I}_1^{*\text{s}} &\displaystyle \frac{4 \left(-8 (4 L+1) S^2+6 (4 L+1) L S+3 (2 L+1) L+10 S^3\right)}{(S+1) (-2 L+S-1) (-6 L+5 S-1)} c(B)\\
\hline 
\text{E}_6 & \text{E}_6 &  \text{IV}^{*\text{s}}&\displaystyle 	3\  \frac{ 4  L+ 12 L^2 -12 S^2 + 6 S L   - 81 S^2 L + 54 S L^2 + 30 S^3}{(1+S) (1+6L-5 S) (1+3L-2S)}c(B) \\
\hline 
\text{E}_7 & \text{E}_7 &     \text{III}^{*}& \displaystyle  2\  \frac{ 6 L+24L^2 +7 L S-21 S^2 +120 L^2 S-190 L S^2+75 S^3}{(1+S) (1+6 L-5 S) (1+4L-3 S)} c(B)\\
\hline 
\text{E}_8 & \text{E}_8 &  \text{II}^{*} &\displaystyle 12\  \frac{L+6 L S- 5S^2}{(1+S)(1+6L-5S)}c(B)\\
\hline 
\hline
\text{A}_1 & \text{SO}(3) &  \text{I}_2^{\text{ns}} & {\displaystyle   \frac{ 12 L}{1+ 4 L} c(B) }\\
\hline 
\text{B}_2 & \text{SO}(5) &  \text{I}_4^{\text{ns}} &{\displaystyle \  \frac{ 4 L (3+4L)}{(1+ 2 L)^2} c(B)}\\
\hline 
\text{A}_3 & \text{SO}(6) & { \text{I}_4^{\text{s}} }& {\displaystyle    \frac{12 L}{1 + 2 L } c(B)}\\
\hline 
\end{array}
$}
\end{center}
\caption{Generating functions of Euler characteristic of crepant resolutions of Tate's models with trivial Mordell-Weil groups. $S$ is the divisor over which the generic fiber is of type given by the Kodaira fiber and $L=c_1(\mathscr{L})$ where $\mathscr{L}$ is the fundamental line bundle of the Weierstrass model. \label{Table:Push}}
 \end{table}
 \clearpage

 \begin{table}[htb]
 \begin{center}
\scalebox{.98}{$ 
 \begin{array}{|c | c |}
  \hline
 \text{Models}  & \text{$\chi(Y_3)$,  Euler characteristic}  \\
  \hline
\text{Smooth Weierstrass}  &   12 L (c_1-6 L)\\
  \text{SU}(2) & 6 (2 c_1 L-12 L^2+5 L S-S^2)\\
 \text{SU}(3)\ \text{or}\  \text{G}_2& 12  (c_1 L-6 L^2+4 L S-S^2)\\
\text{SU}(4)\ \text{or}\  \text{Spin}(7)& 4  (3 c_1 L-18 L^2+16 L S-5 S^2)\\
\text{Spin}(8) \ \text{or} \  F_4 & 12 (c_1 L-6 L^2+6 L S-2 S^2)\\
\text{SU}(5) &  2 (6 c_1 L-36 L^2+40 L S-15 S^2)\\
\text{Spin}(10) & 4 (3 c_1 L-18 L^2+21 L S-8 S^2)\\
 \text{E}_6& 6  (2 c_1 L-12 L^2+15 L S-6 S^2)\\
 \text{E}_7& 2 (6 c_1 L-36 L^2+49 L S-21 S^2)\\
\text{E}_8&   12  (c_1 L-6 L^2+10 L S-5 S^2)\\
\hline\hline
\text{SO}(3)& 12L(c_1-4L) \\
\text{SO}(5)&  4L(3c_1 - 8L) \\
\text{SO}(6)&12L(c_1 - 2L)  \\
 \hline
\end{array}
$}
\caption{Euler characteristic for elliptic threefolds 
\label{Table:Euler3}}
\end{center}

 \begin{center}
\scalebox{.98}{$ 
 \begin{array}{|c|c |}
 \hline
 \text{Models}  & \text{ $\chi(Y_4)$,    Euler characteristic} \\
 \hline
\text{Smooth Weierstrass} &  12 L  (-6 c_1 L+c_2+36 L^2) \\
 \text{SU}(2) & 6  (-12 c_1 L^2+5 c_1 L S-c_1 S^2+2 c_2 L+72 L^3-54 L^2 S+15 L S^2-S^3)\\
 \text{SU}(3) \text{or}\  G_2 & 12  (-6 c_1 L^2+4 c_1 L S-c_1 S^2+c_2 L+36 L^3-42 L^2 S+17 L S^2-2 S^3)\\
\text{SU}(4)\ \text{or}\  \text{Spin}(7)&  4  (-18 c_1 L^2+16 c_1 L S-5 c_1 S^2+3 c_2 L+108 L^3-166 L^2 S+89 L S^2-15 S^3)\\
\text{SU}(5)& -72 c_1 L^2+80 c_1 LS -30c_1 S^2+12 c_2 L+432 L^3-830 L^2 S+555 L S^2-120 S^3\\
\text{Spin}(10)&  4 (-18c_1 L^2+21c_1 L S-8c_1 S^2+3c_2 L+108 L^3-210 L^2 S+140 L S^2-30 S^3)\\
\text{Spin}(8) \ \text{or} \  F_4& 12  (-6 c_1 L^2+c_2 L+36 L^3+6 c_1 L S-2 c_1 S^2-60 L^2 S+34 L S^2-6 S^3) \\
 \text{E}_6& 3  (-24 c_1 L^2+30 c_1 L S-12 c_1 S^2+4 c_2 L+144 L^3-288 L^2 S+195 L S^2-42 S^3)\\
 \text{E}_7& 2  (-36 c_1 L^2+49 c_1 L S-21 c_1 S^2+6 c_2 L+216 L^3-454 L^2 S+321 L S^2-72 S^3)\\
\text{E}_8&  12  (-6 c_1 L^2+10 c_1 L S-5 c_1 S^2+c_2 L+36 L^3-90 L^2 S+75 L S^2-20 S^3)\\
\hline\hline
\text{SO}(3)& 12L(16L^2- 4c_1L +c_2)  \\
\text{SO}(5)& 4L(20L^2-8c_1 L+3c_2)   \\
\text{SO}(6)&  12L(4L^2 - 2Lc_1 + c_2)  \\
\hline
\end{array}
$}
\caption{Euler characteristic for elliptic fourfolds \label{Table:Euler4}}
\end{center}

\begin{center}
\scalebox{.98}{$ 
 \begin{array}{|c|c |}
 \hline
 \text{Models}  & \text{ $\chi(Y_4)$,   Euler characteristic} \\
 \hline
\text{Smooth Weierstrass} &  12 c_1 c_2 + 360 c_1^3 \\
 \text{SU}(2) & 6  (2 c_1 c_2 + 60 c_1^3 - 49 c_1^2 S + 14 c_1 S^2 - S^3)\\
 \text{SU}(3)\  \text{or}\  G_2 & 12  (c_1 c_2 + 30 c_1^3 - 38 c_1^2 S + 16 c_1 S^2 - 2 S^3)\\
\text{SU}(4)\ \text{or}\  \text{Spin}(7)&  12 (3 c_1 c_2 + 30 c_1^3 - 50 c_1^2 S + 28 c_1 S^2 - 5 S^3))\\
\text{Spin}(8) \ \text{or} \  F_4& 12  (c_1 c_2 +30 c_1^3 - 54 c_1^2 S + 32 c_1 S^2 - 6 S^3) \\
\text{SU}(5)&  3 (4 c_1 c_2+120 c_1^3-250 c_1^2 S+175 c_1 S^2-40 S^3)\\
\text{Spin}(10)& 12 (c_1c_2+30c_1^3-63 c_1^2 S+44 c_1 S^2-10 S^3) \\
 \text{E}_6& 3  (4 c_1 c_2+120 c_1^3  - 258 c_1^2 S + 183 c_1 S^2 - 42 S^3)\\
 \text{E}_7& 6 ( 2 c_1 c_2+60 c_1^3 - 135 c_1^2 S + 100 c_1 S^2 - 24 S^3)\\
\text{E}_8&  12  ( c_1 c_2+ 30 c_1^3 - 80 c_1^2 S + 70 c_1 S^2 - 20 S^3)\\
\hline\hline
\text{SO}(3)& 12 c_1(12 c_1^2+c_2) \\
\text{SO}(5)&  12 c_1 (4 c_1^2+c_2)  \\
\text{SO}(6)&   12c_1( 2c_1^2 + c_2) \\
\hline
\end{array}
$}
\caption{Euler characteristic for Calabi-Yau elliptic fourfolds where $c_1=L$.}
\label{table:CY4}
\end{center}
\end{table}

\clearpage

\begin{table}[htb]
 \begin{center}
\scalebox{.97}{$
\begin{array}{|c |c |   c |    c|c|c|}
\hline 
\text{Algebra} & \text{Group} &\text{ Kodaira Fiber} &  h^{1,1}(Y_3) & h^{2,1} (Y_3) & \chi(Y_3)  \\\hline
& & & & &\\
-& \{\mathrm{e}\}& \text{I}_1&11-K^2&11 + 29 K^2 &  -60 K^2 \\
& & & & & \\
\hline
& & & & &\\
\text{A}_1 & SU(2) &\displaystyle 
\begin{matrix}
		\text{I}_2^\text{s},     \text{I}^{\text{ns}}_2\\
		 {\text{I}^{\text{ns}}_3},  \text{III} \\   \text{IV}^{\text{ns}} 
		 \end{matrix}
& 12 - K^2&12+ 29 K^2 + 15 KS + 3 S^2 & - 60 K^2 - 30 KS - 6 S^2 \\
& && & &\\
\hline
& & & && \\
\text{A}_2 &\text{SU}(3) &  \text{I}^{\text{s}}_3, \      \text{IV}^s & &&\\
& &  &13 - K^2 &13 + 29 K^2 + 24 K S + 6 S^2 &-60 K^2 - 48 K S - 12 S^2\\
\text{G}_2 & \text{G}_2  & \text{I}^{* \text{ns}}_0 & & &\\
  & & & &&\\
  \hline
  & & & &&\\
  \text{A}_3 & \text{SU}(4) &  \text{I}^{\text{s}}_4& &&
\\
&  & &14 - K^2 &14 + 29 K^2 + 32 KS + 10 S^2& -60 K^2-64 K S-20 S^2 \\
\text{B}_3 & \text{Spin}(7) & \text{I}_0^{* \text{ss}}  &  &&
\\
& & & & &\\
\hline
& & & & &\\
\text{D}_4 & \text{Spin}(8) &  \text{I}^{*\text{s}}_0 & &&
\\
& &  & 15-K^2 & 15 + 29 K^2 + 36 KS + 12 S^2  & -60 K^2-72 K S-24 S^2\\
\text{F}_4 & \text{F}_4 &  \text{IV}^{*\text{ns}}& & &\\
& &  & & &\\
\hline 
& & & & &\\
\text{A}_4 & \text{SU}(5) &  \text{I}^{\text{s}}_5 & 15-K^2  & 15+29 K^2+40 K S+15 S^2 & -60 K^2-80 K S-30 S^2
\\
& &  & & &  \\
\hline 
& &  & & &  \\
\text{D}_5 & \text{Spin}(10) &  \text{I}^{*\text{s}}_1& 16-K^2 & 16+29 K^2+42 K S+16 S^2 & -60 K^2-84 K S-32 S^2 \\
& &  & & &\\
\hline 
& & & & &\\
\text{E}_6 & \text{E}_6 &  \text{IV}^{*\text{s}}&17 - K^2&17 + 29 K^2 + 45 KS + 18 S^2 & -60 K^2-90 K S-36 S^2 \\
& &  & && \\
\hline 
&& & & & \\
\text{E}_7 & \text{E}_7 &     \text{III}^{*}& 18 - K^2 & 18 + 29 K^2 + 49 KS + 21 S^2 &-60 K^2 - 98 K S - 42 S^2  \\
& & & & & \\
\hline 
& &  & & &\\
\text{E}_8 & \text{E}_8 &  \text{II}^{*} & 19 - K^2&  19 + 29 K^2 + 60 KS + 30 S^2 &-60 K^2 - 120 K S - 60 S^2 \\
& &  & & &\\
\hline 
\hline
& &  & & &\\
\text{A}_1 & \text{SO}(3) & \text{I}_2^{\text{ns}} &  12- K^2& 12+17 K^2 &-36K^2 \\
\hline 
& &  & & &\\
 \text{B}_2 &\text{SO}(5) & \text{I}_4^{\text{ns}} &  14- K^2&  14+9 K^2 & -20K^2 \\
\hline 
& &  & & &\\
\text{A}_3 & \text{SO}(6) & \text{I}_4^\text{s} &  14- K^2&  14+5 K^2 & -12K^2 \\
\hline 
\end{array}
$}
\end{center}
\caption{Hodge numbers and Euler characteristic of Calabi-Yau threefolds obtained from crepant resolutions of Tate's models.  \label{Table.HodgeCY3} }
\end{table}
\clearpage

 \section{Discussion} \label{sec:discuss}

In this paper, we have computed the generating functions for the Euler characteristics of $G$-models obtained by crepant resolutions of Weierstrass models with bases of arbitrary dimension. The case of $G$-models that are also  Calabi-Yau varieties is 
important in string theory and is treated here as a special case. In particular, we list the Euler characteristic of $G$-models that are  elliptic threefolds and fourfolds. For Calabi-Yau threefolds, we also compute the Hodge numbers. These results are insensitive to the particular choice of resolution due to Batyrev's theorem on the Betti numbers of crepant birational equivalent varieties and Kontsevich's theorem on the Hodge numbers of birational equivalent Calabi-Yau varieties \cite{Batyrev.Betti,Kontsevich.Orsay}.
We have considered all possible $G$-models with $G$ a simple Lie group, except for the case of Kodaira fibers I$_{n>5}$ and I$^*_{n>1}$ that we will treat in a follow-up paper.

We start with a $G$-model given by a singular Weierstrass model 
$\varphi: Y_0\longrightarrow B$
with  a fundamental line bundle $\mathscr{L}$ (in the Calabi-Yau case, $c_1(\mathscr{L})=c_1(TB)$). 
 Given a crepant resolution $f:Y\longrightarrow Y_0$ determined by a sequence of blowups with smooth centers that are complete intersections with normal crossings,  
we compute the Euler characteristic of $Y$ as the degree of its 
total Chern class defined in homology
$$
\chi(Y)=\int_Y c (Y).
$$
We work relative to a smooth   base $B$ of arbitrary dimension. 
Using the functorial properties of the degree, we  pushforward first to the Chow ring of the projective bundle and then to the Chow ring of the base: 
$$
\chi(Y)= \int_B\pi_* f_* c(Y). 
$$
The final result is a generating function for the Euler characteristic. 

A key result of this work is Theorem \ref{Thm:Push}, which has numerous applications in intersection theory.
We also provide a simple proof of an identity (Lemma \ref{Jacobi}) that can be traced back to Jacobi's thesis  and appears in numerous situations in mathematics and physics, which is instrumental in the proof of Theorem \ref{Thm:Push}.

We also retrieve in a unifying way known results on the Euler characteristics and Hodge numbers of Calabi-Yau threefolds.  
Furthermore, we have proven {\it en passant} a conjecture of 
Blumenhagen, Grimm, Jurke, and Weigand \cite{BGJW}
 on the  Euler characteristics of Calabi-Yau fourfolds that are  $G$-models with $G$ belonging to the exceptional series. One interesting point that is almost trivial from the perspective taken in this paper is that certain $G$-models with different $G$ will have the same Euler characteristic just because they are resolved by the same sequence of blowups.

\section*{Acknowledgements}
The authors are grateful to  Paolo Aluffi, Jim Halverson,  Remke Kloosterman, Cody Long, Kenji Matsuki, Julian Salazar, Shu-Heng Shao, and Shing-Tung Yau for helpful discussions.
 The authors would like in particular to acknowledge Andrea Cattaneo for  many useful comments and suggestions. 
The authors are thankful to all the participants of the workshop  ``A Three-Workshop Series on the Mathematics and Physics of F-theory'' supported by the National Science Foundation (NSF) grant DMS-1603247. 
M.E. is supported in part by the National Science Foundation (NSF) grant DMS-1701635 ``Elliptic Fibrations and String Theory''. 
P.J.  is  supported by NSF grant PHY-1067976. P.J. would like to extend his gratitude to Cumrun Vafa for his tutelage and continued support. 
M.J.K. would like to acknowledge partial support from NSF grant PHY-1352084.
 M.J.K. is thankful to Daniel Jafferis for his guidance and constant support. 
 

\appendix

\section{Jacobi's Partial Fraction Identity}\label{Sec:Jacobi}
 In this section, we prove a formula of Jacobi  and exploit the theorem to give a simple proof of a formula of Louck and Biedenharn \cite[Appendix A, p. 2400]{Louck:1970jd} by demonstrating its equivalence with the following theorem of Jacobi. 

\begin{thm}[Jacobi, {\cite[Section III.17, p.  29-30]{Jacobi.Thesis}}] 
\label{Thm:Jacobi}
Let $a_i$  ($i=1, \ldots, d$) be $d$ distinct  elements of an integral domain. Then 
\begin{equation}\label{Jacobi.Id}
\prod_{i=1}^d \frac{1}{x-a_i}= \sum_{i=1}^d  \frac{1}{x-a_i}
\prod^d_{\substack{ j=1\\
 j\neq i}} 
\frac{1}{a_i-a_j}.
\end{equation}
\end{thm}
\begin{proof}
Let 
\begin{equation}\label{J1}
F(x)=\prod_{i=1}^d \frac{1}{x-a_i},
\end{equation}
where $a_i\neq a_j$ for $i\neq j$.  
We would like to find the partial fraction expansion of $F(x)$. That is, we would like to find coefficients $A_i$ ($i=1,\cdots, d$) such that 
\begin{equation}\label{J2}
F(x)= \sum_{i=1}^d  \frac{A_i}{x-a_i}.
\end{equation}
We determine $A_i$ by the {\em method of residues}. 
Multiplying \eqref{J2} by $(x-a_i)$, simplifying, and evaluating at $x=a_i$  gives 
 $$\left.(x-a_j) F(x)\right|_{x=a_j}=A_j.$$ 
 Applying the above formula to \eqref{J1}, we get 
$
A_j=\prod_{i\neq j} \frac{1}{a_i-a_j}
$, which is the identity of Jacobi:  
\begin{equation}
\prod_{i=1}^d \frac{1}{x-a_i}= \sum_{i=1}^d  \frac{1}{x-a_i}\prod^d_{\substack{ j=1\\
 j\neq i}}  \frac{1}{a_i-a_j}.
\end{equation}
\end{proof}

\begin{thm}[Jacobi, Louck--Biedenharn, Cornelius] 
\label{Thm.Jacobi}
Let $h_r(x_1, \cdots, x_d)$ be the homogeneous complete symmetric polynomial  of degree $r$  in $d$ variables of an integral domain. Then,
$$
h_{r}(x_1, \cdots, x_d)=
\sum_{\ell=1}^d x_\ell^{r+d-1} \prod^d_{\substack{ m=1\\
 m\neq \ell}} \frac{1}{ x_\ell -x_m }.
$$

\end{thm}
This theorem was proven by Louck-Biedenharn \cite[Appendix A, p. 2400]{Louck:1970jd} and Cornelius \cite{MR2840951}. 
We present a new and much simpler proof below by showing that the theorem is simply a  reformulation of  Jacobi's identity (Theorem \ref{Thm:Jacobi}).
\begin{proof}
 Substituting $x\to 1/t$ in Equation  \eqref{Jacobi.Id} gives:
\begin{align}
& \prod_{i=1}^d \frac{t}{1-a_i t}= \sum_{i=1}^d  \frac{t}{1-a_i t}\prod^d_{\substack{ j=1\\
 j\neq i}}  \frac{1}{a_i-a_j}.\nonumber
\end{align}
Expanding $1/(1-a_i t)$ in both side of the equation gives 
\begin{align}\nonumber
& t^{d}\sum_{r=0}^\infty  h_r(a_1, \ldots , a_d) t^r= t\sum_{i=1}^d \sum_{k=0}^\infty  a_i^k t^k \prod^d_{\substack{ j=1\\
 j\neq i}}  \frac{1}{a_i-a_j}\\
& \sum_{r=0}^\infty  h_r(a_1, \ldots , a_d) t^{r+d-1}= \sum_{k=0}^\infty\Big( \sum_{i=1}^d a_i^k  \prod^d_{\substack{ j=1\\
 j\neq i}} \frac{1}{a_i-a_j}\Big) t^k.\nonumber
\end{align}
Comparing terms of the same degree in $t$, we get the final expression of Lemma \ref{Jacobi}:
\begin{equation}
h_{r} (a_1,\ldots, a_d)=\sum_{i=1}^d a_i^{r+d-1}  \prod^d_{\substack{ j=1\\
 j\neq i}}  \frac{1}{a_i-a_j}.\nonumber
\end{equation}
\end{proof}

\section{The Euler Characteristic as the Degree of the Top Chern Class}
\label{appB.Euler}
The purpose of this section is to explain  from different points of view why the Euler characteristic is the degree of the top Chern class.
 Traditionally, this statement is seen as a generalization of the Poincar\'e--Hopf  theorem that 
asserts that the total degree of a vector field defined on a smooth manifold $M$ is the Euler characteristic of $M$. This statement can also be seen as a generalization of the Gauss--Bonnet--Chern Theorem (which is  itself is a consequence of Poincar\'e--Hopf theorem). 
Here we will  review three different  approaches. The first one relies on Leftschetz fixed point theorem.  The second one uses 
he Poincar\'e--Hopf  theorem  using the interpretation of  Chern classes as related to the class of some 
degenerated loci as discussed in Chapter 3 of Fulton. The third one  is an application of the Hirzebruch--Riemann--Roch theorem and the Hodge decomposition theorem.

Let $M$ be a smooth compact manifold. The $k$th Betti number of $M$ is by definition the dimension of the cohomology group $H^k(M,\mathbb{Q})$. The Euler characteristic of $M$ is denoted by $\chi(X)$ and is defined as the  following alternative sum of Betti numbers of $M$: 
$$
\chi(M):=\sum_{k=0}^{\dim  M} (-1)^k b_k, \quad b_k:= \dim H^i(M,\mathbb{Q}).
$$

 \subsection{Lefschetz fixed point theorem and the Euler characteristic as an intersection number}

 \begin{thm}[Lefschetz fixed point theorem]
Let $M$ be a compact smooth manifold of dimension $m$ and $f:M\longrightarrow M$ a continuous  map.
We define the Lefschetz number of $f$ as 
$$
L(f):=\sum_{k=0}^m (-1)^k\   \mathrm{tr} \Big(f^*|H^k(M,\mathbb{Q})\Big), \quad f^*: H^k(M,\mathbb{Q})\longrightarrow H^k(M,\mathbb{Q}).
$$ 
 Then $L(f)$ is equal to the intersection number of the graph $\Gamma_f$ of $f$ and the diagonal $\Delta$  in $M\times M$	
$$
L(f)=\int_{M\times M} \Gamma_f \cdot \ \Delta.
$$Thus, the Leftschetz number $L(f)$ is the number of fixed points of $f$  counted with multiplicities.  
\end{thm}

 \begin{cor}
 Let $M$ be a compact smooth manifold and $\Delta$ be the diagonal of $M\times M$, then the Euler characteristic of $M$, $\chi(M)=\sum_i (-1)^k \dim H^i(M, \mathbb{Q})$,  is equal to the self-intersection of $\Delta$ in $M\times M$: 
 $$
 \chi(M)=\int_{M\times M} \Delta \cdot \Delta.
 $$
\end{cor}
\begin{proof}
Consider the special case of  Lefschetz theorem  for which $f$ is the identify map on $M$.  Then, the  Leftschetz number  reduces to the Euler characteristic $\chi(M)$ as the trace $\mathrm{tr} \Big(f^*|H^k(M,\mathbb{Q})\Big)$ becomes the $k$th Betti number $b^k$ of $M$  
 and the  intersection number $\int_{M\times M} \Gamma_{f} \cdot \ \Delta$ becomes the  self-intersection of the diagonal $\Delta$ in $M\times M$.
\end{proof}

\begin{thm}[Self-intersection formula, {see \cite[Corollary 6.3, p. 102-103]{Fulton.Intersection}}]
Let $i: Z\to X$ be a regular imbedding of codimension $d$ and normal bundle $N$. 
Then for any $\alpha\in A_*(Z)$ we have the self-intersection formula 
$$
i^* i_* (\alpha)=c_d(N)\cap \alpha.
$$
\end{thm}

\begin{thm}\label{Thm:PoincareHopfChern}
If $X$ is a nonsingular complete algebraic variety, then the Euler characteristic of $X$ is equal to the degree of the  total homological Chern class of $X$:
$$
\chi(X)=\int c(X), \quad c(X):=c(TX)\cap [X].
$$
\end{thm}
\begin{proof}
The theorem follows from the previous corollary expressing the Euler characteristic $\chi(X)$ as the 
self-intersection of the diagonal $\Delta$ in $X\times X$, followed by the self-intersection formula expressing $\Delta \cdot \Delta$ as the class $c_{\dim  X} (N_\Delta {X\times X})\cap [\Delta]$. 
Since  the normal bundle of $\Delta$ in $X\times X$ 
is isomorphic to the tangent bundle of $X$ (see for example  \cite[Lemma 11.23, p. 127]{BottTu}), it follows that  \cite[Example 8.1.12, p. 136]{Fulton.Intersection},  the self-intersection of the diagonal $\Delta$ in $X\times X$ is   $\int c_{\dim X} (TX)\cap[X]=\int c(TX)\cap [X]$:
 \begin{align}
\begin{aligned}
\chi(X) &=\int_{X\times X} \Delta\cdot {\Delta}=\int c(N_\Delta {X\times X})\cap [\Delta]=\int c(TX)\cap[X].
\end{aligned}\nonumber
\end{align}
\end{proof}

\subsection{Poincar\'e-Hopf theorem and the Euler characteristic}

\begin{thm}[Poincar\'e-Hopf]
Let $M$ be a smooth compact manifold without boundary and $v$ be a vector field with isolated zeros.  Then the sum of  the local indices at the zeros of $v$  is equal to the Euler characteristic of $M$. 
\end{thm}
\begin{rem}
This theorem can be generalized to manifolds with boundaries by requiring $v$ to point outward. 
Poincar\'e proved a two dimensional version of this theorem in 1885. The general version was proven by Hopf in 1926. 
\end{rem}
\begin{thm}[{\cite[Example 3.2.16, p. 61]{Fulton.Intersection}}] 
Let $E$ be a vector bundle of rank $r$ on a smooth variety $X$, let $s$ be a section of $E$, and $Z$ the zero-scheme of $s$. If $X$ is purely $n$-dimensional and $s$ is a regular section, then $Z$ is purely $(n-r)$-dimensional, and 
$$
[Z]=c_r(E)\cap [X].
$$
\end{thm}

In particular, if $E$ is the tangent bundle $TX$ of $X$, then $r$ (i.e. the rank of $E$) is the dimension of $X$, and the section $s$ of $E$ is just a vector field. The zero-scheme $Z$ is a $0$-cycle that is the sum of  the isolated singularities of $s$ counted with multiplicities. 
Hence, the degree of the  top Chern class of $TX$  gives  the index of the vector field $s$, which is  the Euler characteristic of $M$ by the Poincar\'e--Hopf theorem. Since the degree of $c(X)$ is exactly the degree of $c_{r}(TX)\cap [X]$,  we retrieve Theorem \ref{Thm:PoincareHopfChern}:
$$
\chi(X)=\int c(X).
$$

 \subsection{Hirzebruch--Riemann--Roch theorem and the Euler characteristic}
 In this sub-section, using the Hirzebruch--Riemann--Roch theorem and the Hodge decomposition theorem, we prove that the Euler characteristic of a nonsingular  projective variety is the degree of its homological total  Chern class.  
We follow  Fulton (\cite[Example 18.3.7, p. 362]{Fulton.Intersection} and  \cite[Example 3.2.5, p. 57]{Fulton.Intersection}) as presented by D. R\"ossler \cite{Rossler}.  
 We denote the Todd class, the Chern character, and the dual  of a vector bundle $E$ by $\mathrm{td}(E)$ and $\mathrm{ch}(E)$, and $E^\vee$ respectively.

Let $X$ be a projective variety of dimension $d$ and $V$ a coherent sheaf defined over $X$. 
We denote by  $H^q(X,V)$  the $q$-th cohomology group of $X$ with coefficients in the sheaf of germs of local sections of $V$.  
The cohomology groups $H^q(X,V)$ vanish for $q>d$ and are all finite dimensional for $0 \leq q\leq d$.
The Euler characteristic of $V$ in  $X$ is by definition the finite number
$$
\chi(X,V):=\sum_{q=0}^d (-1)^q \dim H^q(X, V).
$$ 
The Hirzebruch--Riemann--Roch theorem provides an expression for $\chi(X,V)$ in terms of characteristic classes of $TX$ and $V$ realizing a conjecture of Serre in a letter to Kodaira and Spencer.  
\begin{thm}[Hirzebruch--Riemann--Roch] Let $V$ be a coherent sheaf over a nonsingular variety $X$. Then 
$$\chi(X,V)=\int_X\mathrm{ch}(V)  \mathrm{td}(TX).$$
\end{thm}

We will also need the following  lemma relating the Todd class and the Chern character. This lemma is instrumental in the proof of the  Hirzebruch--Riemann--Roch theorem of Borel and Serre \cite[Lemma 18, p. 128]{BorelSerre}, and is also discussed by Fulton in  \cite[Example 3.2.5, p. 57]{Fulton.Intersection}. 
\begin{lem}[Hirzebruch, {\cite[Theorem 10.1.1, page 92]{Hirzebruch})}] \label{Lem:Hirzebruch}
Let $E$ be a vector bundle of rank $r$. Then 
$$
\mathrm{ch}\Big(\sum_{q=0}^r (-1)^q  {\bigwedge}^{q} E^\vee\Big) \mathrm{td}(E)=c_r(E).
$$
\end{lem}
\begin{proof}   
By the splitting principal, we can always formally factorize the total Chern class of $E$ as  $c(E)=\prod_i (1+ a_i )$, where $a_i$ are the Chern roots of $E$. Then by definition
$$
\mathrm{ch}(E):=\sum_{i=1}^r e^{a_i}, \quad \mathrm{td}(E):=\prod_i \frac{a_i}{(1-e^{-a_i})}.$$
 We have the classical relations (see \cite[Theorem 4.4.3, page 64]{Hirzebruch} or \cite[Remark 3.2.3, p. 54--56]{Fulton.Intersection})
$$c(E^\vee)=\prod_i (1-a_i), \quad c({\bigwedge}^q E)= \prod_{1\leq i_1< \cdots < i_q\leq r} \Big(1+a_{i_1} +\cdots +a_{i_q}\Big)$$
Hence 
$$\quad \mathrm{ch}\Big({ \bigwedge}^q E^\vee\Big)= \sum_{1\leq i_1< \cdots < i_q\leq r} e^{-(a_{i_1} +\cdots +a_{i_q})}$$
Thus by the additive properties of the Chern character and the definition of the Todd class: 
\begin{align}
\mathrm{ch}\Big(\sum_{q=0}^r (-1)^q  {\bigwedge}^{q} E^\vee\Big)&= \sum_{q=0}^r (-1)^q  \mathrm{ch}\Big({\bigwedge}^q E^\vee\Big)=\prod_{i=1}^r (1- e^{-a_{i}})\nonumber\\
&=(a_1\ldots a_r) \prod_{i=1}^r \frac{(1- e^{-a_{i}})}{a_i}=c_r(E) \mathrm{td}^{-1}(E).\nonumber
\end{align}
\end{proof}

\begin{thm}
Let $X$ be  a nonsingular  complete projective variety defined over the complex numbers. Then the Euler characteristic 
$$
\chi(X)=\int c(X).
$$
\end{thm}
\begin{proof}
 For $X$ a nonsingular variety of dimension $d$, we apply  lemma \ref{Lem:Hirzebruch}
 to the tangent bundle  $E=TX$ and we note that  $E^\vee=TX^\vee:=\Omega_X$,  where $\Omega_X$ is the sheaf of differentials of $X$, and by definition, the sheaf of differential $p$-forms is $\bigwedge^q \Omega_X :=\Omega_X^q$. Hence,  we get
$$
\mathrm{ch}\Big(\sum_{q=0}^d (-1)^q  \Omega_X^q \Big) \mathrm{td}(TX)=c_r(TX).
$$
We rewrite the left hand side of the previous equation as follows
$$
\begin{aligned}
\int_X\mathrm{ch}\Big(\sum_{q=0}^d (-1)^q   \Omega_X^q\Big) \mathrm{td}(TX) &=
\sum_{q=0}^d (-1)^q \int_X\mathrm{ch} ( \Omega^q_X) \mathrm{td}(TX)\\
&=\sum_{q=0}^d (-1)^{q}\int_X \chi(X,  \Omega_X^q )\\
&=\sum_{q=0}^d \sum_{p=0}^d (-1)^{p+q}   \rm{dim}\   H^p (X, \Omega_X^q)\\
&=\sum_{k=0}^d (-1)^k \sum_{p+q=k}   \rm{dim}\   H^p (X, \Omega_X^q)\\
&=\sum_{k=0}^d (-1)^k b_k\\
&=\chi(X). 
\end{aligned}
$$
The first equality is a direct consequence of the additive property of the Chern character, the second equality is due to the Hirzebruch--Riemann-Roch theorem applied to $\Omega_X^q$, the third equality follows from the definition 
of the Euler characteristic of a sheaf, and the fifth equality is a direct application  of the Hodge decomposition theorem 
$
\Omega^k=\bigoplus_{p+q=k} \Omega^{p,q} 
$ 
and Dolbeault's theorem, which asserts that the Dolbeault cohomology is isomorphic to the sheaf cohomology of the sheaf of differential forms:
$
H^{p,q}(X)\cong  H^p (X, \Omega_X^q).
$ 
In particular, $h^{p,q}(X)=\rm{dim}\   H^p (X, \Omega_X^q)$ are the Hodge numbers of $X$. 
The last equality is by the definition of the Euler characteristic.  
Hence, since $\int c(X)=\int c(TX)\cap [X]=\int_X c_r (TX)$, we get 
$$
\int c(X)=\chi(X).
$$
 \end{proof}

\section{Basic Notions}
\label{Sec:BasicNotions}

The local ring of a subvariety  $S$ of $X$ is denoted $\mathscr{O}_{X,S}$, its maximal ideal is $\mathscr{M}_{X,S}$  and the quotient field is the residue field  $\kappa(S)=\mathscr{O}_{X,S}/\mathscr{M}_{X,S}$. 
The local ring $\mathscr{O}_{X,S}$ is the stalk of the structure sheaf of $X$ at the generic point  $\eta_S$ of $S$ and $\kappa(S)$ is the function field of $S$.
If $S$ is a divisor, $\mathscr{O}_{X,S}$ is a one dimensional local domain. In case $X$ is nonsingular along $S$, $\mathscr{O}_{X,S}$ is a discrete valuation ring and the order of vanishing  is given by the usual valuation.

\subsection{Fiber types, dual graphs, Kodaira symbols}
\label{Sec:FiberTypes}

\begin{defn}[Algebraic cycle]
An algebraic cycle of a Noetherian scheme $X$ is a finite formal sum  $\sum_i N_i V_i$  of  subvarieties $V_i$ with integer coefficients $N_i$. 
If all the subvarieties  $V_i$ have the same dimension $d$, the cycle is called a $d$-cycle. The free group generated by subvarieties of dimension $d$ is  denoted $Z_d(X)$.  The group of all cycles, denoted $Z(X)=\bigoplus_d Z_d(X)$, is the free group generated by  subvarieties of $X$. 
\end{defn}
\begin{defn}[Degree of a zero-cycle{\cite[Chapter 1, Definition 1.4, p. 13]{Fulton.Intersection}}]
\label{def:degree}
Let  $X$ be a complete scheme.   
The {\em degree} of a zero-cycle $\sum N_i p_i$ of $X$ is $$\mathrm{deg} (\sum_i N_i p_i)=\sum_i N_i [\kappa(p_i): k],$$
 where $[\kappa(p_i): k]$ is the degree of the field extension $\kappa(p_i)\rightarrow  k$. 
\end{defn}

Let $\Theta$ be an algebraic one-cycle with irreducible decomposition  $\Theta=\sum_i m_i \Theta_i$. 
 We denote by $\Theta_i \cdot \Theta_j $ the zero-cycle defined by the  intersection of $\Theta_i$ and $\Theta_j$ for $i\neq j$. 
 \begin{defn}[$n$-points, tree]
A {\em $n$-point }of an algebraic one-cycle $\Theta$ is  a point in $\bigcup_i \Theta_i$, which belongs to exactly $n$ distinct irreducible components  $\Theta_i$. An algebraic one-cycle  $\Theta$ is said to be a {\em tree} if it does not have $n$-points for $n>2$. Two curves intersect transversally if their intersection consists of isolated reduced closed points. 
\end{defn}

Following Kodaira \cite{Kodaira.IIandIII}, we introduce the following definition:

 \begin{defn}[Fiber type]
By the {\em type} of an algebraic one-cycle $\Theta\in Z_1(X)$  with irreducible decomposition $\Theta=\sum_i m_i \Theta_i$, we mean the isomorphism class of each irreducible curve $\Theta_i$, together with the topological structure of the reduced polyhedron $\sum \Theta_{i}$  (that is the collection of zero-cycles $\Theta_i\cdot \Theta_j$ ($i\neq j$)), and  the homology class of $\Theta=\sum_i m_i \Theta_{i}$ in the Chow group $A_1(X)$. 
\end{defn}
\begin{exmp}
For instance,  $\Theta_1\cdot \Theta_2= 2p_1 +3 p_2$ indicates that the two curves $\Theta_1$ and $\Theta_2$ meet at two points $p_1$ and $p_2$ with respective intersection multiplicity $2$ and $3$. 
\end{exmp}

\begin{defn}[Dual graph]
To an algebraic one-cycle $\Theta$ with irreducible decomposition  
$\Theta=\sum_i m_i \Theta_{i}$,  we associate a  weighted graph (called the {\em dual graph} of $\Theta$) such that: 
\begin{itemize}
\item 
 The vertices are the irreducible components of the fiber.
\item
 The weight of a vertex  corresponding to the irreducible component $\Theta_i$ is its multiplicity  $m_i$. 
 When the multiplicity is one, it can be omitted.
\item 
 The vertices corresponding to the irreducible components $\Theta_i$ and $\Theta_j$ ($i\neq j$) are connected by $\hat{\Theta}_{i,j}=\mathrm{deg}( \Theta_i\cdot \Theta_j)$ edges.
\end{itemize}
\end{defn}

\begin{defn}[Kodaira symbols, See   {\cite[Theorem 6.3]{Kodaira.IIandIII}}]\label{def.Kodaira.Fib} Kodaira has introduced the following symbols characterizing the type of one-cycles appearing in the study of minimal elliptic surfaces. 
See Table \ref{Table.KodairaTate} for a visualization of these fibers.
 \begin{enumerate}
\item Type I$_0$:  a smooth curve of genus 1.
\item Type I$_1$: an irreducible nodal  rational curve. 
\item Type II: an irreducible  cuspidal rational curve. 
\item Type I$_2$:  $\Theta=\Theta_1+\Theta_2$ and $\Theta_1\cdot \Theta_2=p_1+p_2$: two smooth rational curves  intersecting transversally at two distinct points $p_1$ and $p_2$. 
The dual graph of I$_2$ is $\widetilde{A}_1$. 
\item Type III: $\Theta=\Theta_1+\Theta_2$ and  $\Theta_1\cdot \Theta_2=2p$: two smooth rational curves  intersecting at a double point. Its dual graph is $\widetilde{A}_1$. 
\item Type IV: $\Theta=\Theta_1+\Theta_2+\Theta_3$ and $\Theta_1\cdot \Theta_2=\Theta_1\cdot \Theta_3=\Theta_2\cdot \Theta_3=p$:  a 3-star composed of smooth rational curves. Its dual graph is $\widetilde{A}_2$. 
\item Type I$_n$ $(n\geq 3)$:  
$\Theta=\Theta_0+\cdots \Theta_n$ with $\Theta_i\cdot \Theta_{i+1} =p_i$ $i=0,\cdots, {n-1}$ and $\Theta_n \cdot \Theta_0= p_n$.
 Its dual graph is the affine Dynkin diagram $\widetilde{A}_{n-1}$. 
\item Type I$^\ast_n$ $(n\geq 0)$: 
$\Theta=\Theta_0+\Theta_1   + 2 \Theta_{2}+\cdots +2\Theta_{n+2}+\Theta_{n+3}+\Theta_{n+4}$, 
with 
$\Theta_{i}\cdot \Theta_{i+1}=p_{i}$ $(i=1,\ldots, n+2)$,  $\Theta_0\cdot \Theta_2=p_0$,  $\Theta_{n+4}\cdot \Theta_{n+2}=p_{n+4}$.
The  dual graph the affine Dynkin diagram $\widetilde{D}_{4+n}$.
\item Type IV$^\ast$: 
$\Theta=\Theta_0+\Theta_1  +2 \Theta_2+2\Theta_3 +3\Theta_4 +2\Theta_5 +\Theta_6$ with 
$\Theta_{i}\cdot \Theta_{i+1}=p_i$ ($i=3,\ldots ,6$),  $\Theta_1\cdot \Theta_3=p_1$, $\Theta_0\cdot \Theta_2=p_0$, $\Theta_2\cdot \Theta_4=p_2$. 
The dual graph is the affine Dynkin diagram $\widetilde{E}_{6}$.

\item Type III$^\ast$: 
$\Theta=\Theta_0+2\Theta_1  +2 \Theta_2+3\Theta_3 +4\Theta_4 +3\Theta_5 +2\Theta_6 + \Theta_7$ with 
$\Theta_{i}\cdot \Theta_{i+1}=p_i$ ($i=3,\ldots ,6$),  $\Theta_1\cdot \Theta_3=p_1$, $\Theta_0\cdot \Theta_1=p_0$, $\Theta_2\cdot \Theta_4=p_2$.
The dual graph is the affine Dynkin diagram $\widetilde{E}_{7}$.

\item Type II$^\ast$: 
$\Theta=2\Theta_1  +3 \Theta_2+4\Theta_3 +6\Theta_4 +5\Theta_5 +4\Theta_6 + 3\Theta_7 + 2\Theta_8 + \Theta_0,$ with 
$\Theta_{i}\cdot \Theta_{i+1}=p_i$ ($i=3,\ldots ,7$),  $\Theta_1\cdot \Theta_3=p_1$, $\Theta_8\cdot \Theta_0=p_8$, and $\Theta_2\cdot \Theta_4=p_2$. 
The dual graph the affine Dynkin diagram $\widetilde{E}_{8}$.
\end{enumerate}

\end{defn}
\subsection{Elliptic fibrations, generic versus geometric fibers}
\label{Sec:GenericFiber}
\begin{defn}[Elliptic fibrations]
A surjective proper morphism $\varphi:Y\longrightarrow B$  between two algebraic varieties $Y$ and $B$ is called an elliptic fibration if the generic fiber of $\varphi$ is a smooth projective curve of genus one and $\varphi$ has a rational section. 
When $B$ is a curve, $Y$ is called an elliptic surface. When $B$ is a surface, $Y$ is said to be an elliptic threefold. In general,  if $B$ has dimension $n-1$, $Y$ is called an elliptic $n$-fold. 
\end{defn}
 The locus of singular fibers of $\varphi$ is called the discriminant locus of $\varphi$ and is denoted $\Delta(\varphi)$ or simply $\Delta$ when the context is clear. 
If the base $B$ is smooth, the discriminant locus is a divisor \cite{Dolgavcev.Purity}. 
The singular fibers of a minimal elliptic surface have been classified by Kodaira and N\'eron. 
The dual graphs of these geometric  fibers  are affine Dynkin diagrams. 
We denote these singular fibers by their Kodaira symbols as described in Definition \ref{def.Kodaira.Fib} and presented in Table \ref{Table.KodairaTate}.

The language of schemes streamlines many notions in the study of fibrations. We review some basic definitions.

 \begin{defn}[Fiber over a point]
 Let $\varphi: Y\longrightarrow B$ be a morphism of schemes. For any $p\in B$, the fiber over $p$ is denoted $Y_p$ and defined using a fibral product\footnote{Given three sets ($A_1$, $A_2$, and $S$) and two maps $\varphi_1:A_1\rightarrow B$ and  $\varphi_2:A_2\rightarrow B$, we define the fibral product $A_1\times_S A_2$ as the subset of $A_1\times A_2$ composed of couples $(a_1,a_2)$ such that 
$\varphi_1 (a_1)=\varphi_2(a_2)$.} as 
 $$Y_p=Y\times_B \rm{Spec}\ \kappa(p).$$ 
 \end{defn}
The first projection $Y_p\longrightarrow Y$ induces an homeomorphism from $Y_p$ onto $f^{-1}(p)$ \cite[\S 3.1 Proposition 1.16]{QLiu.AGAC} .
The second projection gives $Y_p$ the structure of a scheme over  the residue field $\kappa(p)$.

If $p$ is not a closed point\footnote{For example, if $p$ is the generic point of a subvariety of $B$.}, the residue field $\kappa(p)$ is not necessarily algebraically closed. 
Certain components of $Y_p$ could be $\kappa(p)$-irreducible  (i.e. irreducible when defined over $\kappa(p)$) while they become  reducible after an appropriate field extension. 
 An irreducible  scheme over a field $k$ is said to be {\em geometrically irreducible} when it stays  irreducible after any  field extension. 
The most refined description of the fiber $Y_p$ is always the one corresponding to the algebraic closure $\overline{\kappa(p)}$ of $\kappa(p)$. This motivates the following definition. 
\begin{defn}
The geometric fiber over $p$ is the fiber $Y_p\times_{\kappa(p)} \overline{\kappa(p)}$, the fiber $Y_p$ after the   base change induced by the field extension $\kappa(p)\to \overline{\kappa(p)}$ to the algebraic closure of $\kappa(p)$. 
\end{defn}
By construction, a geometric fiber is always composed of geometrically irreducible components. 

\begin{defn}
We say that the type of a fiber $Y_p$ is {\em geometric} if it does not change after a field extension.
\end{defn}
\begin{rem}
To emphasize the difference between the fiber $Y_p$ and its geometric fiber, we will refer to the fiber $Y_p$ (defined with respect to the residue field $\kappa(p)$) as the {\em arithmetic fiber}. 
\end{rem}

For an elliptic $n$-fold, the Kodaira fibers are also the  {\em geometric generic fibers} of the irreducible components of the reduced discriminant locus. 
 While the dual graph of a Kodaira fiber is an affine Dynkin diagram of type $\widetilde{A}_k$, $\widetilde{D}_{4+k}$, $\widetilde{E}_6$, $\widetilde{E}_7$, or $\widetilde{E}_8$, 
 the dual graph of the generic (arithmetic) fiber itself can also be a twisted Dynkin diagram of type $\widetilde{B}_{3+k}^t$, $\widetilde{C}_{2+k}^t$, $\widetilde{G}_2^t$, or $\widetilde{F}_4^t$. 
 This is reviewed in  Tables  \ref{Table:Affine}   and \ref{Table:DualGraph}. These dual graphs are not geometric in the sense that after an appropriate base change they become $\widetilde{D}_{4+n}$, $\widetilde{A}_{2+2k}$ or $\widetilde{A}_{1+2k}$, and $\widetilde{E}_6$ respectively.
 The Kodaira fibers of the following type never need a field extension:
$$\text{  I$_1$, II, III, III$^*$, and  II$^*.$}$$

  The remaining Kodaira fibers 
   (IV, I$_{n>1}$, I$_n^*$, and IV$^*$) 
 can come from fibers Y$_p$ whose types are not geometric and require at least a field extension of degree $2$ to describe a fiber with a  geometric type. 
 When the fiber $Y_p$ has a geometric type, the type of the fiber is said to be {\em split}.
 Otherwise, the type of $Y_p$ is said to be non-split. 
  When that is the case we mark the fiber with an  ``ns'' superscript: IV$^{\text{ns}}$, I$_n^{\text{ns}}$, I$_n^{*\text{ns}}$, $(n\geq 2)$ and IV$^{*\text{ns}}$. 
  When a field extension is not needed, the fibers are marked with an ``s'' superscript (``split''): IV$^{\text{s}}$, I$_n^{\text{s}}$, I$_n^{*\text{s}}$, $(n\geq 2)$ and IV$^{*\text{s}}$. 
  The fiber of type I$_0^*$ can be split, semi-split, or non-split if the Kodaira types require no field extension, a quadratic extension, or a cubic extension. 
  The corresponding dual graphs are respectively $\widetilde{D}_4$, $\widetilde{B}_3^t$, and $\widetilde{G}_2^t$.    

\subsection{Weierstrass models and Deligne's formulaire }
\label{sec:Wmodel}
 We follow the notation of  Deligne  \cite{Formulaire}. 
Let  $\mathscr{L}$ be a line bundle over  a quasi-projective variety  $B$.  We define the following projective bundle (of lines):
\begin{equation}
\pi: X_0=\mathbb{P}_B[\mathscr{O}_B\oplus \mathscr{L}^{\otimes 2}\oplus \mathscr{L}^{\otimes 3}]\longrightarrow B.
\end{equation} 
The relative projective coordinates of $X_0$ over $B$ are denoted $[z:x:y]$,  where $z$, $x$, and $y$ are defined respectively  by the natural injection of 
 $\mathscr{O}_B$,   $\mathscr{L}^{\otimes 2}$, and $\mathscr{L}^{\otimes 3}$ into $\mathscr{O}_B\oplus \mathscr{L}^{\otimes 2}\oplus \mathscr{L}^{\otimes 3}$. Hence, 
  $z$ is a section of $\mathscr{O}_{X_0}(1)$, $x$ is a section of $\mathscr{O}_{X_0}(1)\otimes \pi^\ast \mathscr{L}^{\otimes 2}$, and
$y$ is a section of  $\mathscr{O}_{X_0}(1)\otimes \pi^\ast \mathscr{L}^{\otimes 3}$.

\begin{defn}
 A  Weierstrass model is an elliptic fibration $\varphi: Y\to B$  cut out by the zero locus of  a section of the  
line bundle $\mathscr{O}(3)\otimes \pi^\ast \mathscr{L}^{\otimes 6}$ in $X_0$. 
\end{defn}
The most general Weierstrass equation is written in the notation of Tate as \cite{Formulaire} $F=0$ with 
\begin{equation}
F=y^2z+ a_1 xy z + a_3  yz^2 -(x^3+ a_2 x^2 z + a_4 x z^2 + a_6 z^3),
\end{equation} 
where $a_i$ is a section of $\pi^\ast \mathscr{L}^{\otimes i}$. 
The line bundle $\mathscr{L}$ is called the {\em fundamental line bundle} of the Weierstrass model $\varphi:Y\to B$  and can be defined directly from the elliptic fibration $Y$ as 
$\mathscr{L}=R^1 \varphi_\ast \mathscr{O}_Y$. 
Following Tate and Deligne, we introduce the following quantities \cite{Formulaire} 
\begin{align}
\begin{cases}
b_2 &= a_1^2 + 4 a_2\\
b_4 &= a_1 a_3 + 2 a_4\\
b_6 &= a_3^2 + 4 a_6\\
b_8 &= a_1^2 a_6 - a_1 a_3 a_4 + 4 a_2 a_6 + a_2 a_3^2 - a_4^2\\
c_4 &= b_2^2 - 24 b_4\\
c_6 &= -b_2^3 + 36 b_2 b_4 - 216 b_6\\
\Delta &= -b_2^2 b_8 - 8 b_4^3 - 27 b_6^2 + 9 b_2 b_4 b_6\\
j& = {c_4^3}/{\Delta}
\end{cases}
\end{align}
These quantities  satisfy the following two relations
\begin{align}
1728 \Delta=c_4^3-c_6^2, \quad 4b_8 = b_2 b_6 - b_4^2.
\end{align}
The  $b_i$ ($i=2,3,4,6)$ and $c_i$   ($i=4,6$) are  sections of $\pi^\ast \mathscr{L}^{\otimes i}$. 
The discriminant $\Delta$ is a section of $\pi^\ast \mathscr{L}^{\otimes 12}$. Geometrically, the discriminant $\Delta$ is the locus of points over which the elliptic fiber is singular.  
The $j$-invariant characterizes a smooth elliptic curve up to isomorphism. 
If we  complete the square in $y$ in the Weierstrass equation,  the equation becomes 
\begin{equation}
zy^2 =x^3 +\tfrac{1}{4}b_2 x^2 z + \tfrac{1}{2} b_4 x z^2+ \tfrac{1}{4} b_6 z^3.
\end{equation}
In addition, if we complete the cube in $x$ gives the short form of the Weierstrass equation, the equation becomes
\begin{equation}
zy^2 =x^3 -\tfrac{1}{48} c_4 x z^2 -\tfrac{1}{864} c_6 z^3.
\end{equation}

\subsection{Tate's algorithm}
 Let $R$ be a complete discrete valuation ring with valuation $v$, uniformizing parameter $s$, and perfect residue field $\kappa=R/(s)$. 
  We are interested in the case where $\kappa$ has characteristic zero. 
 We recall that a discrete valuation ring has only three ideals, the zero ideal, the ring itself, and the principal ideal $s R$. 
  We take the convention in which the ring itself is not a prime ideal.
 It follows that the scheme $\mathrm{Spec}(R)$ has only two points:  the generic point (defined by the zero ideal) and the closed point (defined by the principal ideal  $s R$).
 
 Let $E/R$ be an elliptic curve over $R$ with  Weierstrass equation
  $$
 y^2+a_1 xy + a_3 y=x^3+a_2 x^2+ a_4 x + a_6, \quad a_i \in R.
 $$
  The generic fiber is a regular elliptic curve.
 After a resolution of singularities, we have a regular model $\mathscr{E}$ over $R$ and the {\em special fiber} is  the fiber over the closed point $\mathrm{Spec}\ R/ (s)$.
  Tate's algorithm determines the type of the geometric fiber over the closed point of $\mathrm{Spec} (R)$ by manipulating the valuations of the coefficients and the discriminant, and the arithmetic properties of some auxiliary polynomials. The type of the {\em geometric fiber} is denoted  by its  Kodaira's symbol (see Definition \ref{def.Kodaira.Fib}). 
The special fiber becomes geometric after  a quadratic or a cubic field extension $\kappa'/\kappa$. Keeping track of the field extension used gives a classification of  the special fiber as a $\kappa$-scheme---this is what we call the arithmetic fiber. 
The information on the required field extension needed to have geometrically irreducible components is already carefully encoded in Tate's original algorithm, as it is needed to compute 
  the local index (denoted by $c$ in Tate's notation). 
  In the language of N\'eron's model, the local index $c$  is the order of the component group;   
 geometrically, the local index is the number of reduced components of the special fiber  defined over $\kappa$. 
.  Following Tate, we use the convenient notation $$a_{i,j}= a_i s^{-j}.$$

Tate's algorithm consists of the following eleven steps (see \cite{Tate},  \cite[\S IV.9]{MR1312368}, \cite{Papadopoulos}, \cite{MR3085154}, \cite{Esole.Elliptic}). 
 For Step 7, we use the more refined description of Papadopoulos \cite[Part III, page 134]{Papadopoulos} who also gives in  \cite[\S 1, page 122]{Papadopoulos} an exhaustive list of errata of Tate's original paper \cite{Tate}. 
 Tate's algorithm is discussed in F-theory in \cite{Bershadsky:1996nh,Katz:2011qp}. 
 Subtleties in Step 6 and the distinction between two G$_2$-models depending on $[\kappa':\kappa]$  are explained in \cite{G2}.  
 We follow the presentation of \cite{Esole.Elliptic}:
\begin{enumerate}[label=Step \arabic*.]
\item $v(\Delta)=0 \implies$ I$_0$.
\item If $v(\Delta)\geq 1$, change coordinates so that $v(a_3)\geq 1$, $v(a_4)\geq 1$, and $v(a_6)\geq 1$.\\
\noindent If $v(b_2)=0$, the  type is I$_{v(\Delta)}$.  To have  a fiber with geometric irreducible components,  it is enough to work in the splitting field $\kappa'$ of the following polynomial of $\kappa[T]$:
$$
T^2+a_1 T-a_2.
$$
The discriminant of this quadric is $b_2$. If $b_2$ is a square in $\kappa$, then $\kappa'=\kappa$, otherwise $\kappa'\neq \kappa$:\\
(a)   $\kappa'=\kappa \implies$ I$_n^\text{s}$
\quad  (b) $\kappa'\neq \kappa\implies$ I$_n^{\text{ns}}$

\item $v(b_2)\geq 1$, $v(a_3)\geq 1$, $v(a_4)\geq 1$,  and $v(a_6)= 1$ $\implies$  II. 

\item $v(b_2)\geq 1$, $v(a_3)\geq 1$, $v(a_4)= 1$,  and $v(a_6)\geq 2$ $ \implies$ III.

\item $v(b_2)\geq 1$, $v(a_3)\geq 1$, $v(a_4)\geq 2$,  $v(a_6)\geq 2$, and $v(b_6)=2$ $ \implies$ IV.   \\ 
The  fiber has geometric irreducible components over the splitting field $\kappa'$ of the polynomial 
$$
T^2+a_{3,1} T-a_{6,2}
$$
Its discriminant is $b_{6,2}$. If $b_{6,2}$ is a square in $\kappa$, then $\kappa'=\kappa$ otherwise $\kappa'\neq \kappa$. \\ 
(a)   $\kappa'=\kappa \implies$ IV$^\text{s}$
\quad  (b) $\kappa'\neq \kappa\implies$ IV$^{\text{ns}}$
\item 

$v(b_2)\geq 1$, $v(a_3)\geq 1$, $v(a_4)\geq 2$, $v(a_6)\geq 3$, $v(b_6)\geq 3$, $v(b_8)\geq 3$. Then  make a change of coordinates such that $v(a_1)\geq 1$, 
$v(a_2)\geq 1$,  $v(a_3)\geq 2$,  $v(a_4)\geq 2$, and $v(a_6)\geq 3$. 
Let 
$$
P(T)= T^3+ a_{2,1} T^2 + a_{4,2} T + a_{6,3}
$$
If $P(T)$ is a separable polynomial in $\kappa$, that is if $P(T)$ has three distinct roots in a field extension of $\kappa$,  then the type is I$_0^\ast$. 
The geometric fiber is defined over the splitting field $\kappa'$ of $P(T)$ in $\kappa$. The type of the special fiber before to go to the splitting field depends on the degree of the field extension $\kappa'\to \kappa$:
\begin{itemize}
 \item $[\kappa':\kappa]=6\   \text{or}\    3\implies$ I$_0^{\ast \text{ns}}$ with dual graph $\widetilde{\text{G}}_2^t$. 
\item   $[\kappa':\kappa]=2\quad \quad\implies$ I$_0^{\ast \text{ss}}$ with dual graph $\widetilde{\text{B}}_3^t$. 
\item    $[\kappa':\kappa]=1\quad \quad\implies$ I$_0^{\ast \text{s}}$ with dual graph $\widetilde{\text{D}}_4$. 
\end{itemize}
 where ``ns'', ``ss'', and ``s'' stand respectively for ``non-split'', ``semi-split'', and ``split''. 
In the notation of Liu, these fibers are respectively  I$^\ast_{0,3}$,  I$^\ast_{0,2}$, and I$^\ast_{0}$.
{ The Galois group is the symmetric group $S_3$, the cyclic group $\mathbb{Z}/3\mathbb{Z}$, the cyclic group  $\mathbb{Z}/2\mathbb{Z}$ or the identity when 
the degree is respectively $6$, $3$, $2$, and $1$. }
\item 
 If $P(T)$ has a  double root, then the type is I$^\ast_n$ with $n\geq 1$. 
Make a change of coordinates such that the double root is at the origin. Then 
$v(a_1)\geq 1$, \quad
$v(a_2)= 1$,  $v(a_3)\geq 2$,\quad  $v(a_4)\geq 3$, ,\quad$v(a_6)\geq 4$, and $v(\Delta)=n+6$ $(n\geq 1)$. 
We now assume that, except for their valuations, the Weierstrass coefficients are generic. We then  
distinguish between even and odd values of $n$. 
\begin{enumerate}
\item If $n=2\ell-3$ ($\ell \geq 2$), then $v(a_1)\geq 1$, $v(a_2)= 1$,  $v(a_3)\geq \ell$,\quad  $v(a_4)\geq \ell+1$,\quad$v(a_6)\geq 2\ell$,  $v(b_6)=2\ell$, $v(b_8)=2\ell+1$, and 
$$T^2+a_{3,\ell} T -a_{6,2\ell}$$ has two distinct roots in its splitting field $\kappa'$. 
If the two roots are rational ($[\kappa':\kappa]=1$) then we have I$_{2\ell-3}^{*s}$ with dual graph $\widetilde{\text{D}}_{2\ell+1}$, otherwise 
 ($[\kappa':\kappa]=2$) we have the fiber type 
I$_{2\ell-3}^{* ns}$ with dual graph $\widetilde{\text{B}}^t_{2\ell}$. 
\item 
If $n=2\ell-2$ ($\ell\geq 2$) then, $v(a_1)\geq 1$, \quad
$v(a_2)= 1$,  $v(a_3)\geq \ell+1$,\quad  $v(a_4)\geq \ell+1$,\quad$v(a_6)\geq 2\ell+1$, and $v(b_8)=2\ell+2$. 
The polynomial 
$$a_{2,1} T^2+a_{4,\ell+1} T -a_{6,2\ell+1}$$ has two distinct roots in its splitting field. 
If the two roots are rational then we have I$_{2\ell-2}^{*s}$ with dual graph $\widetilde{\text{D}}_{2\ell+2}$, otherwise I$_{2\ell-2}^{* ns}$ with dual graph $\widetilde{\text{B}}^t_{2\ell+1}$. 
\end{enumerate}

\item If $P(T)$ has a triple root, change coordinates such that the triple root is zero. Then  
$v(a_1)\geq 1$, 
$v(a_2)\geq 2$,  $v(a_3)\geq 2$,  $v(a_4)\geq 3$, $v(a_6)\geq 4$.  

Let $$Q(T)=T^2+a_{3,2} T -a_{6,4}$$

 If $Q$ has two distinct roots ($v(b_6)=4$ or equivalently  $v(\Delta)=8$) the type is IV$^\ast$. 
 
 The split type depends on the rationality of the roots. If $b_{6,4}$ is a perfect square modulo $s$, the fiber is IV$^{\ast \text{s}}$ with dual graph $\widetilde{\text{E}}_6$, otherwise the fiber is IV$^{\ast \text{ns}}$ with dual graph 
 $\widetilde{\text{F}}_4^t$. 
   The split form can be enforced with $v(a_6)\geq 5$ and $v(a_3)=2$. 
    
\item If $Q$ has a double root, we change coordinates so that the double root is at the origin. 
Then:\\
$ v(a_1)\geq 1$,\quad  $v(a_2)\geq 2$,\quad $v(a_3)\geq 3$,\quad $v(a_4)=3$,\quad $v(a_6)\geq 5\implies $  type  III$^\ast$.
\item $v(a_1)\geq 1$,\quad  $v(a_2)\geq 2$,\quad $v(a_3)\geq 3$,\quad $v(a_4)\geq 4$,\quad $v(a_6)= 5\implies$  type II$^\ast$.
\item Else $v(a_i)\geq i$ and the equation is not minimal. Divide all the $a_i$ by $s^i$ and start again with the new equation. 
\end{enumerate}


\end{document}